\documentclass[11pt,a4paper]{article}

\usepackage[margin=25mm]{geometry}
\usepackage[T1]{fontenc}
\usepackage[utf8]{inputenc}
\usepackage{lmodern}
\usepackage{microtype}
\usepackage{graphicx}
\usepackage{multirow}
\usepackage{amsmath,amssymb,amsfonts,amsthm}
\usepackage{mathtools}
\usepackage[title]{appendix}
\usepackage{xcolor}
\usepackage{textcomp}
\usepackage{manyfoot}
\usepackage{booktabs}
\usepackage{algorithm}
\usepackage{algorithmicx}
\usepackage{algpseudocode}
\usepackage[shortlabels]{enumitem}
\setlist[enumerate]{leftmargin=.5in}
\setlist[itemize]{leftmargin=.5in}
\allowdisplaybreaks[4]

\usepackage[numbers,sort&compress]{natbib}
\usepackage[hidelinks]{hyperref}
\hypersetup{
  pdftitle={Picard Approximation and Martingale Integrand Estimation for Path-Dependent FBSDEs},
  pdfauthor={Jiuk Jang and Hyungbin Park}
}

\theoremstyle{plain}
\newtheorem{theorem}{Theorem}[section]
\newtheorem{lemma}[theorem]{Lemma}

\newtheorem{proposition}[theorem]{Proposition}

\newtheorem{prop}[theorem]{Proposition}
\theoremstyle{remark}

\newtheorem{remark}[theorem]{Remark}

\theoremstyle{definition}

\newtheorem{assume}[theorem]{Assumption}
\numberwithin{equation}{section}
\numberwithin{algorithm}{section}
\newsavebox{\fitteddisplaybox}
\newcommand{\fitdisplay}[1]{%
  \begingroup
  \sbox{\fitteddisplaybox}{$\displaystyle #1$}%
  \ifdim\wd\fitteddisplaybox>\dimexpr\linewidth-3em\relax
    \resizebox{\dimexpr\linewidth-3em\relax}{!}{\usebox{\fitteddisplaybox}}%
  \else
    \usebox{\fitteddisplaybox}%
  \fi
  \endgroup
}

\title{Picard Approximation and Martingale Integrand Estimation for Path-Dependent FBSDEs}
\author{%
  Jiuk Jang\thanks{Department of Mathematical Sciences, Seoul National University,
  Seoul 08826, South Korea. Email: \href{mailto:jiukjang@snu.ac.kr}{jiukjang@snu.ac.kr}.}
  \and
  Hyungbin Park\thanks{Department of Mathematical Sciences
  and Research Institute of Mathematics, Seoul National University,
  Seoul 08826, South Korea. Emails: \href{mailto:hyungbin@snu.ac.kr}{hyungbin@snu.ac.kr},
  \href{mailto:hyungbin2015@gmail.com}{hyungbin2015@gmail.com}.}
}
\date{}

\begin{document}
\maketitle

\begin{abstract}
This paper develops a Picard iteration method for path-dependent forward-backward stochastic differential equations (FBSDEs). We construct a Picard iteration scheme, prove its convergence to the exact solution, and derive an explicit convergence rate. Although the underlying problem is conceptually straightforward, establishing rigorous convergence results in the path-dependent setting is technically demanding.
Path dependence introduces a logarithmic correction: the time-discretization error bound is of order $|\pi|\log(1/|\pi|)$, compared with the order $|\pi|$ obtained in the Markovian setting, where $|\pi|$ denotes the mesh size of the time partition.
We also introduce an estimator for the martingale integrand of the FBSDE that is tailored to path dependence and establish, within a Monte Carlo framework, a concentration inequality that quantifies its statistical error. The applicability of the proposed method is demonstrated in two non-Markovian derivative-pricing settings.
The primary objective of this paper is to provide a theoretical convergence analysis of the Picard iteration method in the path-dependent setting and illustrate its application to option pricing, rather than to develop a faster or more computationally efficient algorithm.
\end{abstract}



\section{Introduction}

\subsection{Overview}

Non-Markovianity is a fundamental concept across various disciplines, including mathematics, physics, engineering, and finance. It means that the future state of a system depends not only on its present state but also on its entire past history. 
Analyzing the non-Markovian property is essential for understanding how certain outcomes emerge from their historical context.
While the memoryless nature of Markovian processes leads to standard forward-backward stochastic differential equations (FBSDEs) and partial differential equations (PDEs),
non-Markovian problems give rise to path-dependent FBSDEs and PDEs.
This added complexity makes it more difficult to derive generalizable insights or to design efficient numerical schemes for non-Markovian problems.

Markovian FBSDEs and PDEs have been extensively studied.
A classical example of a Markovian SDE is the geometric Brownian motion
$$X(s)=x+\mu\int_0^s X(u)\,du+\sigma \int_0^s X(u)\,dW(u)\,,\;s\ge 0$$
for  $x,\mu,\sigma\in\mathbb{R}$ and a Brownian motion $W$.
Since the drift and diffusion depend only on the current state $X(u)$, the solution $X$ is Markovian.
Quantities of interest include expectations of the form 
$\mathbb{E}[g(X(T))+\int_0^Tf(s,X(s))\,ds]$ for suitable functions $f$ and $g$. This expectation is associated with  the BSDE
\begin{equation} 
	\begin{aligned}
		Y(s)=g(X(T))+\int_{s}^{T} f(u,X(u))\,du-\int_{s}^{T} Z(u)\, dW(u) \,,\;   0\le s\le T
	\end{aligned}
\end{equation}
and the PDE
\begin{equation}
	\begin{aligned}
		&\partial_{s}v+\frac{1}{2}\sigma^2x^2\partial_{xx} v +\mu x \partial_{x} u +f(s,x)=0 \,,\;0\le s<T\\
		&v(T,x)=g(x)\,.
	\end{aligned}
\end{equation} 
Extensive numerical methodologies have been studied for this type of BSDEs and PDEs.

In this paper we investigate numerical methods for path-dependent FBSDEs and PDEs, where the coefficients may depend on the entire past trajectory. For instance, consider the path-dependent SDE
$$X(s)=x+\mu\int_0^s \int_0^uX(r)\,dr\,du+\sigma \int_0^s \max_{0\le r\le u} X(r)\,dW(u)\,,\;0\le s\le T$$
Because the drift and diffusion depend on the history  $(X(r))_{0\le r\le u}$, the solution $X$ is non-Markovian.
As an example, consider the expectation 
$\mathbb{E}[g(\min_{0\le r\le T}X(r))+\int_0^Tf(s,\int_0^se^{-ar}X(r)\,dr)\,ds]$.
This expectation is associated with the path-dependent BSDE
\begin{equation} 
	\begin{aligned}
		Y(s)=g(\min_{0\le r\le T}X(r))+\int_{s}^{T} f(u,\int_0^ue^{-ar}X(r)\,dr)\,du-\int_{s}^{T} Z(u)\, dW(u) \,,\;   0\le s\le T
	\end{aligned}
\end{equation}
and the path-dependent PDE
\begin{equation} 
	\begin{aligned}
		&\partial_{s}v+\frac{1}{2}\sigma^2(\max_{0\le u \le s}\omega(u))^2\partial_{\omega\omega} v +\mu \Big(\int_0^s \omega(u)\,du \Big)\partial_{\omega} v\\
		&\qquad +f\Big(s,\int_0^s e^{-au}\omega(u)\,du\Big)=0 \,,\;0\le s<T\\
		&v(T,\omega)=g(\min_{0\le u\le T}\omega(u))
	\end{aligned}
\end{equation} 
where $\partial_\omega v$ and $\partial_{\omega\omega} v$ denote the vertical first and second functional derivatives. A more general and rigorous formulation of this path-dependent setting is presented in Section~\ref{sec:prelim}.

This paper makes two theoretical contributions and develops two applications.
First, we establish the validity of Picard iteration for solving path-dependent FBSDEs. Although forward Picard schemes are well established in the Markovian setting, a convergence-rate analysis for such schemes under general non-anticipative path dependence has not been systematically developed. To address this gap, we develop a Picard iteration scheme that approximates the solution through a sequence of iterates, each of which retains the dependence on the entire past trajectory of the underlying process.
A detailed error analysis yields an explicit convergence rate (Theorem~\ref{thm:main}). In particular, path dependence introduces a logarithmic correction: the convergence error is of order $|\pi|\log(1/|\pi|)$, rather than the order $|\pi|$ obtained in the Markovian setting by Bender and Denk~\cite{forwardbsde}, where $|\pi|$ denotes the mesh size of the time discretization. 
We further show that this rate is sharp and therefore cannot, in general, be improved. Although the underlying formulation is conceptually straightforward, establishing these results in the path-dependent framework requires substantial technical work.

Second, this paper provides an estimator for the martingale integrand of the FBSDE that is tailored to path dependence and establish, within a Monte Carlo framework, a concentration inequality that quantifies its statistical error.
In numerical algorithms for FBSDEs, estimating the martingale integrand is not merely a subsidiary step but one of the primary challenges and a core component of the entire method.
Several approaches have been developed for estimating $Z$, including least-squares regression methods (\cite{forwardbsde, secondbsdeparabpde}), probabilistic operator-based techniques (\cite{probmethpde}), and Malliavin calculus-based techniques (\cite{bouchard2004discrete, gobet2016approximation, gobet2017adaptive}).
The performance of these methods strongly depends on the Markovian structure of the problem and the choice of basis functions in the linear approximation.
Therefore, they are often inapplicable or inefficient for path-dependent FBSDEs.

Estimating the martingale integrand for non-Markovian functionals is significantly more difficult than in the Markovian case.
Our proposed estimator is based on a weak approximation of the martingale representation
(\cite{weakmtg}).
A key advantage of our estimator is its robustness in path-dependent settings, without the need to select basis functions, often a major source of error in existing methods.   Furthermore, we establish a concentration inequality
that quantifies the statistical error of the estimator in a Monte Carlo framework (Theorem \ref{thm_con_rate}).
This theoretical guarantee, together with its practical implementability, makes the estimator a versatile tool for a wide range of path-dependent problems.

We also present two applications of the proposed framework.
First, we propose a supervised learning framework based on neural networks for solving path-dependent partial differential equations (PPDEs). 
To construct the training dataset, we apply our Picard iteration method to generate approximate solutions that capture the full path-dependence of the problem. Building on this foundation, we establish an approximation theorem for neural network learning of PPDEs. 
In particular, we prove that  every functional satisfying suitable regularity conditions
 can be approximated to arbitrary accuracy by the chosen network class (Theorem \ref{thm:NN}). This result ensures that our method has a solid theoretical guarantee of universal approximation in the path-dependent setting.

Second, as a financial application, we develop a systematic numerical methodology for computing derivative prices and the associated martingale integrands in non-Markovian models. By formulating pricing problems as path-dependent FBSDEs, the proposed Picard iteration scheme accommodates path dependence in both derivative payoffs and market coefficients. We first consider a floating-strike lookback call option and a straddle option in a Black--Scholes model with a dividend rate determined by the recent history of stock prices. We then apply the method to a duration-dependent CEV model in which the CEV volatility coefficient depends on a weighted average of stock prices over a prescribed past duration. For this model, we compute European call and double-barrier digital option values, together with estimates of the martingale integrand, and examine how the duration parameter affects prices and local hedging responses.  These examples demonstrate that the theoretical framework provides a practical computational methodology for derivative pricing when path dependence enters through the payoff, the dividend specification, or the volatility dynamics.

\subsection{Related literature}

The theory of backward stochastic differential equations was initiated by Pardoux and Peng~\cite{initbsde}; Zhang~\cite{BSDE} gives a comprehensive account, and Chessari et al.~\cite{chessari2023numerical} provide a recent survey of numerical methods. In the Markovian setting, approximation schemes have been developed through time discretization, conditional-expectation regression, and forward iteration; representative works include \cite{douglas,bouchard2004discrete,zhangbsde,gobetlemorwarin2005,milstein,forwardbsde}. Picard fixed-point representations also form the basis of multilevel Picard algorithms for high-dimensional semilinear PDEs \cite{multilevelpicard2021}. Most implementable versions of these methods represent the backward variables through a finite-dimensional Markov state. The problem studied here is different because the coefficients, the solution, and every Picard iterate are non-anticipative functionals of the entire trajectory of the forward process.

The analytical study of path-dependent equations is built on functional It\^{o} calculus and stochastic representations. Cont and Fourni{\'e}~\cite{contfournie2013} developed a functional It\^{o} calculus and a stochastic integral representation for martingales. Peng and Wang~\cite{pathbsdebm} established a nonlinear Feynman--Kac formula linking non-Markovian BSDEs and semilinear PPDEs, while Wang et al.~\cite{pathkac} obtained a path-dependent Feynman--Kac formula for forward-backward stochastic Volterra integral equations. Hu et al.~\cite{hurgentouzi2023} studied multidimensional path-dependent FBSDEs through decoupling random fields and established solvability and stability results. Viscosity solution theories for semilinear and fully nonlinear PPDEs were developed in \cite{ekrenkellertouzizhang2014,ekrentouzizhang2016a,ekrentouzizhang2016b}. These works provide analytical notions of solution and well-posedness results. To the best of our knowledge, however, they do not analyze the forward Picard discretization considered in this paper or provide an explicit iteration error bound for path-dependent FBSDEs.

Numerical approximation of PPDEs and non-Markovian control problems has followed several directions. For a special memory structure, Zhang and Xiong~\cite{zhangxiong2019delay} established strong convergence of a discretization for forward-backward stochastic equations with a delay and an anticipated term. Extending the monotone-scheme framework of Barles and Souganidis~\cite{monoschpde}, Zhang and Zhuo~\cite{pathmonosch} proposed monotone schemes for fully nonlinear parabolic PPDEs, Ren and Tan~\cite{rentan2017} established a general convergence theorem for such schemes, and Tan~\cite{tan2014pathcontrol} developed a discrete-time probabilistic approximation for path-dependent stochastic control. Neural approaches include the path-dependent deep Galerkin method \cite{deepgalerkin}, recurrent and signature-based methods \cite{sabate2020solving,deepsignature,bayraktar2024deep,sunbao2025}, and a neural rough differential equation model \cite{fangniwu2026}. These methods approximate the PPDE solution or the associated FBSDE by optimizing a neural-network loss. By contrast, our primary objective is to establish the mathematical convergence of the path-dependent Picard map before any neural approximation is introduced. The neural-network component of this paper subsequently uses Picard-generated labels in a supervised regression framework.

The approximation of the martingale integrand $Z$ requires separate attention. Standard methods rely on regression of conditional expectations \cite{gobetlemorwarin2005,forwardbsde}, probabilistic operators \cite{probmethpde}, or Malliavin-weight representations \cite{bouchard2004discrete,gobet2016approximation,gobet2017adaptive}. In a genuinely path-dependent problem, regression-based implementations generally require either a finite-dimensional state augmentation or a choice of path features and basis functions. Our estimator instead builds on the weak approximation of martingale representations of Cont and Lu~\cite{weakmtg}. It does not require a regression basis, and we derive a non-asymptotic concentration bound for its Monte Carlo error.

 The remainder of the paper is organized as follows. Section~\ref{sec:prelim} introduces the mathematical preliminaries and assumptions underlying the construction and analysis of the numerical scheme for FBSDEs. Section~\ref{sec:Picard} presents the Picard iteration scheme and establishes its convergence rate. Convergence is first proved under smoothness assumptions on the coefficients and then extended to a more general setting through suitable coefficient approximations. Section~\ref{sec:PPDE} develops a supervised learning framework for solving path-dependent partial differential equations. An error analysis of the Monte Carlo estimator for the martingale integrand is provided in Section~\ref{sec:error}, and an algorithm for computing the required conditional expectations is presented in Section~\ref{sec:condi}. Section~\ref{sec:NS} reports numerical experiments for benchmark problems, including a floating-strike lookback option. Section~\ref{sec:deri} applies the proposed framework to derivative pricing under a Black--Scholes model with path-dependent dividends. Section~\ref{sec:cev_application} introduces a duration-dependent CEV model and applies the Picard framework to European call and double-barrier digital options, with particular emphasis on how the duration parameter affects option values and the martingale integrand. The final section concludes the paper.

\section{Preliminaries}\label{sec:prelim}

In this section, we present the mathematical preliminaries relevant to this paper. Specifically, we cover four topics: functional It\^o  calculus, path-dependent FBSDEs, the Euler scheme for path-dependent SDEs, and the weak approximation of martingale representations.

\subsection{Functional It\^o calculus}

Functional It\^o calculus extends classical It\^o calculus to functionals that depend on the entire path of a stochastic process. We present a concise review of functional It\^o calculus in this section and refer the reader to Bally et al.~\cite{functionalito} for more details.
Let $D([0,T],\mathbb{R}^{d})$ denote the collection of all  c$\grave{\text{a}}$dl$\grave{\text{a}}$g paths from $[0,T]$ to $\mathbb{R}^{d}$. We define the space of stopped paths as
\begin{equation}
	\Lambda_T:= \{(s,\omega(s \wedge \cdot))|\, (s,\omega)\in [0,T]\times D([0,T],\mathbb{R}^{d})\}\,,
\end{equation}
endowed with the metric 
$$d_\infty((s,\omega),(s',\omega')):=|s-s'|+\sup_{0\le u \le T} |\omega(s \wedge u) - \omega'(s' \wedge u)|= |s-s'|+\lVert \omega_s-\omega'_{s'} \rVert_\infty$$ 
where $\omega_s:=\omega(s\wedge \cdot)$.

We now introduce the basic notions of regularity for non-anticipative functionals. 
A map $F:[0,T]\times D([0,T],\mathbb{R}^d)\rightarrow \mathbb{R}^k$ is called a non-anticipative functional if $F(s,\omega)=F(s,\omega_s)$
for all $(s,\omega) \in [0,T]\times D([0,T],\mathbb{R}^d).$
A non-anticipative functional $F$ is said to be 
continuous at fixed times if for all $ u\in [0,T]$, $F(s,\cdot)$ is continuous under the uniform norm. 
It is said to be	
left-continuous  (right-continuous, respectively) if for any $(s,\omega)\in \Lambda_T$ and $\epsilon>0$, there exists $\eta>0$ such that    
$|F(s,\omega)-F(s',\omega')|<\epsilon$ whenever  
$s>s'$ ($s<s'$, respectively) and $d_\infty((s,\omega),(s',\omega'))<\eta$ 
for
$(s',\omega')\in \Lambda_T.$
A non-anticipative functional $F$ is said to be
horizontally differentiable at $(s,\omega)\in \Lambda_T$ if
$$\partial_t F(s,\omega)=\lim_{h \rightarrow 0^+} \dfrac{F(s+h,\omega)-F(s,\omega)}{h}$$
exists. We say $F$ is  horizontally differentiable if $\partial_t F(s,\omega)$ exists for all $(s,\omega)\in \Lambda_T$.
Similarly, $F$ is 
vertically differentiable at $(s,\omega)$ if for $i=1,2,\cdots,d$,
$$\partial_i F(s,\omega)=\lim_{h\rightarrow 0} \dfrac{F(s,\omega_s+he_i \mathbf{1}_{[s,T]})-F(s,\omega)}{h}$$
exists where $\{e_i\}_{i=1,\cdots,d}$ is the standard basis of $\mathbb{R}^d$. Define $\nabla_{\omega} F(s,\omega)=(\partial_i F(s,\omega),i=1,2,\cdots,d)$. 
We say that $F$ is vertically differentiable if $\nabla_{\omega} F(s,\omega)$ exists for all $(s,\omega)\in \Lambda_T$. The $k$-th order vertical derivatives, denoted by $\nabla_{\omega}^k F$, can be defined analogously for $k \ge 2$. 
Finally, $F$ is said to be 
boundedness-preserving if for every compact set $K$ of $\mathbb{R}^d$
and $s_0 \in [0,T)$, there exists a constant $C>0$ such that 
$ |F(s,\omega)| <C$
for all $ s \in [0,s_0]$ and $ (s,\omega) \in \Lambda_T$ with $  \omega([0,s]) \subseteq K$.

We denote by $C_{b,r}^{1,2}(\Lambda_T)$ the class of left-continuous functionals $F$ that are horizontally differentiable with $\partial_t F$ boundedness-preserving and continuous at fixed times, and twice vertically differentiable with boundedness-preserving, right-continuous derivatives $\nabla_{\omega}F$ and $\nabla_{\omega}^2F$.
A non-anticipative functional $F$ is called locally regular if there exist an increasing sequence of stopping times $(\tau_n)_{n\in \mathbb{N}}$ with $\tau_0=0$ and $\tau_n \rightarrow \infty$ as $n \rightarrow \infty$, and a sequence  $(F_n)_{n\in\mathbb{N}} \subseteq C_{b,r}^{1,2}(\Lambda_T)$ such that
$$F(s,\omega)=\sum_{n\in \mathbb{N}} F_n(s,\omega)\mathbf{1}_{[\tau_n(\omega),\tau_{n+1}(\omega))}(s)\,,\; \forall (s,\omega)\in \Lambda_T\,.$$
We denote by $C_{loc,r}^{1,2}(\Lambda_T)$ the set of all locally regular non-anticipative functionals.


\subsection{Path-dependent FBSDEs}
This section introduces underlying assumptions and mathematical preliminaries relevant to the study of path-dependent BSDEs.
See Peng and Wang~\cite{pathbsdebm} for more details.
Let $(W_s)_{s\ge0}$ be an $\ell$-dimensional Brownian motion
on a probability space $(\Omega, \mathcal{F}, \mathbb{P})$ with the natural filtration $(\mathcal{F}_s)_{s\ge0}$. 
For $m\in \mathbb{N}$, we define the  spaces
\begin{equation}
\fitdisplay{
\begin{aligned}
		\mathbb{L}^0(t,T;\mathbb{R}^m)&:=\{X:[t,T]\to\mathbb{R}^m\,|\, X \textnormal{ is progressively measurable}\}\,,\\ 
		\mathbb{S}^2(t,T;\mathbb{R}^m)&:=\Big\{X \in \mathbb{L}^0(t,T;\mathbb{R}^m)\,|\, X  \textnormal{ is c$\grave{\text{a}}$dl$\grave{\text{a}}$g in time and  }   \mathbb{E} [ \sup_{t\le u \le T} |X(u)|^2 ] < \infty \Big\}\,, \\
		\mathbb{H}^2(t,T;\mathbb{R}^m)&:=\Big\{Z \in \mathbb{L}^0(t,T;\mathbb{R}^m)\,|\, \mathbb{E}\Big[\int_t^T |Z(u)|^2\,du\Big]<\infty \Big\}\,.
	\end{aligned}
}
\end{equation}
For $t \in [0, T]$ and $\gamma \in D([0,T],\mathbb{R}^{d})$, consider the path-dependent SDE 
\begin{equation} \label{eqn:psde}
	\begin{aligned}
		& X^{\gamma_{t}}(s)=\gamma_{t}(t)+\int_{t}^{s} b(u,X_{u}^{\gamma_{t}})\,du+ \int_{t}^{s} \sigma(u,X_{u}^{\gamma_{t}})\, dW(u)\,,  && t\le s \le T \\
		& X^{\gamma_{t}}(s)=\gamma_{t}(s):=\gamma(t \wedge s)\,,  && 0 \le s \le t
	\end{aligned}
\end{equation}
where $b:\Lambda_T \rightarrow \mathbb{R}^d$ and $\sigma: \Lambda_T \rightarrow \mathbb{R}^{d\times \ell} (\mathbb{R})$ are non-anticipative functionals.

We consider the following assumptions regarding the non-anticipative functionals $b$ and $\sigma$. 
\begin{assume} \label{asmsde}
	There exists a constant $K_1>0$ such that for all $s,s'\in [0,T]$ and $\omega, \omega'\in D[0,T],\mathbb{R}^{d}),$
	\begin{equation}
		|b(s,\omega)-b(s',\omega')|+\lVert \sigma(s,\omega)-\sigma(s',\omega')\rVert \le K_1(\sqrt{|s-s'|}+\lVert\omega_s-\omega'_{s'}\rVert_{\infty})\,.
	\end{equation}
	Moreover, there exists a constant $K_2>0$ such that
	\begin{equation}
		\sup_{0\le u \le T} (|b(u,\bar{0})|+\lVert\sigma(u,\bar{0})\rVert) \le  K_2 
	\end{equation}
	where $\bar{0}$ denotes the constant path that takes value $0$ on $[0,T]$. 
\end{assume}

\begin{proposition} \label{regsde}
	Let Assumption \ref{asmsde} hold.
	Then for any $t \in [0, T]$ and $\gamma \in D([0,T],\mathbb{R}^{d})$, there exists a  unique solution $X^{\gamma_{t}}$
	to \eqref{eqn:psde} in $\mathbb{S}^2(0,T;\mathbb{R}^d)$.
	Moreover, there exists a constant $C>0$ such that
	\begin{enumerate}[]
		\item $\mathbb{E}\left[ \sup_{u\in [0,T]}  |X^{\gamma_{t}}(u) |^2\right] \le C(1+\lVert \gamma_{t} \rVert_{\infty}^2)\,,$
		\item $\mathbb{E} \left[|X^{\gamma_{t}}({s})-X^{\gamma_{t}}(s')|^2\right] \le C(1+\lVert \gamma_{t} \rVert_{\infty}^2)|s-s'|\,,$
		\item $\mathbb{E}\left[ \sup_{u\in [0,T]} |X^{\gamma_t}(u)-X^{\gamma'_{t'}}(u)|^2 \right] \le C(1+\lVert\gamma_t\rVert_{\infty}^2+\lVert \gamma'_{t'}\rVert_{\infty}^2)(\lVert \gamma_t - \gamma'_{t'}\rVert_{\infty}^2+|t-t'|)$
	\end{enumerate}
	for all $t,t' \in [0, T]$, $\gamma,\gamma' \in D([0,T],\mathbb{R}^{d})$, and $s,s'\in [t,T],$ 
\end{proposition}

\noindent The first two inequalities can be proven directly using standard arguments for SDEs, while the last inequality is established in \cite[Lemma B.3]{seongje}.

We now consider the BSDE
\begin{equation} \label{eqn:pbsde}
	\begin{aligned}
		Y^{\gamma_{t}}(s)=g(X_{T}^{\gamma_{t}})+\int_{s}^{T} f(u,X_{u}^{\gamma_{t}},Y^{\gamma_{t}}(u),Z^{\gamma_{t}}(u))\,du-\int_{s}^{T} Z^{\gamma_{t}}(u)\, dW(u) \,,\;   t\le s\le T\,.
	\end{aligned}
\end{equation}

\begin{assume} \label{asmbsde} Let $f:\Lambda_T \times \mathbb{R}\times \mathbb{R}^\ell \rightarrow \mathbb{R}$ be progressively measurable and $g:D([0,T],\mathbb{R}^{d}) \rightarrow \mathbb{R}$ be $\mathcal{F}_{T}$-measurable.
	There exists a constant $K_1>0$ such that for all $s,s'\in [0,T],$ $\omega, \omega'\in D[0,T],\mathbb{R}^d) $ and $y,y',z,z'\in \mathbb{R},$
	\begin{equation}
		\begin{aligned}
			&|f(s,\omega,y,z)-f(s',\omega',y',z')|\le K_1(\sqrt{|s-s'|}+\lVert\omega_s-\omega'_{s'}\rVert_{\infty}+|y-y'|+|z-z'|)\,,\\
			&|g(\omega)-g(\omega')|\le K_1 \lVert\omega-\omega'\rVert_{\infty}\,.
		\end{aligned}
	\end{equation}
	Moreover, there exists a constant $K_2>0$ such that,
	\begin{equation}
		\sup_{0\le u \le T} (|f(u,\bar{0},0,0)|+|g(\bar{0})|)\le K_2\,.
	\end{equation}
\end{assume}

\begin{proposition}\label{eqn:BSDE_soln}
	Let Assumptions \ref{asmsde} and \ref{asmbsde} hold. For $t \in [0, T]$ and $\gamma \in D([0,T],\mathbb{R}^{d})$, there exists a unique  solution $(Y^{\gamma_t},Z^{\gamma_t})$ to \eqref{eqn:pbsde} in $\mathbb{S}^2(t,T;\mathbb{R})\times \mathbb{H}^2(t,T;\mathbb{R}^\ell)$.
\end{proposition}

\subsection{Euler scheme for path-dependent SDEs}

We now demonstrate the Euler scheme for the SDE \eqref{eqn:psde} together with its convergence rate.
Let $\pi=\{t=t_0<t_1<t_2<\ldots <t_n=T\}$ be a time partition with mesh size $|\pi|:=\max_{0 \le i \le n-1} |t_{i+1}-t_i|$. 
We define a process $X^{\gamma_{t},\pi}$ on $[0,T]$ by setting
$X^{\gamma_{t},\pi}(s,\omega)=\gamma(s)$ for $0 \le s \le t_0$ and for  $i=0,1, \cdots,n-1$, 
\begin{equation} \label{eulerscheme}
	\begin{aligned}
		X^{\gamma_{t},\pi}(s,\omega)
		&=X^{\gamma_{t},\pi}(t_i,\omega)\,,\qquad s\in [t_i,t_{i+1})\,,\\
		X^{\gamma_{t},\pi}(t_{i+1},\omega)
		&=X^{\gamma_{t},\pi}(t_{i},\omega)
		+b(t_{i},X^{\gamma_{t},\pi}(\cdot\wedge t_i,\omega))(t_{i+1}-t_i)\\
		&\quad+\sigma(t_{i},X^{\gamma_{t},\pi}(\cdot\wedge t_i,\omega))
		(\omega(t_{i+1}-)-\omega(t_{i}-)).
	\end{aligned}
\end{equation}
According to Cont and Lu~\cite{weakmtg}, under Assumption \ref{asmsde}, there exists a constant $C>0$ such that
\begin{equation}
	\label{eqn:eulerconv}
	\mathbb{E}\left[ \sup_{s\in [0, T]} | X^{\gamma_{t}}(s)-X^{\gamma_{t}, \pi}(s)|^2 \right] \le C(1+\lVert \gamma_{t} \rVert_{\infty}^2)\big(|\pi|+|\pi|\log\frac{1}{|\pi|}\big) \,.
\end{equation}

For the Euler approximation $X^{\gamma_{t},\pi}$, consider the BSDE 
\begin{equation}
	\begin{aligned}
		Y^{\gamma_{t},\pi}(s)
		&=g(X_{T}^{\gamma_{t},\pi})
		+\int_{s}^{T} f(u,X_{u}^{\gamma_{t},\pi},Y^{\gamma_{t},\pi}(u),Z^{\gamma_{t},\pi}(u))\,du\\
		&\quad-\int_{s}^{T} Z^{\gamma_{t},\pi}(u)\,dW(u)\,,\qquad t\le s\le T\,.
	\end{aligned}
\end{equation}
According to Zhang~\cite{BSDE}, under 
Assumptions \ref{asmsde} and \ref{asmbsde}, this BSDE has a unique solution $(Y^{\gamma_{t},\pi},\,Z^{\gamma_{t},\pi})$ in $\mathbb{S}^2(t,T;\mathbb{R})\times \mathbb{H}^2(t,T;\mathbb{R}^\ell)$ and 
\begin{equation}
	\label{eqn:bsde}
	\begin{aligned}
	&\mathbb{E}\bigg[ \sup_{t\le u \le T} |Y^{\gamma_t}(u)-Y^{\gamma_t,\pi}(u)|^2\bigg]
	+\mathbb{E}\bigg[ \int_t^T |Z^{\gamma_t}(u)-Z^{\gamma_t,\pi}(u)|^2 du \bigg]\\
	&\qquad \le C(1+\lVert \gamma_{t} \rVert_{\infty}^2)\big(|\pi|+|\pi|\log\frac{1}{|\pi|}\big)
	\end{aligned}
\end{equation}
for some constant $C>0.$

\section{Picard iteration method}
\label{sec:Picard}

We develop a numerical scheme for path-dependent FBSDEs with non-anticipative coefficients. 
Fix $t \in [0, T]$, $\gamma \in D([0,T],\mathbb{R}^{d})$
and a time partition
$\pi=\{t=t_0<t_1<t_2<\ldots <t_n=T\}$
throughout this section. Define the concatenation operator between two paths $\omega$ and $\tilde{\omega}$ as
$$\omega \mathop{\oplus}_s \tilde{\omega} =
\begin{cases}
	\omega(u), & \mbox{if } u\in [0,s) \,,\\
	\tilde{\omega}(u)-\tilde{\omega}(s)+\omega(s), & \mbox{if } u\in [s,T]\,.
\end{cases}$$
We construct a sequence of processes $(Y^{\gamma_{t},\pi,m},Z^{\gamma_{t},\pi,m})_{m\in\mathbb{N}\cup\{0\}}$ as follows. Let
$W^k$ be   the $k$-th coordinate of the Brownian motion $W$
and define  $\Delta W_{i}^k:=W^k(t_{i+1})-W^k(t_i)$ for  $k=1,2,\cdots,\ell$ and $i=0,1,\cdots,n-1$.   
We initialize with $(Y^{\gamma_t,\pi,0},Z^{\gamma_t,\pi,0})=(0,0)$ and then, for $i=0,1,\dots,n$   and $k=1,2,\dots,\ell$, define iteratively 
\begingroup\footnotesize
\begin{equation}\label{scheme}
\fitdisplay{
\begin{aligned}
		Y^{\gamma_{t},\pi,m+1}(t_i,\omega)&= \mathbb{E}\bigg[ g(X_{T}^{\gamma_{t},\pi}(\omega \mathop{\oplus}_{t_i} W_T))\\
		&
		\quad\left.+\sum_{j=i}^{n-1} f(t_j,X_{t_j}^{\gamma_{t},\pi}(\omega \mathop{\oplus}_{t_i} W_T),Y^{\gamma_{t},\pi,m}(t_j,\omega \mathop{\oplus}_{t_i} W_T),Z^{\gamma_{t},\pi,m}(t_j,\omega \mathop{\oplus}_{t_i} W_T))(t_{j+1}-t_j)\right]\\
		Z_{k}^{\gamma_{t},\pi,m+1}(t_i,\omega)&=\mathbb{E}\Bigg[\frac{\Delta W_i^k}{t_{i+1}-t_i} \Bigg( g(X_{T}^{\gamma_{t},\pi}(\omega \mathop{\oplus}_{t_i} W_T))\\
		&\quad+\sum_{j=i+1}^{n-1} f(t_j,X_{t_j}^{\gamma_{t},\pi}(\omega \mathop{\oplus}_{t_i} W_T),Y^{\gamma_{t},\pi,m}(t_j,\omega \mathop{\oplus}_{t_i} W_T),Z^{\gamma_{t},\pi,m}(t_j,\omega \mathop{\oplus}_{t_i} W_T))(t_{j+1}-t_j)\Bigg) \Bigg]
	\end{aligned}
}
\end{equation}
\endgroup
with the  convention $\Delta W_{n}^k=0$  and extend $(Y^{\gamma_t,\pi,m+1},Z_k^{\gamma_t,\pi,m+1})$ to RCLL processes by piecewise constant interpolation.

A motivation for this iterative definition  is as follows. 	
Given   $(Y^{\gamma_t,\pi,m},Z^{\gamma_t,\pi,m})$, 
let $(\tilde{Y},\tilde{Z})$ be a solution to the BSDE 
\begin{equation}\label{eqn:motiv}
\fitdisplay{
\tilde{Y}(s)=g(X_T^{\gamma_t,\pi})+\int_s^Tf(\underline{u},X_{\underline{u}}^{\gamma_t,\pi},Y^{\gamma_t,\pi,m}(\underline{u}),Z^{\gamma_t,\pi,m}(\underline{u}))\,du-\int_s^T\tilde{Z}(u)\,dW(u) \,,\;t\le s\le T
}
\end{equation}
where $\underline{u}=t_i$ for $t_i\le u<t_{i+1}$.
We then define
$Y^{\gamma_t,\pi,m+1}(s)=\tilde{Y}(s)$,  $Z^{\gamma_t,\pi,m+1}(s)=\tilde{Z}(s)$ for $s=t_0,\cdots,t_n$ and
extend them to RCLL processes by piecewise constant interpolation.
This construction represents a typical  Picard-type iteration procedure for BSDEs and is used to compute the next iterative step $(Y^{\gamma_t,\pi,\,m+1}, Z^{\gamma_t,\pi,\,m+1}).$

We now show that $\tilde{Y}(t_i)$ and $\tilde{Z}(t_i)$ can be represented in the expectation form given in \eqref{scheme}. 
The expression for $\tilde{Y}(t_i)$ is obtained directly by taking the conditional expectation $\mathbb{E}(\,\cdot\,|\mathcal{F}_{t_i})$ to \eqref{eqn:motiv} with $s=t_i.$
For $\tilde{Z}(t_i)$, observe that 
\begin{equation} 
	\begin{aligned}
		\int_t^T\tilde{Z}(u)\,dW(u)
		&=-\tilde{Y}(t)+g(X_T^{\gamma_t,\pi})\\
		&\qquad+\sum_{j=0}^{n-1} f(t_j,X_{t_j}^{\gamma_{t},\pi},Y^{\gamma_{t},\pi,m}(t_j),Z^{\gamma_{t},\pi,m}(t_j))(t_{j+1}-t_j)\,.
	\end{aligned}
\end{equation}
Taking the conditional expectation $\mathbb{E}(\,\cdot\,|\mathcal{F}_s)$ for $s\in [t,T],$ we have
$\int_t^s\tilde{Z}(u)\,dW(u)		=-\tilde{Y}(t)+F^{\pi,m}(s,W_s)$
where
\begin{equation}
	\begin{aligned}
		F^{\pi,m}(s,\omega)&:= \mathbb{E}\bigg[ g(X_{T}^{\gamma_{t},\pi}(\omega \mathop{\oplus}_{s} W_T))\\
		&+\sum_{j=0}^{n-1} f(t_j,X_{t_j}^{\gamma_{t},\pi}(\omega \mathop{\oplus}_{s} W_T),\\
		&\qquad Y^{\gamma_{t},\pi,m}(t_j,\omega \mathop{\oplus}_{s} W_T),Z^{\gamma_{t},\pi,m}(t_j,\omega \mathop{\oplus}_{s} W_T))(t_{j+1}-t_j)\bigg].
	\end{aligned}
\end{equation}
As the processes $X^{\gamma_t,\pi}$, $Y^{\gamma_t,\pi,m}$, $Z^{\gamma_t,\pi,m}$ depend only on $$\omega(t_0),\,\omega(t_1-)-\omega(t_0),\cdots,\omega(t_n-)-\omega(t_{n-1}-)\,,$$ 
by an argument similar to that given in Cont and Lu~\cite[Theorem 3.2]{weakmtg}, we obtain
$F^{\pi,m} \in C_{loc,r}^{1,2}(\Lambda_T)$.   Then, by weak approximation of martingale representation formula described in Cont and Lu~\cite[Theorem 3.1 and Corollary 3.1]{weakmtg}, it follows that 
$\tilde{Z}(t_i)=\nabla_{\omega}F^{\pi,m}(t_i,\omega)$ and 
$\tilde{Z}(t_i)$ has the form given in \eqref{scheme}.

\subsection{Smooth approximation}
\label{sec:smooth}

In this section, 
we prove that the scheme in \eqref{scheme} converges to the solution $(Y^{\gamma_t},Z^{\gamma_t})$ under the assumption of sufficiently smooth coefficients. 
We first introduce  the Fr$\acute{e}$chet derivative with respect to the path variable.
For $\varphi:[0,T]\times D([0,T],\mathbb{R}^{d})\to\mathbb{R}$, the Fr$\acute{e}$chet derivative of $\varphi$
with respect to $\omega$ at  $(s,\omega)$ is  defined as a linear operator  $D\varphi(s,\omega)$ satisfying
\begin{equation}
	\lim_{\lVert \tilde{\omega} \rVert_{\infty} \rightarrow 0} \dfrac{| \varphi(s,\omega+\tilde{\omega})-\varphi(s,\omega)-D\varphi(s,\omega)(\tilde{\omega})|}{\lVert \tilde{\omega} \rVert_{\infty}}=0\,.
\end{equation}
A function $\varphi$ is said to be
Fr$\acute{e}$chet differentiable  if
$\varphi$ admits a Fr$\acute{e}$chet derivative at every
$(s,\omega)\in [0,T]\times D([0,T],\mathbb{R}^{d}).$	
We say the Fr$\acute{e}$chet derivative $D\varphi$ is
bounded if $| D\varphi(s,\omega)(\tilde{\omega})| \le C \lVert \tilde{\omega} \rVert_{\infty}$ for all $(s,\omega,\tilde{\omega})\in [0,T]\times D([0,T],\mathbb{R}^{d})\times D([0,T],\mathbb{R}^{d})$, and is  continuous if  the map $(s,\omega) \mapsto D\varphi(s,\omega)(\tilde{\omega})$ is continuous for all $\tilde{\omega}\in D([0,T],\mathbb{R}^{d})$.

\begin{assume} \label{convasm} Consider the following conditions.
	\begin{enumerate}[(i)]
		\item For    $\varphi=b_i,\sigma_{ij},g,$ $1\le i\le d,$ $1\le j\le \ell,$ the function $\varphi$ is twice Fr$\acute{e}$chet differentiable  with bounded derivatives, and $D^2\varphi$ is uniformly Lipschitz continuous in $\omega$.
		\item The function $f$ is jointly twice differentiable in $(\omega, y,z)$ with bounded derivatives, and the second-order derivatives are uniformly Lipschitz continuous in $(\omega, y, z)$.

	\end{enumerate}
\end{assume}

We observe that under Assumptions \ref{asmsde} and \ref{asmbsde},  by Proposition \ref{eqn:BSDE_soln}, the BSDE \eqref{eqn:pbsde} has a unique  solution $(Y^{\gamma_t},Z^{\gamma_t})$ in $\mathbb{S}^2(t,T;\mathbb{R})\times \mathbb{H}^2(t,T;\mathbb{R}^\ell)$.
The following lemma is directly obtained from Wang et al.~\cite[Theorem 3.1]{pathkac}, thus we omit the proof.

We establish the regularity estimates for $(Y^{\gamma_t},Z^{\gamma_t})$
by first approximating the path dependence with finitely many observations.
Throughout the following arguments, unless otherwise stated,
Assumptions~\ref{asmsde}, \ref{asmbsde}, and~\ref{convasm} hold,
and $C$ denotes a positive constant, possibly varying from line to line,
depending only on $T,d,\ell,K_1,K_2$.
For notational simplicity, all proofs are presented for $d=\ell=1$;
the extension to general dimensions is straightforward.
Fix $t<T$ and $\gamma\in D([0,T],\mathbb R^d)$, and write
$(X,Y,Z)=(X^{\gamma_t},Y^{\gamma_t},Z^{\gamma_t})$.
Let $\rho_N=\{t=r_0<r_1<\cdots<r_N=T\}$ be deterministic partitions
such that $|\rho_N|\to0$. We refer to the grid points $r_j$ as the
observation times. These are auxiliary partitions, distinct from
the partition $\pi$ used in the regularity estimates below.
For a path $\omega$ satisfying $\omega=\gamma$ on $[0,t]$, define
\begin{equation}\label{eq:reg-sampling}
 \omega^{\rho_N}(v)=
 \begin{cases}
  \gamma(v),&0\le v<t,\\
  \omega(r_k),&r_k\le v<r_{k+1},\quad k=0,\ldots,N-1,\\
  \omega(T),&v=T.
 \end{cases}
\end{equation}
We set
\begin{equation}\label{eq:reg-sampled-coefficients}
 \begin{aligned}
 b^N(s,\omega)&=b\bigl(s,(\omega^{\rho_N})_s\bigr),
 &\sigma^N(s,\omega)&=\sigma\bigl(s,(\omega^{\rho_N})_s\bigr),\\
 f^N(s,\omega,y,z)&=f\bigl(s,(\omega^{\rho_N})_s,y,z\bigr),
 &g^N(\omega)&=g(\omega^{\rho_N}),
 \end{aligned}
\end{equation}
and consider the FBSDE
\begin{equation}\label{eq:reg-aux-fbsde}
\fitdisplay{
\begin{aligned}
 X^N(s)&=\gamma(t)+\int_t^s b^N(q,X^N_q)\,dq
                         +\int_t^s\sigma^N(q,X^N_q)\,dW(q),\\
 Y^N(s)&=g^N(X^N_T)+\int_s^T
                f^N(q,X^N_q,Y^N(q),Z^N(q))\,dq
                         -\int_s^T Z^N(q)\,dW(q),
 \qquad t\le s\le T,
 \end{aligned}
}
\end{equation}
with $X^N=\gamma$ on $[0,t]$.

By \cite[Proposition~2.2]{pathkac} and \cite[Chapters~3--4]{BSDE}, \eqref{eq:reg-aux-fbsde} has a unique solution in
$S^2(0,T;\mathbb R^d)\times S^2(t,T;\mathbb R)
\times H^2(t,T;\mathbb R^\ell)$ and the usual SDE and BSDE estimates give
\begin{equation}\label{eq:reg-moment-estimates}
 \begin{aligned}
 &\mathbb E\left[
    \sup_{0\le s\le T}|X^N(s)|^p
   +\sup_{t\le s\le T}|Y^N(s)|^p
   +\left(\int_t^T|Z^N(s)|^2\,ds\right)^{p/2}\right]
 \\
 &\qquad\le C_p(1+\|\gamma_t\|_\infty^p),\qquad p=2,4,
 \end{aligned}
\end{equation}
and the corresponding estimates hold for $(X,Y,Z)$. Moreover, the standard SDE and BSDE stability yield
\begin{equation}\label{eq:reg-fbsde-convergence}
 \begin{aligned}
 &\mathbb E\left[
   \sup_{t\le s\le T}|X^N(s)-X(s)|^2
  +\sup_{t\le s\le T}|Y^N(s)-Y(s)|^2
  +\int_t^T|Z^N(s)-Z(s)|^2\,ds\right]\\
 &\qquad\le C\mathbb E\|X^{\rho_N}-X\|_\infty^2
 \longrightarrow0.
 \end{aligned}
\end{equation}
Denoting by
$\sigma^N(s)=\sigma^N(s,X^N_s)$, we also have
\begin{equation}\label{eq:reg-diffusion-convergence}
 \mathbb E\int_t^T\|\sigma^N(s)-\sigma(s,X_s)\|^2\,ds
 \le C\mathbb E\|X^{\rho_N}-X\|_\infty^2\longrightarrow0.
\end{equation}

To describe the dependence on the observation values, put
$x_0=\gamma(t)$ and, for $k=0,\ldots,N$, define
\begin{equation}\label{eq:reg-coordinate-path}
 \omega^{N,k}(x_1,\ldots,x_k)(v)=
 \begin{cases}
  \gamma(v),&0\le v<t,\\
  x_j,&r_j\le v<r_{j+1},\quad 0\le j<k,\\
  x_k,&r_k\le v\le T.
 \end{cases}
\end{equation}
In particular, $\omega^{N,0}=\gamma_t$, and the last line for $k=N$
specifies the value at $T$. For $r_k\le s<r_{k+1}$, let
\begin{equation}\label{eq:reg-coordinate-coefficients}
 \begin{aligned}
 b^{N,k}(s,x_1,\ldots,x_k)
   &=b\bigl(s,\omega^{N,k}(x_1,\ldots,x_k)\bigr),\\
 \sigma^{N,k}(s,x_1,\ldots,x_k)
   &=\sigma\bigl(s,\omega^{N,k}(x_1,\ldots,x_k)\bigr),\\
 f^{N,k}(s,x_1,\ldots,x_k,y,z)
   &=f\bigl(s,\omega^{N,k}(x_1,\ldots,x_k),y,z\bigr).
 \end{aligned}
\end{equation}
We also write $g^N(x_1,\ldots,x_N)
=g(\omega^{N,N}(x_1,\ldots,x_N))$.

We define functions $u_k^N(s,x_1,\ldots,x_k,x)$ by backward induction,
where $x_1,\ldots,x_k$ are the stored observations and $x$ is the current
state. At $T$, let
$u_N^N(T,x_1,\ldots,x_N,x)=g^N(x_1,\ldots,x_N)$.
For $k=N-1,\ldots,0$, suppose that $u_{k+1}^N$ has been defined.
For
$a\in[r_k,r_{k+1})$ and fixed $(x_1,\ldots,x_k,x)$, solve on
$[a,r_{k+1}]$
\begin{equation}\label{eq:reg-local-fbsde}
 \begin{aligned}
 \widetilde X(q)
   ={}&x+\int_a^q b^{N,k}(v,x_1,\ldots,x_k)\,dv
       +\int_a^q\sigma^{N,k}(v,x_1,\ldots,x_k)\,dW(v),\\
 \widetilde Y(q)
   ={}&u_{k+1}^N\bigl(r_{k+1},x_1,\ldots,x_k,
                       \widetilde X(r_{k+1}),\widetilde X(r_{k+1})\bigr)\\
   &+\int_q^{r_{k+1}}f^{N,k}
        (v,x_1,\ldots,x_k,\widetilde Y(v),\widetilde Z(v))\,dv
       -\int_q^{r_{k+1}}\widetilde Z(v)\,dW(v),
 \end{aligned}
\end{equation}
and set $u_k^N(a,x_1,\ldots,x_k,x)=\widetilde Y(a)$.

For $1\le j\le k$ and $r_k\le s<r_{k+1}$, denote the derivatives of $b^{N,k}$, $\sigma^{N,k}$,
and $f^{N,k}$ with respect to $x_j$ by $b_j^N(s)$, $\sigma_j^N(s)$,
and $f_j^N(s)$, evaluated at
$(x_1,\ldots,x_k)=(X^N(r_1),\ldots,X^N(r_k))$ and
$(y,z)=(Y^N(s),Z^N(s))$. They are set to zero for $s<r_j$.
In these derivatives, $y,z$ are held fixed. Moreover, set
\begin{equation}\label{eq:reg-alpha-beta}
 \begin{aligned}
 g_j^N&=\partial_{x_j}g^N(X^N(r_1),\ldots,X^N(r_N)),\\
 \alpha(s)&=\partial_yf^N(s,X^N_s,Y^N(s),Z^N(s)),
 &\beta(s)&=\partial_zf^N(s,X^N_s,Y^N(s),Z^N(s)).
 \end{aligned}
\end{equation}
By construction and Assumption~\ref{convasm}, the coordinate
coefficients are differentiable. Moreover, by an argument analogous
to that in \cite[Lemma~5.2.3]{BSDE}, each $u_k^N$ is continuously
differentiable in its spatial variables.
For $r_k\le s<r_{k+1}$, define
\begin{equation}\label{eq:reg-variation-definitions}
 \begin{aligned}
 \nabla Y^N(s)
 &=\partial_xu_k^N\bigl(s,X^N(r_1),\ldots,X^N(r_k),X^N(s)\bigr),\\
 \partial_jY^N(s)
 &=\partial_{x_j}u_k^N\bigl(s,X^N(r_1),\ldots,X^N(r_k),X^N(s)\bigr),
 \qquad j\le k.
 \end{aligned}
\end{equation}

At the terminal time, set
\[
 \nabla Y^N(T)=0,\qquad
 \partial_jY^N(T)=g_j^N,\qquad 1\le j\le N.
\]
At each observation date before $T$, the values in
\eqref{eq:reg-variation-definitions} are taken from the interval
beginning at that date.

For each $k$, take $a=r_k$ in \eqref{eq:reg-local-fbsde}.
We denote by $\nabla Z^N$ on $[r_k,r_{k+1})$ the $H^2$ derivative
of $\widetilde Z$ with respect to the current initial state $x$,
evaluated at
\[
 (x_1,\ldots,x_k,x)
 =(X^N(r_1),\ldots,X^N(r_k),X^N(r_k)).
\]
The derivative is understood as the $H^2$ limit of the corresponding
difference quotients, as in \cite[equation~(5.2.6)]{BSDE}.
The corresponding local parameter variations and It\^o's product
rule give a local semimartingale decomposition of each
$\partial_jY^N$, $j\le k$. Let $\widetilde Z_j^N$ denote the
Brownian integrand in its continuous local martingale part.
The decompositions are derived explicitly in the proof below.
These integrands are joined over successive observation intervals;
they are identified $ds\otimes d\mathbb P$-almost everywhere,
and their values at the observation dates may be chosen arbitrarily.
Thus $\nabla Y^N$ and $\partial_jY^N$ take values in
$\mathbb R^{1\times d}$, whereas $\nabla Z^N$ and
$\widetilde Z_j^N$ take values in $\mathbb R^{\ell\times d}$.
We use the Euclidean norm for vectors and the Frobenius norm for matrices.

The following lemma gives the representation and integrability
properties of these derivative processes.

\begin{lemma}\label{lem:reg-first-variations}
Let Assumptions~\ref{asmsde}, \ref{asmbsde}, and~\ref{convasm} hold.
For each $N$, the processes $\nabla Y^N$ and $\partial_jY^N$
defined above are adapted, with progressively measurable integrands
$\nabla Z^N$ and $\widetilde Z_j^N$, and satisfy
\begin{equation}\label{eq:reg-z-representation}
 Z^N(s)=\nabla Y^N(s)\sigma^N(s)
 \quad ds\otimes d\mathbb P\text{-a.e.},
 \qquad
 \sup_{t\le s\le T}|\nabla Y^N(s)|\le C
 \quad\mathbb P\text{-a.s.}
\end{equation}
The process $\nabla Y^N$ is right-continuous with left limits on
$[t,T]$ and continuous on each $[r_k,r_{k+1})$.
Each $\partial_jY^N$ is continuous on $[r_j,T]$.
Moreover, for each fixed $N$,
\begin{equation}\label{eq:reg-fixed-N-integrability}
 \begin{aligned}
 &\mathbb E\left(\int_t^T|\nabla Z^N(s)|^2\,ds\right)^2<\infty,\\
 &\mathbb E\left[
   \sup_{r_j\le s\le T}|\partial_jY^N(s)|^4
  +\left(\int_{r_j}^T|\widetilde Z_j^N(s)|^2\,ds\right)^2\right]<\infty,
 \qquad 1\le j\le N.
 \end{aligned}
\end{equation}
The constant in \eqref{eq:reg-z-representation} is independent of
$N$, $\gamma_t$, and the second-derivative bounds in
Assumption~\ref{convasm}.
\end{lemma}

\begin{proof}
By uniqueness of the restarted FBSDEs,
\begin{equation}\label{eq:reg-value-along-solution}
 Y^N(s)=u_k^N\bigl(s,X^N(r_1),\ldots,X^N(r_k),X^N(s)\bigr),
 \qquad r_k\le s<r_{k+1}.
\end{equation}

Consider two restarted systems on $[a,T]$ with observation vectors
$\boldsymbol x=(x_1,\ldots,x_k)$ and
$\boldsymbol x'=(x'_1,\ldots,x'_k)$, and current states $x,x'$.
Put $\delta=\max_{1\le j\le k}|x_j-x'_j|$, with $\delta=0$ for $k=0$.
By the standard FBSDE estimates, we have
\begin{equation}\label{eq:reg-restarted-stability}
\fitdisplay{
\begin{aligned}
 \mathbb E\left[
  \sup_{a\le q\le T}|\widetilde X(q)-\widetilde X'(q)|^2
 +\sup_{a\le q\le T}|\widetilde Y(q)-\widetilde Y'(q)|^2
 +\int_a^T|\widetilde Z(q)-\widetilde Z'(q)|^2\,dq\right]
 \le C(|x-x'|^2+\delta^2).
 \end{aligned}
}
\end{equation}
In particular,
\begin{equation}\label{eq:reg-value-lipschitz}
 |u_k^N(a,\boldsymbol x,x)-u_k^N(a,\boldsymbol x',x')|
 \le C(|x-x'|+\delta),
\end{equation}
with $C$ independent of $k,N$ and the initial data.
By \eqref{eq:reg-variation-definitions} and
\eqref{eq:reg-value-lipschitz}, both $|\nabla Y^N|$ and
$|\partial_jY^N|$ are bounded by $C$.

From \eqref{eq:reg-local-fbsde}, we observe the terminal matching identity
\begin{equation}\label{eq:reg-matching}
 u_{j-1}^N(r_j-,x_1,\ldots,x_{j-1},x)
   =u_j^N(r_j,x_1,\ldots,x_{j-1},x,x).
\end{equation}

With the stored values fixed, the derivative of $\widetilde X$
with respect to its current initial value is $1$ on $[a,r_{k+1}]$.
The representation and
variation equations of \cite[Lemma~5.2.3]{BSDE} consequently
give
\begin{equation}\label{eq:reg-live-local}
 \begin{aligned}
 Z^N(s)&=\nabla Y^N(s)\sigma^N(s),\\
 d\nabla Y^N(s)
  &=-[\alpha\nabla Y^N+\beta\nabla Z^N](s)\,ds
       +\nabla Z^N(s)\,dW(s)
 \end{aligned}
\end{equation}
on each observation interval.

To derive the equation for $\partial_jY^N$, fix
$a\in[r_k,r_{k+1})$ and $j\le k$. In \eqref{eq:reg-local-fbsde},
perturb $x_j$ while keeping $a$, the current initial state $x$, and the
other stored coordinates fixed. Denote the total local variations by
$(\nabla_jX,\nabla_jY,\nabla_jZ)$.
After evaluation at the unperturbed initial observations and state at
$a$, these variations satisfy, on $[a,r_{k+1}]$,
\begin{equation}\label{eq:reg-total-local}
 \begin{aligned}
 d\nabla_jX(s)
  &=b_j^N(s)\,ds+\sigma_j^N(s)\,dW(s),
 \qquad\nabla_jX(a)=0,\\
 d\nabla_jY(s)
  &=-[f_j^N+\alpha\nabla_jY+\beta\nabla_jZ](s)\,ds
       +\nabla_jZ(s)\,dW(s).
 \end{aligned}
\end{equation}
The chain rule gives
\begin{equation}\label{eq:reg-chain-rule}
 \nabla_jY(s)=\partial_jY^N(s)+\nabla Y^N(s)\nabla_jX(s).
\end{equation}
Applying It\^o's product rule to the second term on the right yields
\begin{equation}\label{eq:reg-product-rule}
 \begin{aligned}
 d(\nabla Y^N\nabla_jX)
 ={}&[-\alpha\nabla Y^N\nabla_jX
       -\beta\nabla Z^N\nabla_jX
       +\nabla Y^N b_j^N+\nabla Z^N\sigma_j^N]ds\\
 &+[\nabla Z^N\nabla_jX+\nabla Y^N\sigma_j^N]dW.
 \end{aligned}
\end{equation}
Subtracting \eqref{eq:reg-product-rule} from the equation for
$\nabla_jY$ and using \eqref{eq:reg-chain-rule}, we obtain the dynamics of
$\partial_jY^N$:
\begin{equation}\label{eq:reg-memory-local}
 d\partial_jY^N
   =-[q_j^N+\alpha\partial_jY^N+\beta\widetilde Z_j^N]ds
          +\widetilde Z_j^N\,dW,
\end{equation}
where
\begin{equation}\label{eq:reg-memory-source}
 q_j^N=f_j^N+\nabla Y^N b_j^N
          +(\nabla Z^N+\beta\nabla Y^N)\sigma_j^N, 
\end{equation}
and
\begin{equation}\label{eq:reg-memory-integrand}
 \widetilde Z_j^N
   =\nabla_jZ-\nabla Z^N\nabla_jX-\nabla Y^N\sigma_j^N.
\end{equation}

At a later observation date $r_m$, $j<m$, differentiation of
\eqref{eq:reg-matching} with respect to $x_j$ gives
$\partial_jY^N(r_m-)=\partial_jY^N(r_m)$. Thus each $\partial_jY^N$ is continuous on $[r_j,T]$, and the local
equations \eqref{eq:reg-memory-local} join with the prescribed terminal
value $\partial_jY^N(T)=g_j^N$. Thus, for $r_j\le s\le T$,
\begin{equation}\label{eq:reg-memory-bsde}
\partial_jY^N(s)=g_j^N+\int_s^T[q_j^N+\alpha\partial_jY^N           +\beta\widetilde Z_j^N](r)\,dr
  -\int_s^T\widetilde Z_j^N(r)\,dW(r).
\end{equation}
Differentiation of the same matching identity with respect to the
current state instead gives
\begin{equation}\label{eq:reg-current-matching}
 \nabla Y^N(r_j-)=\partial_jY^N(r_j)+\nabla Y^N(r_j).
\end{equation}
Here the two copies of the current state on the right of
\eqref{eq:reg-matching} both vary. At $r_N=T$, the terminal convention
$\nabla Y^N(T)=0$ gives $\nabla Y^N(T-)=g_N^N$.
Taking the right-hand value at each observation date and summing
\eqref{eq:reg-live-local} and the jumps
$-\partial_jY^N(r_j)$ gives, for $t\le a\le s\le T$,
\begin{equation}\label{eq:reg-live-global}
\fitdisplay{
\nabla Y^N(s)-\nabla Y^N(a)
 =-\int_a^s[\alpha\nabla Y^N+\beta\nabla Z^N](r)\,dr
 +\int_a^s\nabla Z^N(r)\,dW(r)
       -\sum_{j:a<r_j\le s}\partial_jY^N(r_j).
}
\end{equation}

We next establish
\begin{equation}\label{eq:reg-derivative-sum}
 \sum_{j=1}^N|g_j^N|\le C,\qquad
 \sum_{j=1}^N\left(|b_j^N(s)|+|\sigma_j^N(s)|
                                          +|f_j^N(s)|\right)\le C,
\end{equation}
almost surely, and almost everywhere in time for the second inequality,
with $C$ independent of $N$, $\gamma_t$, and the second-derivative bounds.
If a differentiable function $h:\mathbb R^k\to\mathbb R$ is
$K_1$-Lipschitz for the maximum norm, choose
$v_j=\operatorname{sgn}(\partial_{x_j}h)$.
Then $|v_j|\le1$ and $\partial_{x_j}h\,v_j=|\partial_{x_j}h|$.
Directional differentiation of the Lipschitz inequality yields
\begin{equation}\label{eq:reg-maxnorm-duality}
 \sum_{j=1}^k|\partial_{x_j}h|
    =Dh(v_1,\ldots,v_k)\le K_1.
\end{equation}
The map \eqref{eq:reg-coordinate-path} satisfies
$\|\omega^{N,k}(\boldsymbol x)-\omega^{N,k}(\boldsymbol x')\|_\infty
=\max_j|x_j-x'_j|$. Therefore
\eqref{eq:reg-maxnorm-duality} applies to $g^N$, $b^{N,k}$,
$\sigma^{N,k}$, and $f^{N,k}$ with $y,z$ fixed, proving
\eqref{eq:reg-derivative-sum} with no factor depending on $N$.

Finally, on each observation interval the current-state variation $(\nabla Y^N,\nabla Z^N)$ is a
linear BSDE with bounded terminal value and zero source. Its fourth
moment estimate gives a finite second moment for the local energy of
$\nabla Z^N$. A finite number of intervals therefore gives the first
assertion of \eqref{eq:reg-fixed-N-integrability}. For each $j$,
$|q_j^N|\le C(1+|\nabla Z^N|)$, so the ordinary fourth moment estimate
for \eqref{eq:reg-memory-bsde} gives the second inequality.
\end{proof}

The next lemma estimates a sum of linear BSDE solutions evaluated at
different deterministic times.
\begin{lemma}\label{lem:reg-finite-sum}
Let $t<r_1<\cdots<r_N=T$ and let $\alpha$ and
$\beta=(\beta_1,\ldots,\beta_\ell)$ be progressively measurable with
$|\alpha|+|\beta|\le L$. Suppose that $\xi_j$ is real-valued and
$\mathcal F_T$-measurable and that $q_j$ is progressively measurable
on $[r_j,T]$. Extend $q_j$ by zero on $[t,r_j)$ and assume
\begin{equation}\label{eq:reg-finite-sum-data}
 \mathbb E\left(\sum_{j=1}^N|\xi_j|
       +\int_t^T\sum_{j=1}^N|q_j(s)|\,ds\right)^4<\infty.
\end{equation}
For each $j$, let $(U_j,V_j)$ solve
\begin{equation}\label{eq:reg-linear-family}
 U_j(s)=\xi_j+\int_s^T[q_j+\alpha U_j+\beta\cdot V_j](r)\,dr
               -\int_s^T V_j(r)\,dW(r),\qquad r_j\le s\le T,
\end{equation}
in the class satisfying
$\mathbb E\sup_{r_j\le s\le T}|U_j(s)|^4
+\mathbb E(\int_{r_j}^T|V_j(s)|^2\,ds)^2<\infty$.
Then there exists $C=C(T,L,\ell)$, independent of $N$ and the times
$r_j$, such that
\begin{equation}\label{eq:reg-finite-sum-bound}
 \mathbb E\left(\sum_{j=1}^N|U_j(r_j)|\right)^4
 \le C\mathbb E\left(\sum_{j=1}^N|\xi_j|
            +\int_t^T\sum_{j=1}^N|q_j(s)|\,ds\right)^4.
\end{equation}
\end{lemma}

\begin{proof}
We first assume that all terminal values and sources are bounded,
with constants independent of their bounds, and remove this restriction
by truncation at the end.
For each $j$, let $(\overline U_j,\overline V_j)$ solve on $[t,T]$
\begin{equation}\label{eq:reg-positive-family}
 \overline U_j(s)=|\xi_j|
  +\int_s^T[|q_j|+\alpha\overline U_j
                  +\beta\overline V_j](r)\,dr
       -\int_s^T\overline V_j(r)\,dW(r).
\end{equation}
The comparison theorem gives $\overline U_j\ge0$ and
$|U_j|\le\overline U_j$ on $[r_j,T]$.
Since $\alpha,\beta$ are common to all the equations, summation gives
\begin{equation}\label{eq:reg-summed-linear-bsde}
 \begin{aligned}
 \sum_j\overline U_j(s)
 ={}&\sum_j|\xi_j|
   +\int_s^T\left[\sum_j|q_j(r)|
       +\alpha(r)\sum_j\overline U_j(r)
       +\beta(r)\sum_j\overline V_j(r)\right]dr\\
 &-\int_s^T\left(\sum_j\overline V_j(r)\right)dW(r).
 \end{aligned}
\end{equation}
Thus $(\sum_j\overline U_j,\sum_j\overline V_j)$ is one real-valued
BSDE with terminal value $\sum_j|\xi_j|$ and generator
$\sum_j|q_j(s)|+\alpha(s)y+\beta(s) z$.
Its Lipschitz constant in $(y,z)$ is bounded by $L$, independently of
$N$. Applying Bouchard et al.~\cite[Theorem~4.1]{BouchardPossamaiTanZhou2018}
with exponent $4$ gives
\begin{equation}\label{eq:reg-same-time-estimate}
 \begin{aligned}
 \mathbb E\sup_{t\le s\le T}\left|\sum_j\overline U_j(s)\right|^4
 &\le C\mathbb E\left[
      \left(\sum_j|\xi_j|\right)^4
     +\left(\int_t^T\sum_j|q_j(s)|\,ds\right)^4\right]\\
 &\le C\mathbb E\left(
      \sum_j|\xi_j|+\int_t^T\sum_j|q_j(s)|\,ds\right)^4.
 \end{aligned}
\end{equation}
The bounded-data assumption
verifies the required source integrability, and the deterministic
exponential weight on the source is absorbed in $C$.

Define
\begin{equation}\label{eq:reg-remaining-processes}
 R(s)=\sum_{j:r_j>s}\overline U_j(s),\qquad
 V(s)=\sum_{j:r_j>s}\overline V_j(s),
 \qquad t\le s\le T,
\end{equation}
and write $R^*=\sup_{t\le s\le T}|R(s)|$.
By nonnegativity of the comparison solutions,
\begin{equation}\label{eq:reg-remaining-domination}
 0\le R(s)\le\sum_j\overline U_j(s),\qquad
 R(t)=\sum_j\overline U_j(t),\qquad R(T)=0.
\end{equation}
In particular, $\mathbb E(R^*)^4$ is bounded by the right-hand side of
\eqref{eq:reg-same-time-estimate}. On every open interval between
successive observation dates, the retained indices are fixed. Each
retained source is zero, since $s<r_j$ and $q_j$ was extended by zero
before $r_j$. The continuous part of the dynamics is therefore
\[
 dR(s)=-[\alpha R+\beta V](s)\,ds+V(s)\,dW(s).
\]
At an observation date, continuity of each $\overline U_j$ gives
\[
 R(r_j)-R(r_j-)=-\overline U_j(r_j).
\]
Adding the changes on the observation intervals and these jumps yields
\begin{equation}\label{eq:reg-remaining-forward}
 R(s)=R(t)-\int_t^s[\alpha R+\beta V](r)\,dr
       +\int_t^s V(r)\,dW(r)
       -\sum_{j:r_j\le s}\overline U_j(r_j).
\end{equation}
Equivalently, since $R(T)=0$,
\begin{equation}\label{eq:reg-remaining-dynamics}
 R(s)=\int_s^T[\alpha R+\beta V](r)\,dr
           -\int_s^T V(r)\,dW(r)
           +\sum_{j:s<r_j\le T}\overline U_j(r_j).
\end{equation}
The process $R$ is right-continuous with left limits and has only the
jumps displayed above.

For fixed $N$, the usual fourth-moment estimates for the individual
comparison BSDEs imply
\[
 \mathbb E\sup_{t\le s\le T}|\overline U_j(s)|^4
 +\mathbb E\left(\int_t^T|\overline V_j(s)|^2\,ds\right)^2<\infty.
\]
Consequently,
\begin{equation}\label{eq:reg-remaining-fixed-N}
 \begin{aligned}
 \mathbb E\left(\sum_j\overline U_j(r_j)\right)^4
 &\le N^3\sum_j\mathbb E\sup_{t\le s\le T}|\overline U_j(s)|^4
 <\infty,\\
 \mathbb E\left(\int_t^T|V(s)|^2\,ds\right)^2
 &\le N^3\sum_j\mathbb E\left(
                   \int_t^T|\overline V_j(s)|^2\,ds\right)^2
 <\infty.
 \end{aligned}
\end{equation}
Together with
$\mathbb E(R^*)^4<\infty$,
\[
 \mathbb E\int_t^T |R(s)|^2|V(s)|^2\,ds
 \le \bigl(\mathbb E(R^*)^4\bigr)^{1/2}
      \left[\mathbb E\left(\int_t^T|V(s)|^2\,ds\right)^2\right]^{1/2}
 <\infty.
\]
Thus the stochastic integrals below have the required moments.

By boundedness of $\alpha,\beta$ and Cauchy--Schwarz in time for \eqref{eq:reg-remaining-dynamics}, we obtain
\[
 \sum_j\overline U_j(r_j)
 \le[1+L(T-t)]R^*
       +L\sqrt{T-t}\left(\int_t^T|V(s)|^2\,ds\right)^{1/2}
 +\left|\int_t^T V(s)\,dW(s)\right|.
\]
Taking fourth moments and applying the Burkholder--Davis--Gundy
inequality to the last term gives, for a constant $C_0\ge1$ independent
of $N$ and the observation dates,
\begin{equation}\label{eq:reg-remaining-sum-versus-energy}
 \mathbb E\left(\sum_j\overline U_j(r_j)\right)^4
 \le C_0\mathbb E(R^*)^4
      +C_0\mathbb E\left(\int_t^T|V(s)|^2\,ds\right)^2.
\end{equation}

We next estimate the energy in the reverse direction. At $r_j$,
\[
 |R(r_j)|^2-|R(r_j-)|^2
   =-2R(r_j-)\overline U_j(r_j)+|\overline U_j(r_j)|^2.
\]
Applying It\^o's formula to $|R|^2$ between observation dates, adding
these increments, and using $R(T)=0$, we find
\begin{equation}\label{eq:reg-remaining-square-identity}
 \begin{aligned}
 \int_t^T|V(s)|^2\,ds
 ={}&-|R(t)|^2+2\int_t^T\alpha(s)|R(s)|^2\,ds
       +2\int_t^T R(s)\beta(s) V(s)\,ds\\
 &-2\int_t^T R(s)V(s)\,dW(s)
       +2\sum_jR(r_j-)\overline U_j(r_j)
       -\sum_j|\overline U_j(r_j)|^2.
 \end{aligned}
\end{equation}
Moreover,
\[
 2R(s)\beta(s) V(s)
    \le \tfrac12|V(s)|^2+2L^2|R(s)|^2,
 \qquad
 \sum_jR(r_j-)\overline U_j(r_j)
    \le R^*\sum_j\overline U_j(r_j).
\]
Hence, we derive
\begin{equation}\label{eq:reg-remaining-energy-pathwise}
 \begin{aligned}
 \int_t^T|V(s)|^2\,ds
 \le{}&C(R^*)^2+4R^*\sum_j\overline U_j(r_j)
       +4\left|\int_t^T R(s)V(s)\,dW(s)\right|.
 \end{aligned}
\end{equation}
Squaring and taking expectations yields
\begin{equation}\label{eq:reg-remaining-energy-before-young}
\fitdisplay{
\mathbb E\left(\int_t^T|V(s)|^2\,ds\right)^2
 \le C\mathbb E(R^*)^4
    +C\mathbb E\left[(R^*)^2
                      \left(\sum_j\overline U_j(r_j)\right)^2\right]
 +C\mathbb E\left|\int_t^T R(s)V(s)\,dW(s)\right|^2.
}
\end{equation}
Here the stochastic integral occurs through its second moment, so
It\^o isometry gives
\begin{equation}\label{eq:reg-remaining-martingale-second}
 \mathbb E\left|\int_t^T R(s)V(s)\,dW(s)\right|^2
 =\mathbb E\int_t^T|R(s)|^2|V(s)|^2\,ds\le\mathbb E\left[(R^*)^2\int_t^T|V(s)|^2\,ds\right].
\end{equation}
Young's inequality implies
\[
 C(R^*)^2\int_t^T|V(s)|^2\,ds
 \le \tfrac12\left(\int_t^T|V(s)|^2\,ds\right)^2+C'(R^*)^4.
\]
Substituting into \eqref{eq:reg-remaining-energy-before-young} and
absorbing the first term on the right, we obtain
\[
 \mathbb E\left(\int_t^T|V(s)|^2\,ds\right)^2
 \le C_1\mathbb E(R^*)^4
    +C_2\mathbb E\left[(R^*)^2
                    \left(\sum_j\overline U_j(r_j)\right)^2\right].
\]
For every $\varepsilon>0$, 
\[
 C_2(R^*)^2\left(\sum_j\overline U_j(r_j)\right)^2
 \le \varepsilon\left(\sum_j\overline U_j(r_j)\right)^4
       +\frac{C_2^2}{4\varepsilon}(R^*)^4,
\]
therefore gives
\begin{equation}\label{eq:reg-remaining-energy-versus-sum}
 \mathbb E\left(\int_t^T|V(s)|^2\,ds\right)^2
 \le \varepsilon\mathbb E\left(\sum_j\overline U_j(r_j)\right)^4
       +C_\varepsilon\mathbb E(R^*)^4.
\end{equation}
The constants $C_1,C_2,C_\varepsilon$ depend only on $T,L$ and,
for $C_\varepsilon$, on $\varepsilon$.

Insert \eqref{eq:reg-remaining-energy-versus-sum} into
\eqref{eq:reg-remaining-sum-versus-energy} and choose
$\varepsilon=(2C_0)^{-1}$. Then
\[
 \mathbb E\left(\sum_j\overline U_j(r_j)\right)^4
 \le (C_0+C_0C_\varepsilon)\mathbb E(R^*)^4
       +\tfrac12\mathbb E\left(\sum_j\overline U_j(r_j)\right)^4.
\]
All quantities are finite by \eqref{eq:reg-remaining-fixed-N}, so the
last term can be absorbed. Substituting the resulting bound back into
\eqref{eq:reg-remaining-energy-versus-sum}, we conclude that
\begin{equation}\label{eq:reg-remaining-estimate}
 \mathbb E\left(\sum_j\overline U_j(r_j)\right)^4
  +\mathbb E\left(\int_t^T|V(s)|^2\,ds\right)^2
 \le C\mathbb E\sup_{t\le s\le T}|R(s)|^4.
\end{equation}
Finally, comparison, \eqref{eq:reg-remaining-domination}, and
\eqref{eq:reg-same-time-estimate} imply
\[
 \begin{aligned}
 \mathbb E\left(\sum_j|U_j(r_j)|\right)^4
 &\le\mathbb E\left(\sum_j\overline U_j(r_j)\right)^4\\
 &\le C\mathbb E(R^*)^4
 \le C\mathbb E\sup_{t\le s\le T}
                         \left|\sum_j\overline U_j(s)\right|^4\\
 &\le C\mathbb E\left(\sum_j|\xi_j|
                     +\int_t^T\sum_j|q_j(s)|\,ds\right)^4.
 \end{aligned}
\]
This proves \eqref{eq:reg-finite-sum-bound} for bounded data.

For general case, let $h_m(x)=(-m)\vee(x\wedge m)$ and let
$(U_j^m,V_j^m)$ solve \eqref{eq:reg-linear-family} with terminal
value $h_m(\xi_j)$ and source $h_m(q_j)$. For two truncation
levels $m,m'$, denote the corresponding differences by
$\delta U_j,\delta V_j,\delta\xi_j,\delta q_j$.
The supremum estimate in
\cite[Theorem~4.1 and Remark~4.1]{BouchardPossamaiTanZhou2018},
combined with It\^o's formula and Young's inequality, gives
\begin{equation}\label{eq:reg-truncated-stability}
 \mathbb E\left[
     \sup_{r_j\le s\le T}|\delta U_j(s)|^4
    +\left(\int_{r_j}^T|\delta V_j(s)|^2\,ds\right)^2\right]
 \le C\mathbb E\left[
     |\delta\xi_j|^4
    +\left(\int_{r_j}^T|\delta q_j(s)|\,ds\right)^4\right],
\end{equation}
with $C$ independent of the truncation levels and $N$.
By \eqref{eq:reg-finite-sum-data} and dominated convergence,
the truncated data are Cauchy in the quantities on the right.
Consequently, the truncated solutions converge in the norms
on the left, and their limit coincides with $(U_j,V_j)$
by passage to the limit in the BSDE and uniqueness.
In particular, for fixed $N$,
\[
 \sum_{j=1}^N|U_j^m(r_j)|
 \longrightarrow
 \sum_{j=1}^N|U_j(r_j)|
 \quad\text{in }L^4.
\]
Since $|h_m(x)|\le |x|$, dominated convergence permits passage
to the limit in the bounded-data estimate, whose constant is
independent of both $m$ and $N$. This proves
\eqref{eq:reg-finite-sum-bound}.
\end{proof}

We now obtain estimates for the variation processes that are uniform
in the auxiliary partition.

\begin{lemma}\label{lem:reg-uniform-variations}
Under Assumptions~\ref{asmsde}, \ref{asmbsde}, and~\ref{convasm}, the
processes in Lemma~\ref{lem:reg-first-variations} satisfy
\begin{equation}\label{eq:reg-uniform-estimate}
 \mathbb E\left[
    \left(\sum_{j=1}^N|\partial_jY^N(r_j)|\right)^4
   +\left(\int_t^T|\nabla Z^N(s)|^2\,ds\right)^2\right]\le C.
\end{equation}
The constant is independent of $N$, $\gamma_t$, and the second-derivative
bounds in Assumption~\ref{convasm}.
\end{lemma}

\begin{proof}
Define $S_N=\sum_j|\partial_jY^N(r_j)|$ and
$Q_N=\int_t^T|\nabla Z^N(s)|^2\,ds$.
Their fourth and second moments, respectively, are finite for each
fixed $N$ by Lemma~\ref{lem:reg-first-variations}.
Equations~\eqref{eq:reg-memory-source},
\eqref{eq:reg-z-representation}, and~\eqref{eq:reg-derivative-sum} imply
\begin{equation}\label{eq:reg-source-sum-bound}
\fitdisplay{
\begin{aligned}
 \sum_j|q_j^N(s)|
 &\le \sum_j|f_j^N(s)|
      +|\nabla Y^N(s)|\sum_j|b_j^N(s)|+
    \bigl(|\nabla Z^N(s)|+|\beta(s)|\,|\nabla Y^N(s)|\bigr)
    \sum_j|\sigma_j^N(s)|\\
 &\le C(1+|\nabla Z^N(s)|).
 \end{aligned}
}
\end{equation}
Applying Lemma~\ref{lem:reg-finite-sum} to
\eqref{eq:reg-memory-bsde} gives
\begin{equation}\label{eq:reg-sum-versus-energy}
\fitdisplay{
\mathbb E S_N^4
 \le C\mathbb E\left(
           \sum_j|g_j^N|+\int_t^T\sum_j|q_j^N(s)|\,ds\right)^4
 \le C\mathbb E\left(1+\int_t^T|\nabla Z^N(s)|\,ds\right)^4
 \le C(1+\mathbb E Q_N^2).
}
\end{equation}

Conversely, applying It\^o's formula to $|\nabla Y^N|^2$ in
\eqref{eq:reg-live-global} yields
\begin{equation}\label{eq:reg-variation-square-identity}
 \begin{aligned}
 Q_N={}&-|\nabla Y^N(t)|^2
       +2\int_t^T\alpha(s)|\nabla Y^N(s)|^2\,ds\\
 &+2\int_t^T
       \beta(s)\nabla Y^N(s)\nabla Z^N(s)\,ds
       -2M_N\\
 &+2\sum_j\nabla Y^N(r_j-)\partial_jY^N(r_j)
       -\sum_j|\partial_jY^N(r_j)|^2,
 \end{aligned}
\end{equation}
where
$M_N=\int_t^T
\nabla Y^N(s)\nabla Z^N(s)\,dW(s)$.
We observe that
\begin{equation}\label{eq:reg-variation-energy-pathwise}
 Q_N\le \tfrac12Q_N+C+CS_N+2|M_N|,
 \qquad\text{and hence}\qquad Q_N\le C(1+S_N+|M_N|).
\end{equation}
After squaring and taking expectations, It\^o isometry gives
\begin{equation}\label{eq:reg-variation-energy-square}
 \begin{aligned}
 \mathbb E Q_N^2
 &\le C(1+\mathbb E S_N^2)+C\mathbb E|M_N|^2\\
 &=C(1+\mathbb E S_N^2)
   +C\mathbb E\int_t^T
       |\nabla Y^N(s)\nabla Z^N(s)|^2\,ds\\
 &\le C(1+\mathbb E S_N^2)+C\mathbb E Q_N.
 \end{aligned}
\end{equation}
Using $Cx\le \tfrac12x^2+C'$ to absorb the last term, we obtain
\begin{equation}\label{eq:reg-energy-versus-sum}
 \mathbb E Q_N^2\le C(1+\mathbb E S_N^2).
\end{equation}
Combining \eqref{eq:reg-sum-versus-energy} and
\eqref{eq:reg-energy-versus-sum}, and then using
$Cx^2\le\tfrac12x^4+C'$ for $x\ge0$, gives
$\mathbb E S_N^4\le C+\tfrac12\mathbb E S_N^4$.
It follows that $\mathbb E S_N^4\le C$, and
\eqref{eq:reg-energy-versus-sum} then gives $\mathbb E Q_N^2\le C$.
\end{proof}

The following proposition establishes regularity estimates for the
original solution, which constitute a key ingredient in the proof
of the main theorem.

\begin{proposition}\label{prop:reg-projected-regularity}
Let Assumptions~\ref{asmsde}, \ref{asmbsde}, and~\ref{convasm} hold.
For $\pi=\{t=t_0<t_1<\cdots<t_n=T\}$, put $h_i=t_{i+1}-t_i$ and
define
\begin{equation}\label{eq:reg-hat-definition}
 \widehat Z^{\gamma_t,\pi}(t_i)
 =\frac1{h_i}\mathbb E\left[
     \left.\int_{t_i}^{t_{i+1}}Z^{\gamma_t}(s)\,ds
                          \right|\mathcal F_{t_i}\right],
 \qquad i=0,\ldots,n-1.
\end{equation}
There exists a constant $C>0$ such that, for every $t<T$,
$\gamma\in D([0,T],\mathbb R^d)$, and partition $\pi$ of $[t,T]$,
\begin{equation}\label{eq:reg-final-y}
 \max_{0\le i\le n-1}\mathbb E\left[
      \sup_{t_i\le s\le t_{i+1}}
      |Y^{\gamma_t}(s)-Y^{\gamma_t}(t_i)|^2\right]
 \le C(1+\|\gamma_t\|_\infty^2)|\pi|,
\end{equation}
and
\begin{equation}\label{eq:reg-final-z}
 \sum_{i=0}^{n-1}\mathbb E\left[\int_{t_i}^{t_{i+1}}
    |Z^{\gamma_t}(s)-\widehat Z^{\gamma_t,\pi}(t_i)|^2\,ds\right]
 \le C(1+\|\gamma_t\|_\infty^2)|\pi|.
\end{equation}
Moreover,
\begin{equation}\label{eq:reg-z-growth}
 |Z^{\gamma_t}(s)|\le C\|\sigma(s,X_s^{\gamma_t})\|
       \le C(1+\|X_s^{\gamma_t}\|_\infty)
 \quad ds\otimes d\mathbb P\text{-a.e.}
\end{equation}
The constant is independent of $t$, $\gamma$, $\pi$, and the
second-derivative bounds in Assumption~\ref{convasm}.
\end{proposition}

\begin{proof}
First, \eqref{eq:reg-z-growth} is immediately followed by Lemma \ref{lem:reg-first-variations}.
For $s\in[t_i,t_{i+1}]$, the original BSDE gives
\begin{equation}\label{eq:reg-Y-increment}
 Y(s)-Y(t_i)=-\int_{t_i}^s f(r,X_r,Y(r),Z(r))\,dr
                         +\int_{t_i}^s Z(r)\,dW(r).
\end{equation}
Since $X^{\gamma_t}\in \mathbb{S}^2(0,T;\mathbb{R})$, $Y^{\gamma_t}\in \mathbb{S}^2(t,T;\mathbb{R})$ and \eqref{eq:reg-z-growth}, \eqref{eq:reg-final-y} follows immediately.

We prove \eqref{eq:reg-final-z} by the error decomposition used in
\cite[Theorem~5.2.4]{BSDE}. We observe that
\begin{equation}\label{eq:reg-original-diffusion-estimates}
 \begin{aligned}
 \mathbb E\sup_{t\le s\le T}|\sigma(s,X_s)|^4
   &\le C(1+\|\gamma_t\|_\infty^4),\\
 \mathbb E|\sigma(s,X_s)-\sigma(u,X_u)|^2
   &\le C(1+\|\gamma_t\|_\infty^2)(s-u),
       \qquad t\le u\le s\le T.
 \end{aligned}
\end{equation}

Define $\eta_i^N=\nabla Y^N(t_i)\sigma(t_i,X_{t_i})$.
This variable belongs to $L^2(\mathcal F_{t_i};\mathbb R)$.
By \eqref{eq:reg-z-representation}, for almost every
$s\in[t_i,t_{i+1}]$,
\begin{equation}\label{eq:reg-error-splitting}
 \begin{aligned}
 Z^N(s)-\eta_i^N
 ={}&\nabla Y^N(s)[\sigma^N(s)-\sigma(s,X_s)]\\
   &+\nabla Y^N(s)[\sigma(s,X_s)-\sigma(t_i,X_{t_i})]\\
   &+[\nabla Y^N(s)-\nabla Y^N(t_i)]\sigma(t_i,X_{t_i}).
 \end{aligned}
\end{equation}
By \eqref{eq:reg-z-representation} and
\eqref{eq:reg-original-diffusion-estimates},
\begin{equation}\label{eq:reg-first-two-errors}
 \begin{aligned}
 &\sum_{i=0}^{n-1}\mathbb E\int_{t_i}^{t_{i+1}}
   \Bigl(
     \bigl|\nabla Y^N(s)
       [\sigma^N(s)-\sigma(s,X_s)]\bigr|^2
     +\bigl|\nabla Y^N(s)
       [\sigma(s,X_s)-\sigma(t_i,X_{t_i})]\bigr|^2
   \Bigr)\,ds\\
 &\quad\le
   C\mathbb E\int_t^T
     |\sigma^N(s)-\sigma(s,X_s)|^2\,ds +C(1+\|\gamma_t\|_\infty^2)|\pi|.
 \end{aligned}
\end{equation}
For the last term in \eqref{eq:reg-error-splitting}, we use
\eqref{eq:reg-live-global}. Denote
$Q_N=\int_t^T|\nabla Z^N(r)|^2\,dr$ and
$S_N=\sum_j|\partial_jY^N(r_j)|$ as in
Lemma~\ref{lem:reg-uniform-variations} and Cauchy--Schwarz imply
\begin{equation}\label{eq:reg-weighted-estimates}
 \mathbb E\left[\sup_{t\le r\le T}|\sigma(r,X_r)|^2
                    (1+Q_N+S_N^2)\right]
 \le C(1+\|\gamma_t\|_\infty^2).
\end{equation}
The drift part is bounded by
\begin{equation}\label{eq:reg-drift-error}
 \begin{aligned}
 &\sum_i\mathbb E\int_{t_i}^{t_{i+1}}
   |\sigma(t_i,X_{t_i})|^2
   \left|\int_{t_i}^s
      [\alpha\nabla Y^N+\beta\nabla Z^N](r)\,dr
   \right|^2ds\\
 &\qquad\le C|\pi|^2\mathbb E\left[\sup_r|\sigma(r,X_r)|^2(1+Q_N)\right]
 \le C(1+\|\gamma_t\|_\infty^2)|\pi|^2.
 \end{aligned}
\end{equation}
For the martingale part, we have
\begin{equation}\label{eq:reg-martingale-error}
\fitdisplay{
\begin{aligned}
 &\sum_i\mathbb E\int_{t_i}^{t_{i+1}}
  |\sigma(t_i,X_{t_i})|^2
  \left|\int_{t_i}^s\nabla Z^N(r)\,dW(r)\right|^2ds\\
 &=\sum_i\mathbb E\int_{t_i}^{t_{i+1}}
  |\sigma(t_i,X_{t_i})|^2
                  \int_{t_i}^s|\nabla Z^N(r)|^2\,dr\,ds\le |\pi|\mathbb E\left[\sup_r|\sigma(r,X_r)|^2Q_N\right]
 \le C(1+\|\gamma_t\|_\infty^2)|\pi|.
 \end{aligned}
}
\end{equation}
For the jump part, it follows that
\begin{equation}\label{eq:reg-jump-error}
 \begin{aligned}
 &\sum_i\mathbb E\int_{t_i}^{t_{i+1}}
   |\sigma(t_i,X_{t_i})|^2
    \left|\sum_{j:t_i<r_j\le s}\partial_jY^N(r_j)\right|^2ds\\
 &\le |\pi|\mathbb E\left[
       \sup_r|\sigma(r,X_r)|^2\sum_i\left(\sum_{j:t_i<r_j\le t_{i+1}}
                              |\partial_jY^N(r_j)|\right)^2\right]\\
 &\le |\pi|\mathbb E\left[\sup_r|\sigma(r,X_r)|^2S_N^2\right]
 \le C(1+\|\gamma_t\|_\infty^2)|\pi|.
 \end{aligned}
\end{equation}
Combining \eqref{eq:reg-first-two-errors}--\eqref{eq:reg-jump-error}, we obtain
\begin{equation}\label{eq:reg-auxiliary-error}
 \sum_i\mathbb E\int_{t_i}^{t_{i+1}}|Z^N(s)-\eta_i^N|^2\,ds
 \le C(1+\|\gamma_t\|_\infty^2)|\pi|
 +C\mathbb E\int_t^T|\sigma^N(s)-\sigma(s,X_s)|^2\,ds.
\end{equation}
By the minimizing property of the conditional average,
\begin{equation}\label{eq:reg-projection-limit}
\fitdisplay{
\begin{aligned}
 \sum_i\mathbb E\int_{t_i}^{t_{i+1}}
       |Z(s)-\widehat Z^{\gamma_t,\pi}(t_i)|^2\,ds&\le\sum_i\mathbb E\int_{t_i}^{t_{i+1}}|Z(s)-\eta_i^N|^2\,ds\\
 &\le2\mathbb E\int_t^T|Z(s)-Z^N(s)|^2\,ds
       +2\sum_i\mathbb E\int_{t_i}^{t_{i+1}}|Z^N(s)-\eta_i^N|^2\,ds.
 \end{aligned}
}
\end{equation}
Letting $N\to\infty$ proves \eqref{eq:reg-final-z}.
\end{proof}

For coefficients of the FBSDE satisfying Assumptions \ref{asmsde}, \ref{asmbsde}, and \ref{convasm}, 
we verify that
the Picard-type iteration 
$(Y^{\gamma_{t},\pi,m},Z^{\gamma_{t},\pi,m})$
converges  
to the solution $(Y^{\gamma_{t}},Z^{\gamma_{t}})$ of the BSDE \eqref{eqn:pbsde}. Moreover, we show that the convergence rate
is of the order $ |\pi|\log\frac{1}{|\pi|}+(\frac{1}{2}+|\pi|)^m$ 
as $|\pi|\to0,m\to\infty.$  
We define a pair of processes $(Y^{\gamma_t,\pi,\infty},Z^{\gamma_t,\pi,\infty})$  as follows. For $t_i\in \pi,$  
\begin{equation}
	\begin{aligned}
		&Y^{\gamma_t,\pi,\infty}(t_n)=g(X_T^{\gamma_t,\pi})\\
		&Z_{k}^{\gamma_t,\pi,\infty}(t_i)=\mathbb{E}\bigg[\frac{\Delta W_i^k}{t_{i+1}-t_i} Y^{\gamma_t,\pi,\infty}(t_{i+1}) \bigg|\, \mathcal{F}_{t_i}\,\bigg]\,,\;k=1,2,\cdots,\ell,
	\end{aligned}
\end{equation}
and let $ Y^{\gamma_t,\pi,\infty}(t_{i})$ be a solution to 
\begin{equation}
\fitdisplay{
\begin{aligned}   
		Y^{\gamma_t,\pi,\infty}(t_{i})=\mathbb{E}\Big[Y^{\gamma_t,\pi,\infty}(t_{i+1})\Big|\, \mathcal{F}_{t_i}\,\Big]+f(t_i,X^{\gamma_t,\pi},Y^{\gamma_t,\pi,\infty}(t_{i}),Z^{\gamma_t,\pi,\infty}(t_{i}))(t_{i+1}-t_i)\,.
	\end{aligned}
}
\end{equation}
This solution exists because the mapping  $y\mapsto f(t,\omega,y,z)|\pi|$ has a Lipschitz constant $K_1|\pi|$, which is less than $1$ when $|\pi|$ is sufficiently small.
We extend  $(Y^{\gamma_t,\pi,\infty},Z^{\gamma_t,\pi,\infty})$ to RCLL processes by piecewise constant interpolation.
The following proposition is the   main result of this section.

\begin{proposition} \label{convscheme}
	Let Assumptions \ref{asmsde}, \ref{asmbsde},  \ref{convasm} hold. Then for any $t\in[0,T]$ and $\gamma \in D([0,T],\mathbb{R}^{d})$, there exists a constant $C>0$ such that 
	\begin{equation}
		\begin{aligned}
			&\max_{0\le i\le n-1} \mathbb{E}\bigg[\sup_{t_i\le s \le t_{i+1}} |Y^{\gamma_{t}}(s)-Y^{\gamma_{t},\pi,m}(t_i)|^2\bigg]+\sum_{i=0}^{n-1} \mathbb{E}\bigg[\int_{t_i}^{t_{i+1}} | Z^{\gamma_{t}}(s)-Z^{\gamma_{t},\pi,m}(t_i)|^2\,ds\bigg]\\
			& \le  C \Big(|\pi|\log\frac{1}{|\pi|}+\Big(\frac{1}{2}+C|\pi|\Big)^m\Big)
		\end{aligned}       
	\end{equation}
	provided $|\pi|$ is  sufficiently small.
\end{proposition}
\begin{proof}
	For notational convenience, we set $\ell=1$ and write $h_i := t_{i+1}-t_i$ for $i=0,\cdots,n-1.$ By Proposition \ref{prop:reg-projected-regularity}, we have
	\begin{equation}\label{eqn:ZZ}
		\sum_{i=0}^{n-1}\mathbb{E}\bigg[\displaystyle\int_{t_i}^{t_{i+1}} | Z^{\gamma_t}(s)-\hat{Z}^{\gamma_t,\pi}(t_i)|^2\, ds\bigg] \le C(1+\lVert \gamma_t \rVert_{\infty}^2)|\pi|\,.
	\end{equation} 
	We now prove that
	for any $t\in[0,T]$ and $\gamma \in D([0,T],\mathbb{R}^{d})$,  there is a constant $C>0$ such that  
	\begin{equation}\label{eqn:m_infty}
		\begin{aligned}
			&\max_{0\le i\le n-1} \mathbb{E}\bigg[\sup_{t_i\le u \le t_{i+1}} |Y^{\gamma_{t}}(u)-Y^{\gamma_t,\pi,\infty}(t_i)|^2\bigg]+\sum_{i=0}^{n-1} \mathbb{E}\bigg[\int_{t_i}^{t_{i+1}} | Z^{\gamma_t}(u)-Z^{\gamma_t,\pi,\infty}(t_i)|^2\,du\bigg]\\
			&\le C(1+\lVert{\gamma_t}\rVert_{\infty}^2)|\pi|\log\frac{1}{|\pi|} 
		\end{aligned}
	\end{equation}	 for sufficiently small $|\pi|$.	
	Define $\triangle Y(t_i):=Y^{\gamma_t,\pi,\infty}(t_i)-Y^{\gamma_t}(t_i)$ and $\triangle Z (t_i):=Z^{\gamma_t,\pi\infty}(t_i)-\hat{Z}^{\gamma_t,\pi}(t_i)$ for all $i=0,\cdots,n.$ According to the martingale representation theorem, there exists a process $\bar{Z}^{\gamma_t,\pi,\infty}$ such that 
	\begin{equation}
		Y^{\gamma_t,\pi,\infty}(t_{i+1})=\mathbb{E}\bigg[Y^{\gamma_t,\pi,\infty}(t_{i+1}) \bigg| \mathcal{F}_{t_i}\bigg]+ \int_{t_i}^{t_{i+1}} \bar{Z}^{\gamma_t,\pi,\,\infty}(s)\, dW(s).
	\end{equation}
	From the definition of $Y^{\gamma_t,\pi\,\infty}(t_i)$, 
	\begingroup\small
\begin{equation}
\fitdisplay{
\begin{aligned}
			\triangle Y(t_i)&=\triangle Y(t_{i+1})+\int_{t_i}^{t_{i+1}} (f(t_i,X^{\gamma_t,\pi},Y^{\gamma_t,\pi,\infty}(t_i),Z^{\gamma_t,\pi,\infty}(t_i)-f(s,X^{\gamma_t},Y^{\gamma_t}(s),Z^{\gamma_t}(s))))\,ds\\
			&-\int_{t_i}^{t_{i+1}} \triangle \bar{Z}(s)\, dW(s) 
		\end{aligned}
}
\end{equation}
\endgroup
	where $\triangle \bar{Z}(s):= \bar{Z}^{\gamma_t,\pi,\infty}(s)-Z^{\gamma_t}(s)$. Observe that 
	\begin{equation}
		\begin{aligned}
			Z^{\gamma_t,\pi,\infty}(t_i)&=\frac{1}{h_i}\mathbb{E}\bigg[(W(t_{i+1})-W(t_i))Y^{\gamma_t,\pi,\infty}(t_{i+1})\bigg|\mathcal{F}_{t_i}\bigg]\\
			&=\frac{1}{h_i}\mathbb{E}\bigg[(W(t_{i+1})-W(t_i))\int_{t_i}^{t_{i+1}} \bar{Z} ^{\gamma_t,\pi,\infty}(s)\,dW(s)
			\bigg|\mathcal{F}_{t_i}\bigg]\\
			&=\frac{1}{h_i} \mathbb{E}\bigg[\int_{t_i}^{t_{i+1}} \bar{Z} ^{\gamma_t,\pi,\infty}(s)\,ds \bigg|\mathcal{F}_{t_i}\bigg].
		\end{aligned}
	\end{equation}
	Then, by applying H$\ddot{\text{o}}$lder's inequality and Jensen's inequality, we obtain
	\begin{equation}
		\begin{aligned}
			\mathbb{E}\Big[|\triangle Z(t_i)|^2\Big]&=\mathbb{E}\Big[|Z^{\gamma_t,\pi\infty}(t_i)-\hat{Z}^{\gamma_t,\pi}(t_i)|^2\Big]\\
			&=\mathbb{E}\bigg[\left| \frac{1}{h_i} \mathbb{E}\bigg[\int_{t_i}^{t_{i+1}} \bar{Z} ^{\gamma_t,\pi,\infty}(s)\,ds \bigg|\mathcal{F}_{t_i}\bigg]-\frac{1}{h_i}E\bigg[\int_{t_i}^{t_{i+1}} Z^{\gamma_t}(s)\, ds \bigg| \mathcal{F}_{t_i}\bigg] \right|^2\bigg]\\ &\le \frac{1}{h_i} \mathbb{E}\bigg[ \int_{t_i}^{t_{i+1}}|\triangle \bar{Z}(s)|^2\, ds\bigg].
		\end{aligned}
	\end{equation}

	Now, define $I_s:=f(t_i,X^{\gamma_t,\pi},Y^{\gamma_t,\pi,\infty}(t_i),Z^{\gamma_t,\pi,\infty}(t_i))-f(s,X^{\gamma_t},Y^{\gamma_t}(s),Z^{\gamma_t}(s))$ for $i=0,\cdots,n-1$ and $t_i\le s <t_{i+1}$. From Assumption \ref{asmbsde}, 
	\begin{equation}
		\begin{aligned}
			| I_s |^2 &\le C(| s-t_i |+ \lVert X_s^{\gamma_t}-X_{t_i}^{\gamma_t,\pi}\rVert_{\infty}^2+| Y^{\gamma_t}(s)-Y^{\gamma_t,\pi,\infty}(t_i)|^2+| Z^{\gamma_t}(s)-Z^{\gamma_t,\pi,\infty}(t_i)|^2) \\
			&\le C(| s-t_i |+ \lVert X_s^{\gamma_t}-X_{t_i}^{\gamma_t} \rVert_{\infty}^2 +\lVert X_{t_i}^{\gamma_t}-X_{t_i}^{\gamma_t,\pi}\rVert_{\infty}^2+| Y^{\gamma_t}(s)-Y^{\gamma_t}(t_i)|^2+|\triangle Y(t_i)|^2\\
			&\quad+| Z^{\gamma_t}(s)-\hat{Z}^{\gamma_t,\pi}(t_i)|^2+|\triangle Z(t_i)|^2)\,.
		\end{aligned}
	\end{equation}
	Then, for all $\epsilon>0$,
	\begin{equation}
\fitdisplay{
\begin{aligned}
			&\quad \mathbb{E}\Big[|\triangle Y(t_i)|^2+h_i|\triangle Z(t_i)|^2\Big]
			\le \mathbb{E}\bigg[|\triangle Y(t_i)|^2+\int_{t_i}^{t_{i+1}}|\triangle \bar{Z}(s)|^2\,ds\bigg]\\
			&\le \mathbb{E}\bigg[\bigg(\triangle Y(t_{i+1})+\int_{t_i}^{t_{i+1}} I_s \,ds\bigg)^2\bigg]
			\le \bigg(1+\frac{h_i}{\epsilon}\bigg)\mathbb{E}\bigg[|\triangle Y(t_{i+1})|^2\bigg]+(h_i+\epsilon)\mathbb{E}\bigg[\int_{t_i}^{t_{i+1}} |I_s|^2\, ds\bigg]\\
			&\le \bigg(1+\frac{h_i}{\epsilon}\bigg)\mathbb{E}\bigg[|\triangle Y(t_{i+1})|^2\bigg]+C(h_i+\epsilon)h_i \mathbb{E}\bigg[| \triangle Y(t_i) |^2+| \triangle Z(t_i) |^2\bigg]\\
			&\quad+C(h_i+\epsilon)\mathbb{E}\bigg[\int_{t_i}^{t_{i+1}} | Z^{\gamma_t}(s)-\hat{Z}^{\gamma_t,\pi}(t_i)|^2 \,ds\bigg] + C(1+\lVert \gamma_t \rVert_{\infty}^2)h_i\big(|\pi|+|\pi|\log\frac{1}{|\pi|}\big).
		\end{aligned}
}
\end{equation}
	Let $|\pi|$ and $\epsilon$ be sufficiently small such that 
	\begin{equation}\label{eqn:gron}
\fitdisplay{
\begin{aligned}
			&\mathbb{E}\Big[|\triangle Y(t_i)|^2+\frac{h_i}{2}|\triangle Z(t_i)|^2\Big]\\
			\le\; & (1+Ch_i)\mathbb{E}\bigg[|\triangle Y(t_{i+1})|^2\bigg]+ C\mathbb{E}\bigg[\int_{t_i}^{t_{i+1}} | Z^{\gamma_t}(s)-\hat{Z}^{\gamma_t,\pi}(t_i)|^2 \,ds\bigg]+C(1+\lVert \gamma_t \rVert_{\infty}^2)h_i|\pi|\log\frac{1}{|\pi|}.    
		\end{aligned}
}
\end{equation}
	Using the backward discrete Gr\"onwall inequality and \eqref{eqn:ZZ}, we obtain 
	\begin{equation}
\fitdisplay{
\begin{aligned}
			&\quad \max_{0\le i \le n} \mathbb{E}\Big[|\triangle Y(t_i)|^2\Big]\\
			& \le C\mathbb{E}\bigg[\lVert X_T^{\gamma_t}-X_T^{\gamma_t,\pi}\rVert_{\infty}^2\bigg]+ C\sum_{i=0}^{n-1}\mathbb{E}\bigg[\int_{t_i}^{t_{i+1}} | Z^{\gamma_t}(s)-\hat{Z}^{\gamma_t,\pi}(t_i)|^2 \,ds\bigg]+C(1+\lVert \gamma_t \rVert_{\infty}^2)|\pi|\log\frac{1}{|\pi|}\\
			& \le  C(1+\lVert{\gamma_t}\rVert_{\infty}^2)|\pi|\log\frac{1}{|\pi|}.
		\end{aligned}
}
\end{equation}
	Summing over $i=0,\cdots,n-1$ in \eqref{eqn:gron}, we obtain
	\begin{equation}
		\begin{aligned}
			\sum_{i=0}^{n-1}\frac{h_i}{2}\mathbb{E}\bigg[|\triangle Z(t_i)|^2\bigg]
			&\le(1+Ch_n)\mathbb{E}\bigg[|\triangle Y(t_{n})|^2\bigg]
			+\sum_{i=0}^{n-1}Ch_i\mathbb{E}\bigg[|\triangle Y(t_{i})|^2\bigg]\\
			&\quad+C\sum_{i=0}^{n-1}\mathbb{E}\bigg[\int_{t_i}^{t_{i+1}}
			|Z^{\gamma_t}(s)-\hat{Z}^{\gamma_t,\pi}(t_i)|^2\,ds\bigg]\\
			&\quad+C(1+\lVert \gamma_t \rVert_{\infty}^2)|\pi|\log\frac{1}{|\pi|}\\
			&\le C(1+\lVert{\gamma_t}\rVert_{\infty}^2)|\pi|\log\frac{1}{|\pi|}.
		\end{aligned}
	\end{equation}
	Combining the last two inequalities with \eqref{eq:reg-final-y} and  \eqref{eqn:ZZ}, we conclude \eqref{eqn:m_infty}.

	On the other hand, observe that 
	\begin{equation}
\fitdisplay{
\max_{0\le i \le n}\mathbb{E}\bigg[
			|Y^{\gamma_t,\pi,\infty}(t_i)-Y^{\gamma_t,\pi,m}(t_i)|^2\bigg]
			+\sum_{i=0}^{n-1}\mathbb{E}\bigg[
			|Z^{\gamma_t,\pi,\infty}(t_i)-Z^{\gamma_t,\pi,m}(t_i)|^2\bigg]h_i
			\le C\Big(\frac{1}{2}+C |\pi|\Big)^m.
}
\end{equation}
	provided that $|\pi|$ is sufficiently small.
	This inequality can be established by an argument similar to that in Bender and Denk~\cite[Theorem 5]{forwardbsde}. 
	Combining this estimate with \eqref{eqn:m_infty}, we obtain the desired result.		
\end{proof}

\subsection{Main results}

In this section, we prove Theorem~\ref{thm:main} under
Assumptions~\ref{asmsde} and~\ref{asmbsde}.
More precisely, we show that the Picard-type approximation
$(Y^{\gamma_t,\pi,m},Z^{\gamma_t,\pi,m})_{m\in\mathbb{N}\cup\{0\}}$
converges to the FBSDE solution $(Y^{\gamma_t},Z^{\gamma_t})$, with an
error bound of order
\[
    |\pi|\log\frac{1}{|\pi|}
    +\left(\frac{1}{2}+C|\pi|\right)^m
\]
as $|\pi|\to0$ and $m\to\infty$.
This extends the convergence result of Proposition~\ref{convscheme}
to a setting with weaker assumptions.

The proof relies on smooth approximations of the coefficients and
the terminal functional. We construct
$b^{\epsilon,\hat\pi,\pi}$, $\sigma^{\epsilon,\hat\pi,\pi}$,
$f^{\epsilon,\hat\pi,\pi}$ and $g^{\epsilon,\hat\pi,\pi}$ that approximate
$b,\sigma,f,g$ and satisfy Assumptions~\ref{asmsde},
\ref{asmbsde} and~\ref{convasm}.
The construction consists of a Lipschitz extension, a piecewise-constant
approximation of stopped paths, and a finite-dimensional mollification.
We describe the construction for the driver $f$.

For $(s,\omega,y,z)\in[0,T]\times D([0,T],\mathbb{R}^d)
\times\mathbb{R}\times\mathbb{R}^{\ell}$, we define the
McShane--Whitney Lipschitz extension of $f$
(see~\cite{McShane1934}) by
\begin{equation}\label{eqn:smooth_extension_f}
\bar f(s,\omega,y,z)
:=\inf_{(r,\eta)\in\Lambda_T}
\left\{
    f(r,\eta,y,z)
    +K_1\left(\sqrt{|s-r|}+\|\omega-\eta\|_{\infty}\right)
\right\}.
\end{equation}
For every $(r,\eta)\in\Lambda_T$, Assumption~\ref{asmbsde} gives
\[
 f(r,\eta,y,z)
 \ge f(s,\bar0,y,z)
      -K_1\bigl(\sqrt{|s-r|}+\|\eta\|_{\infty}\bigr).
\]
Adding the penalty term in \eqref{eqn:smooth_extension_f} and using
$\|\omega-\eta\|_{\infty}-\|\eta\|_{\infty}
\ge-\|\omega\|_{\infty}$ shows that every expression under the
infimum is bounded below by $f(s,\bar0,y,z)-K_1\|\omega\|_{\infty}$.
Choosing $(r,\eta)=(s,\bar0)$ gives the corresponding upper bound,
and hence
\[
 f(s,\bar0,y,z)-K_1\|\omega\|_{\infty}
 \le\bar f(s,\omega,y,z)
 \le f(s,\bar0,y,z)+K_1\|\omega\|_{\infty},
\]
so that the extension is finite-valued.
The same argument, with $(s,\omega_s)$ in place of $(s,\bar0)$,
shows that
\[
 f(r,\eta,y,z)
 +K_1\bigl(\sqrt{|s-r|}+\|\omega_s-\eta\|_{\infty}\bigr)
 \ge f(s,\omega,y,z)
\]
for every $(r,\eta)\in\Lambda_T$.
Since equality holds at $(r,\eta)=(s,\omega_s)$, the extension
coincides with the original driver on stopped paths:
\begin{equation}\label{eqn:smooth_extension_agreement}
    \bar f(s,\omega_s,y,z)=f(s,\omega,y,z).
\end{equation}
For any $(r,\eta)\in\Lambda_T$, the Lipschitz continuity of $f$
in $(y,z)$ and the triangle inequalities give
\[
\begin{aligned}
 &f(r,\eta,y,z)
   +K_1\bigl(\sqrt{|s-r|}+\|\omega-\eta\|_{\infty}\bigr)\\
 &\le
   f(r,\eta,y',z')
   +K_1\bigl(\sqrt{|s'-r|}+\|\omega'-\eta\|_{\infty}\bigr)\\
 &
   +K_1\bigl(
      \sqrt{|s-s'|}+\|\omega-\omega'\|_{\infty}
      +|y-y'|+|z-z'|
   \bigr).
\end{aligned}
\]
The last term is independent of $(r,\eta)$, so taking the infimum
on each side over $\Lambda_T$ yields
\[
 \bar f(s,\omega,y,z)
 \le \bar f(s',\omega',y',z')
 +K_1\bigl(
      \sqrt{|s-s'|}+\|\omega-\omega'\|_{\infty}
      +|y-y'|+|z-z'|
   \bigr).
\]
Interchanging the inputs gives the reverse bound, and therefore
\begin{equation}\label{eqn:smooth_extension_lipschitz}
 |\bar f(s,\omega,y,z)-\bar f(s',\omega',y',z')|\le K_1\left(
    \sqrt{|s-s'|}+\|\omega-\omega'\|_{\infty}
    +|y-y'|+|z-z'|
 \right).
\end{equation}
Together with \eqref{eqn:smooth_extension_agreement}, this shows that
$\bar f$ extends the driver from stopped paths to the full path space,
with the same Lipschitz constant $K_1$ for the metric on the
right-hand side. In particular, the path arguments are compared in
the full supremum norm, without stopping them at their respective
evaluation times.

We next introduce the piecewise-constant path approximation. Let
\[
\hat\pi=\{0=\hat t_0<\cdots<\hat t_{\hat n}=t\},
\qquad
\pi=\{t=t_0<\cdots<t_n=T\}.
\]
For $\omega\in D([0,T],\mathbb{R}^d)$, set
\begin{equation}\label{eqn:smooth_piecewise_path}
\omega^{\hat\pi,\pi}(u)
:= \sum_{i=0}^{\hat n-1}
       \omega(\hat t_i)\mathbf{1}_{[\hat t_i,\hat t_{i+1})}(u)
 +\sum_{i=0}^{n-1}
       \omega(t_i)\mathbf{1}_{[t_i,t_{i+1})}(u)
   +\omega(t_n)\mathbf{1}_{\{t_n\}}(u),
   \qquad u\in[0,T].
\end{equation}
Applying this approximation after stopping the input path, we define
\begin{equation}\label{eqn:smooth_discrete_f}
    f^{\hat\pi,\pi}(s,\omega,y,z)
    :=\bar f\bigl(s,(\omega_s)^{\hat\pi,\pi},y,z\bigr).
\end{equation}

To express this functional in finite-dimensional form, write
\[
    \hat x=(\hat x_0,\ldots,\hat x_{\hat n-1}),
    \qquad x=(x_0,\ldots,x_n).
\]
Since $\hat t_{\hat n}=t_0$, we use $x_0$ for the coordinate at $t$
and omit the redundant coordinate $\hat x_{\hat n}$.
Define
\[
\tilde f^{\hat\pi,\pi}:
[0,T]\times(\mathbb{R}^d)^{\hat n}\times(\mathbb{R}^d)^{n+1}
\times\mathbb{R}\times\mathbb{R}^{\ell}\longrightarrow\mathbb{R}
\]
by
\begin{equation}\label{eqn:smooth_cylindrical_f}
 \tilde f^{\hat\pi,\pi}(s,\hat x,x,y,z)
 :=\bar f\Bigg(
 s,\,
 \sum_{i=0}^{\hat n-1}
      \hat x_i\mathbf{1}_{[\hat t_i,\hat t_{i+1})}
 +\sum_{i=0}^{n-1}x_i\mathbf{1}_{[t_i,t_{i+1})}
 +x_n\mathbf{1}_{\{t_n\}},\,y,z
 \Bigg).
\end{equation}
Then
\[
f^{\hat\pi,\pi}(s,\omega,y,z)
=
\tilde f^{\hat\pi,\pi}\Bigl(
    s,\,(\omega(s\wedge\hat t_i))_{i=0}^{\hat n-1},
    (\omega(s\wedge t_i))_{i=0}^{n},\,y,z
\Bigr).
\]

We regularize this finite-dimensional function by convolution in the
spatial variables. For each $q\in\{1,d,\ell\}$, let
$\upsilon^{(q)}\in C_c^{\infty}(\mathbb{R}^q)$ be a nonnegative
mollifier supported in the unit ball, with integral one, and set
\[
    \upsilon_{\epsilon}^{(q)}(v)
    :=\epsilon^{-q}\upsilon^{(q)}(v/\epsilon),
    \qquad 0<\epsilon\le1.
\]
Write $\hat\xi=(\hat\xi_0,\ldots,\hat\xi_{\hat n-1})$ and
$\xi=(\xi_0,\ldots,\xi_{n+2})$, with
\[
    \hat\xi_i,\xi_j\in\mathbb{R}^d
    \quad(0\le i<\hat n,\ 0\le j\le n),
    \qquad \xi_{n+1}\in\mathbb{R},\qquad
    \xi_{n+2}\in\mathbb{R}^{\ell}.
\]
Introduce the product probability measure
\begin{equation}\label{eqn:smooth_mollifier}
 dA^{\epsilon}(\hat\xi,\xi)
 :=
 \left(\prod_{i=0}^{\hat n-1}
       \upsilon_{\epsilon}^{(d)}(\hat\xi_i)\,d\hat\xi_i\right)
 \left(\prod_{i=0}^{n}
       \upsilon_{\epsilon}^{(d)}(\xi_i)\,d\xi_i\right)
 \upsilon_{\epsilon}^{(1)}(\xi_{n+1})\,d\xi_{n+1}
       \upsilon_{\epsilon}^{(\ell)}(\xi_{n+2})\,d\xi_{n+2},
\end{equation}
and define
\begin{equation}\label{eqn:smooth_convolution}
 \tilde f^{\epsilon,\hat\pi,\pi}(s,\hat x,x,y,z)
 :=\int
 \tilde f^{\hat\pi,\pi}\Bigl(
    s,\,\hat x-\hat\xi,\,(x_i-\xi_i)_{i=0}^{n},\,
    y-\xi_{n+1},\,z-\xi_{n+2}
 \Bigr)\,dA^{\epsilon}(\hat\xi,\xi).
\end{equation}
We then define the path-dependent approximation by
\begin{equation}\label{eqn:smooth_final_f}
 f^{\epsilon,\hat\pi,\pi}(s,\omega,y,z):=
 \tilde f^{\epsilon,\hat\pi,\pi}\Bigl(
    s,\,(\omega(s\wedge\hat t_i))_{i=0}^{\hat n-1},
    (\omega(s\wedge t_i))_{i=0}^{n},\,y,z
 \Bigr).
\end{equation}
For later reference, define the perturbation path
\[
\xi^{\hat\pi,\pi}
:=\sum_{i=0}^{\hat n-1}
      \hat\xi_i\mathbf{1}_{[\hat t_i,\hat t_{i+1})}
  +\sum_{i=0}^{n-1}\xi_i\mathbf{1}_{[t_i,t_{i+1})}
  +\xi_n\mathbf{1}_{\{t_n\}}.
\]
With this notation, \eqref{eqn:smooth_final_f} can be written as
\begin{equation}\label{eqn:smooth_integral_f}
 f^{\epsilon,\hat\pi,\pi}(s,\omega,y,z)=\int
 \bar f\Bigl(
    s,(\omega_s)^{\hat\pi,\pi}-\xi^{\hat\pi,\pi},
    y-\xi_{n+1},z-\xi_{n+2}
 \Bigr)\,dA^{\epsilon}(\hat\xi,\xi).
\end{equation}
The approximations $b^{\epsilon,\hat\pi,\pi}$,
$\sigma^{\epsilon,\hat\pi,\pi}$ and $g^{\epsilon,\hat\pi,\pi}$ are
constructed analogously.

To state the approximation result, we introduce the class of paths
that agree with the prescribed history and are piecewise constant on
$[t,T]$. For fixed $\gamma,\pi,t$ and $T$, set
\begin{equation}
\begin{aligned}
 &D^{\gamma_t,\pi}([0,T],\mathbb{R}^d)
 :=\bigl\{\omega\in D([0,T],\mathbb{R}^d):\;\\
 &\qquad\omega(s)=\gamma(s),\quad 0\le s\le t,\\
 &\qquad\omega(s)=\omega(t_i),\quad
       t_i\le s<t_{i+1},\quad i=0,1,\ldots,n-1\bigr\}.
\end{aligned}
\end{equation}
The following lemma summarizes the regularity and approximation
properties of these functionals. Here, $K_1$ and $K_2$ denote the
constants in Assumptions~\ref{asmsde} and~\ref{asmbsde}.

\begin{lemma}\label{smoothapprox}
Let Assumptions~\ref{asmsde} and~\ref{asmbsde} hold.
For every $0<\epsilon\le1$ and every pair of partitions
$\hat\pi,\pi$, the functionals
$b^{\epsilon,\hat\pi,\pi}$, $\sigma^{\epsilon,\hat\pi,\pi}$,
$f^{\epsilon,\hat\pi,\pi}$ and $g^{\epsilon,\hat\pi,\pi}$ constructed
above satisfy Assumptions~\ref{asmsde}, \ref{asmbsde}
and~\ref{convasm}.
The Lipschitz and growth constants in
Assumptions~\ref{asmsde} and~\ref{asmbsde} may be enlarged, but remain
independent of $\epsilon,\hat\pi,\pi$ and depend only on $K_1,K_2$,
the fixed dimensions and the choice of norms.

Moreover, there exists a constant $C>0$, depending only on the fixed
dimensions and the choice of norms, such that, if
\[
    \|\gamma_t-\gamma_t^{\hat\pi}\|_{\infty}<\sqrt{|\pi|},
\]
then
\begin{equation}\label{eqn:smooth_approximation_error}
\begin{aligned}
 &|b(s,\omega)-b^{\epsilon,\hat\pi,\pi}(s,\omega)|
  +\|\sigma(s,\omega)-\sigma^{\epsilon,\hat\pi,\pi}(s,\omega)\|\\
 &\quad+|g(\omega)-g^{\epsilon,\hat\pi,\pi}(\omega)|
  +|f(s,\omega,y,z)-f^{\epsilon,\hat\pi,\pi}(s,\omega,y,z)|\\
 &\qquad\le C K_1\bigl(\epsilon+\sqrt{|\pi|}\bigr)
\end{aligned}
\end{equation}
for all $s\in[0,T]$,
$\omega\in D^{\gamma_t,\pi}([0,T],\mathbb{R}^d)$,
$y\in\mathbb{R}$ and $z\in\mathbb{R}^{\ell}$.
\end{lemma}

\begin{proof}
We give the proof for $f^{\epsilon,\hat{\pi},\pi}$; the other
functionals are treated in the same way.

The representation \eqref{eqn:smooth_final_f} shows that
$f^{\epsilon,\hat\pi,\pi}$ depends on $\omega$ only through the values
$\omega(s\wedge\hat t_i)$ and $\omega(s\wedge t_i)$ and is therefore
non-anticipative. The piecewise-constant approximation also satisfies
\begin{equation}\label{eqn:smooth_projection_contraction}
 \bigl\|(\omega_s)^{\hat\pi,\pi}
       -(\tilde\omega_{\tilde s})^{\hat\pi,\pi}\bigr\|_{\infty}
 \le\|\omega_s-\tilde\omega_{\tilde s}\|_{\infty}.
\end{equation}
Using the same perturbation $\xi^{\hat\pi,\pi}$ in both integrands
of \eqref{eqn:smooth_integral_f}, we therefore obtain
\begin{equation}\label{eqn:smooth_f_lipschitz}
\begin{aligned}
 &|f^{\epsilon,\hat\pi,\pi}(s,\omega,y,z)
   -f^{\epsilon,\hat\pi,\pi}
      (\tilde s,\tilde\omega,\tilde y,\tilde z)|\\
 &\quad\le K_1\Bigl(
    \sqrt{|s-\tilde s|}
    +\bigl\|(\omega_s)^{\hat\pi,\pi}
            -(\tilde\omega_{\tilde s})^{\hat\pi,\pi}\bigr\|_{\infty}
    +|y-\tilde y|+|z-\tilde z|
 \Bigr)\\
 &\quad\le K_1\Bigl(
    \sqrt{|s-\tilde s|}
    +\|\omega_s-\tilde\omega_{\tilde s}\|_{\infty}
    +|y-\tilde y|+|z-\tilde z|
 \Bigr).
\end{aligned}
\end{equation}

On the support of $A^{\epsilon}$,
\begin{equation}\label{eqn:smooth_support}
 \|\xi^{\hat\pi,\pi}\|_{\infty}\le\epsilon,
 \qquad |\xi_{n+1}|\le\epsilon,
 \qquad |\xi_{n+2}|\le\epsilon.
\end{equation}
Thus \eqref{eqn:smooth_extension_lipschitz} yields
\begin{equation}\label{eqn:smooth_mollification_error_f}
 \bigl|f^{\epsilon,\hat\pi,\pi}(s,\omega,y,z)
       -\bar f\bigl(s,(\omega_s)^{\hat\pi,\pi},y,z\bigr)\bigr|
 \le3K_1\epsilon.
\end{equation}
At the origin, the preceding estimate gives
\[
 \sup_{s\in[0,T]}
 |f^{\epsilon,\hat\pi,\pi}(s,\bar0,0,0)|
 \le K_2+3K_1\epsilon.
\]
Together with \eqref{eqn:smooth_f_lipschitz}, this verifies the
conditions on $f^{\epsilon,\hat\pi,\pi}$ in
Assumption~\ref{asmbsde}.

We next verify the regularity conditions in Assumption~\ref{convasm}.
For each fixed $s\in[0,T]$, standard properties of mollification imply
that $\tilde f^{\epsilon,\hat\pi,\pi}(s,\cdot)$ is $C^\infty$ in its
finite-dimensional spatial variables. The required derivative bounds
are obtained from the global Lipschitz estimate
\eqref{eqn:smooth_extension_lipschitz}.

Let $v$ denote the vector of spatial variables and let $a_\epsilon$
be the density of the product mollifier in \eqref{eqn:smooth_mollifier}.
For every nonzero multi-index $\alpha$, we have
\[
 \partial^\alpha
 \tilde f^{\epsilon,\hat\pi,\pi}(s,v)
 =
 \int
 \bigl(
   \tilde f^{\hat\pi,\pi}(s,v-w)
   -\tilde f^{\hat\pi,\pi}(s,v)
 \bigr)
 \partial^\alpha a_\epsilon(w)\,dw.
\]
The Lipschitz continuity of $\tilde f^{\hat\pi,\pi}$ and the scaling
of the mollifier then yield
\begin{equation}\label{eqn:smooth_derivative_bounds}
 \sup_{s,v}
 \left|
 \partial^\alpha
 \tilde f^{\epsilon,\hat\pi,\pi}(s,v)
 \right|
 \le
 C_{\alpha,\hat n,n,d,\ell}
 K_1\epsilon^{1-|\alpha|}.
\end{equation}

For each fixed $s$, the map
\[
 \omega\longmapsto
 \Bigl((\omega(s\wedge\hat t_i))_{i=0}^{\hat n-1},
       (\omega(s\wedge t_i))_{i=0}^{n}\Bigr)
\]
is bounded and linear with respect to the supremum norm. The chain
rule therefore implies that $f^{\epsilon,\hat\pi,\pi}$ is jointly
Fr\'echet smooth in $(\omega,y,z)$. In particular, for
$h\in D([0,T],\mathbb{R}^d)$,
\begin{equation}\label{eqn:smooth_frechet_derivative}
\begin{aligned}
 D_{\omega}f^{\epsilon,\hat\pi,\pi}(s,\omega,y,z)[h]
 ={}&\sum_{i=0}^{\hat n-1}
       \nabla_{\hat x_i}\tilde f^{\epsilon,\hat\pi,\pi}
       \cdot h(s\wedge\hat t_i)\\
 &+\sum_{i=0}^{n}
       \nabla_{x_i}\tilde f^{\epsilon,\hat\pi,\pi}
       \cdot h(s\wedge t_i),
\end{aligned}
\end{equation}
where all derivatives on the right-hand side are evaluated at the
arguments in \eqref{eqn:smooth_final_f}.
Since convolution preserves the Lipschitz bound in the maximum norm,
\[
 \sum_{i=0}^{\hat n-1}
     |\nabla_{\hat x_i}\tilde f^{\epsilon,\hat\pi,\pi}|
 +\sum_{i=0}^{n}
     |\nabla_{x_i}\tilde f^{\epsilon,\hat\pi,\pi}|
 \le K_1.
\]
The derivatives with respect to $y$ and $z$ are likewise bounded by
$K_1$. Applying the chain rule to the second- and third-order bounds
in \eqref{eqn:smooth_derivative_bounds} shows that the second
Fr\'echet derivative is bounded and uniformly Lipschitz in
$(\omega,y,z)$. These higher-order bounds are uniform in the
evaluation time, but may depend on $\epsilon$ and the partitions.

Finally, we establish the approximation estimate
\eqref{eqn:smooth_approximation_error}. Set
\[
    \delta:=\|\gamma_t-\gamma_t^{\hat\pi}\|_{\infty}.
\]
For $\omega\in D^{\gamma_t,\pi}([0,T],\mathbb{R}^d)$, the path
$\omega$ is already piecewise constant along $\pi$ on $[t,T]$, so the
discretization error arises only from approximating the prescribed
history on $[0,t]$. Applying the piecewise-constant approximation
after stopping the path therefore gives, uniformly in $s\in[0,T]$,
\[
    \bigl\|(\omega_s)^{\hat\pi,\pi}-\omega_s\bigr\|_{\infty}
    \le
    \|\gamma_t^{\hat\pi}-\gamma_t\|_{\infty}
    =\delta.
\]
Combining this estimate with
\eqref{eqn:smooth_extension_agreement},
\eqref{eqn:smooth_extension_lipschitz} and
\eqref{eqn:smooth_mollification_error_f}, we obtain
\begin{equation}\label{eqn:smooth_f_error}
 |f^{\epsilon,\hat\pi,\pi}(s,\omega,y,z)
      -f(s,\omega,y,z)|\le
 3K_1\epsilon
 +K_1\bigl\|(\omega_s)^{\hat\pi,\pi}-\omega_s\bigr\|_{\infty}\le
 3K_1\epsilon+K_1\delta.
\end{equation}
\end{proof}

We denote $b^{\epsilon,\hat{\pi},\pi}, \sigma^{\epsilon,\hat{\pi},\pi}, f^{\epsilon,\hat{\pi},\pi}, g^{\epsilon,\hat{\pi},\pi}$ simply by $b^{\epsilon},\sigma^{\epsilon},f^{\epsilon}, g^{\epsilon}$, respectively. We consider a system of SDE 
\begin{equation} 
	\begin{aligned} \label{epsilsde}
		& X^{\gamma_{t},\epsilon}(s)=\gamma_{t}(t)+\int_{t}^{s} b^{\epsilon}(u,X_{u}^{\gamma_{t},\epsilon})\,du+ \int_{t}^{s} \sigma^{\epsilon}(u,X_{u}^{\gamma_{t},\epsilon})\,dW(u)\,, \;  t\le s \le T \,,\\
		& X^{\gamma_{t},\epsilon}(s)=\gamma_{t}(s)\,, \;  0 \le s \le t		 
	\end{aligned}
\end{equation}
and   BSDE
\begin{equation} \label{epsilbsde}
	\begin{aligned}
	Y^{\gamma_{t},\epsilon}(s)&=g^{\epsilon}(X_{T}^{\gamma_{t},\epsilon})\\
	&\quad+\int_{s}^{T} f^{\epsilon}(u,X_{u}^{\gamma_{t},\epsilon},Y^{\gamma_{t},\epsilon}(u),Z^{\gamma_{t},\epsilon}(u))\,du\\
	&\quad-\int_{s}^{T} Z^{\gamma_{t},\epsilon}(u)\,dW(u)\,, \;  t\le s\le T\,.
	\end{aligned}
\end{equation}

\begin{proposition} \label{epnbsde}
	Let Assumptions \ref{asmsde} and \ref{asmbsde} hold. Then, for any partition $\pi$ of $[t,T]$, there exists a constant $C>0$ such that
	\begin{equation}\label{eqn:xepsesti}
		\max_{0\le i\le n} \mathbb{E}\bigg[ \sup_{t_{i} \le s \le t_{i+1}} | X^{\gamma_{t}}(s)-X^{\gamma_{t},\epsilon}(s)|^2\bigg]\le C\bigg((1+\lVert \gamma_t \rVert_{\infty}^2)(|\pi|+|\pi|\log\frac{1}{|\pi|})+\epsilon^2\bigg) 
	\end{equation} and
	\begin{equation}\label{eqn:yzepsesti}
		\begin{aligned}
		&\sup_{t\le s \le T}\mathbb{E}\bigg[ |Y^{\gamma_{t}}(s)-Y^{\gamma_{t},\epsilon}(s)|^2\bigg]\\
		&\quad+\sum_{i=0}^{n-1}\mathbb{E}\bigg[\int_{t_i}^{t_{i+1}}|Z^{\gamma_{t}}(s)-Z^{\gamma_{t},\epsilon}(s)|^2\, ds \bigg]\\
		&\qquad\le C\bigg((1+\lVert \gamma_t \rVert_{\infty}^2)(|\pi|+|\pi|\log\frac{1}{|\pi|})+\epsilon^2\bigg)
		\end{aligned}
	\end{equation}
	for all $\epsilon>0$, partition $\hat{\pi}$ of $[0,t]$ and $\gamma\in D([0,T],\mathbb{R}^d)$ with $\lVert \gamma_t - \gamma_t^{\hat{\pi}}\rVert_{\infty}<\sqrt{|\pi|}$. 
\end{proposition}
\begin{proof}
	We restrict our attention to proving \eqref{eqn:yzepsesti}.	
	In this proof, $C$ is a generic constant  possibly varying from line to line.
	Let $X^{\gamma_{t},\pi,\epsilon}$ be the Euler scheme for the SDE \eqref{epsilsde} and let $(Y^{\gamma_t,\pi,\epsilon},Z^{\gamma_t,\pi,\epsilon})$ be a solution to \eqref{epnbsde} with  $X^{\gamma_t,\epsilon}$ replaced with $X^{\gamma_t,\pi,\epsilon}$. We first show that there is a constant $C>0$ such that
	\begin{equation}
		\begin{aligned}
			&\sup_{t\le s \le T}\mathbb{E}\bigg[ |Y^{\gamma_{t},\pi}(s)-Y^{\gamma_{t},\pi,\epsilon}(s)|^2\bigg]+\sum_{i=0}^{n-1}\mathbb{E}\bigg[\int_{t_i}^{t_{i+1}}|Z^{\gamma_{t},\pi}(s)-Z^{\gamma_{t},\pi,\epsilon}(s)|^2\, ds \bigg]\,\\
			&\le C((1+\lVert \gamma_t \rVert_{\infty}^2)|\pi|+\epsilon^2)
		\end{aligned}
	\end{equation}
	for any partition $\hat{\pi}$ of $[0,t]$ and $\gamma\in D([0,T],\mathbb{R}^d)$ with $\lVert \gamma_t - \gamma_t^{\hat{\pi}}\rVert_{\infty}<\sqrt{|\pi|}$. 
	Once this is proven, by combining this inequality with \eqref{eqn:bsde}, we obtain \eqref{eqn:yzepsesti}.

	Define $\triangle Y=Y^{\gamma_{t},\pi}-Y^{\gamma_{t},\pi,\epsilon}$ and $\triangle Z=Z^{\gamma_{t},\pi}-Z^{\gamma_{t},\pi,\epsilon}$. Then,
	\begin{equation}
		\begin{aligned}
			\triangle Y(s) &= g(X_T^{\gamma_{t},\pi})-g^\epsilon(X_T^{\gamma_{t},\pi,\epsilon})+\int_s^T (f(u,X_{u}^{\gamma_{t},\pi},Y^{\gamma_{t},\pi}(u),Z^{\gamma_{t},\pi}(u))\\& - f^{\epsilon}(u,X_{u}^{\gamma_{t},\pi,\,\epsilon},Y^{\gamma_{t},\pi,\epsilon}(u), Z^{\gamma_{t},\pi,\epsilon}(u)))\,du-\int_s^T \triangle Z(u)\,dW(u).
		\end{aligned}
	\end{equation}
	By Assumption \ref{asmbsde} and Lemma \ref{smoothapprox},
	\begin{equation}
		\begin{aligned} 
			|g(X_T^{\gamma_{t},\pi})-g^\epsilon(X_T^{\gamma_{t},\pi,\,\epsilon})|
			&\leq |g(X_T^{\gamma_{t},\pi})-g(X_T^{\gamma_{t},\pi,\epsilon})|+|g(X_T^{\gamma_{t},\pi,\epsilon})-g^\epsilon(X_T^{\gamma_{t},\pi,\epsilon})| \\
			&\le K_1 \max_{0\le i\le n}|X^{\gamma_{t},\pi}(t_i)-X^{\gamma_{t},\pi,\,\epsilon}(t_i)|+C(\epsilon+\sqrt{|\pi|})\, .
		\end{aligned}
	\end{equation}
	Moreover, 
	\begin{equation}
		\begin{aligned}
			&\quad f(u,X_{u}^{\gamma_{t},\pi},Y^{\gamma_{t},\pi}(u),Z^{\gamma_{t},\pi}(u)) - f^{\epsilon}(u,X_{u}^{\gamma_{t},\pi,\epsilon},Y^{\gamma_{t},\pi,\epsilon}(u),Z^{\gamma_{t},\pi,\epsilon}(u))\\
			&=f(u,X_{u}^{\gamma_{t},\pi},Y^{\gamma_{t},\pi,\,\epsilon}(u),Z^{\gamma_{t},\pi,\epsilon}(u)) - f^{\epsilon}(u,X_{u}^{\gamma_{t},\pi,\epsilon},Y^{\gamma_{t},\pi,\epsilon}(u),Z^{\gamma_{t},\pi,\,\epsilon}(u))\\
			&\quad +\alpha_u\triangle Y(u)+\beta_u\triangle Z(u)\,,\;t\le u \le T 
		\end{aligned}
	\end{equation}
	for some processes  $(\alpha_u)_{t\le u\le T}$ and $(\beta_u)_{t\le u\le T}$, which are bounded by $K_1$. It follows that
	\begin{equation}
		\begin{aligned} 
			&\quad	|f(u,X_{u}^{\gamma_{t},\pi},Y^{\gamma_{t},\pi,\epsilon}(u),Z^{\gamma_{t},\pi,\epsilon}(u)) - f^{\epsilon}(u,X_{u}^{\gamma_{t},\pi,\epsilon},Y^{\gamma_{t},\pi,\epsilon}(u),Z^{\gamma_{t},\pi,\epsilon}(u))| \\
			&\le C(\epsilon+\sqrt{|\pi|})+K_1 \max_{0\le i\le n}|X^{\gamma_{t},\pi}(t_i)-X^{\gamma_{t},\pi,\epsilon}(t_i)|\,.		
		\end{aligned}
	\end{equation}
	The desired result follows from standard BSDE arguments, completing the proof.
\end{proof}

We now state the convergence estimation of the scheme  in \eqref{scheme} and 
provide an error bound.
We highlight that  Assumption \ref{convasm} is not necessary for the convergence. 
The following theorem is one of main results of this paper.
\begin{theorem} \label{thm:main}
	Let Assumptions \ref{asmsde} and \ref{asmbsde} hold. Then there exists a constant $C>0$ such that
	\begin{equation}
		\begin{aligned}
			&\max_{0\le i\le n-1} \mathbb{E}\bigg[\sup_{t_i\le s \le t_{i+1}} |Y^{\gamma_{t}}(s)-Y^{\gamma_t,\pi,m}(t_i)|^2\bigg]+\sum_{i=0}^{n-1} \mathbb{E}\bigg[\int_{t_i}^{t_{i+1}} | Z^{\gamma_t}(s)-Z^{\gamma_t,\pi,m}(t_i)|^2\,ds\bigg]\\
			& \le C \Big(|\pi|\log\frac{1}{|\pi|}+\Big(\frac{1}{2}+C|\pi|\Big)^m\Big)
		\end{aligned}        
	\end{equation}
	provided  $| \pi |$ is sufficiently small.
\end{theorem}

\begin{proof}
Throughout the proof, $C$ denotes a positive constant that may
vary from line to line. Following the argument in the proof of
Proposition~\ref{convscheme}, it suffices to establish the analogues
of \eqref{eq:reg-final-y} and \eqref{eq:reg-final-z}
under Assumptions~\ref{asmsde} and~\ref{asmbsde}, with the factor
$|\pi|$ on the right-hand sides replaced by
$|\pi|\log\frac{1}{|\pi|}$.
We assume that $|\pi|$ is sufficiently small so that
$|\pi|+|\pi|\log(1/|\pi|)\le 2|\pi|\log(1/|\pi|)$.
    
	For the given coefficients $b,\sigma, g, f$ satisfying Assumptions \ref{asmsde} and \ref{asmbsde}, we consider their approximations $b^{\epsilon,\hat{\pi},\pi},$ $ \sigma^{\epsilon,\hat{\pi},\pi},$ $ g^{\epsilon,\hat{\pi},\pi},$ $ f^{\epsilon,\hat{\pi},\pi}$
	given in Lemma \ref{smoothapprox},
	which we denote simply by  $b^{\epsilon},\sigma^{\epsilon},f^{\epsilon}, g^{\epsilon}$ as in \eqref{epsilsde} and \eqref{epsilbsde}.  Set $\epsilon=\sqrt{|\pi|}$. Then, \eqref{eq:reg-final-y} is obtained directly from
	\begin{equation}
		\begin{aligned}
	    \mathbb{E}\bigg[\int_{t_i}^{t_{i+1}} |Z^{\gamma_t}(s)|^2\,ds \bigg]&\le \mathbb{E}\bigg[\int_{t_i}^{t_{i+1}}2|Z^{\gamma_{t}}(s)-Z^{\gamma_{t},\epsilon}(s)|^2+2|Z^{\gamma_{t},\epsilon}(s)|^2\, ds \bigg] \\ &\le C(1+\lVert \gamma_t \rVert_{\infty}^2)\big(|\pi|+|\pi|\log\frac{1}{|\pi|}\big) \,.
		\end{aligned}		
	\end{equation}
	
	We now prove \eqref{eqn:ZZ} without Assumption \ref{convasm}.	  Define $$\hat{Z}^{\gamma_t,\pi}(t_i):=\dfrac{1}{h_i} \mathbb{E}\bigg[\displaystyle\int_{t_i}^{t_{i+1}} Z^{\gamma_t}(s)\, ds \bigg| \mathcal{F}_{t_i}\bigg]\,,\;\hat{Z}^{\gamma_t,\pi,\epsilon}(t_i):=\dfrac{1}{h_i} \mathbb{E}\bigg[\displaystyle\int_{t_i}^{t_{i+1}} Z^{\gamma_t,\epsilon}(s)\, ds \bigg| \mathcal{F}_{t_i}\bigg]\,.$$ 
	Then
	\begin{equation}
		\begin{aligned}
			&\quad\sum_{i=0}^{n-1} \mathbb{E}\bigg[\int_{t_i}^{t_{i+1}} | \hat{Z}^{\gamma_t,\pi,\epsilon}(t_i)-\hat{Z}^{\gamma_t,\pi}(t_i)|^2\,ds\bigg]\\
			&\le  \sum_{i=0}^{n-1} \dfrac{1}{h_i} \mathbb{E}\bigg[ \left| \mathbb{E}\bigg[\int_{t_i}^{t_{i+1}} Z^{\gamma_t}(s)-Z^{\gamma_t,\epsilon}(s) \,ds \bigg| \mathcal{F}_{t_i}\bigg] \right|^2\bigg] \\
			&\le \sum_{i=0}^{n-1} \mathbb{E}\bigg[\int_{t_i}^{t_{i+1}} |Z^{\gamma_t}(s)-Z^{\gamma_t,\epsilon}(s)|^2\, ds\bigg]\\
			&= \mathbb{E}\bigg[\int_{t}^{T} |Z^{\gamma_t}(s)-Z^{\gamma_t,\epsilon}(s)|^2\, ds\bigg]\\
			&\le C(1+\lVert{\gamma_t}\rVert_{\infty}^2)\big(|\pi|+|\pi|\log\frac{1}{|\pi|}\big)  
		\end{aligned} 	
	\end{equation}
	for some constant $C>0$ by  Proposition \ref{epnbsde}.	Since the approximations $b^{\epsilon,\hat{\pi},\pi},$ $ \sigma^{\epsilon,\hat{\pi},\pi},$ $ g^{\epsilon,\hat{\pi},\pi},$ $ f^{\epsilon,\hat{\pi},\pi}$ satisfy Assumption \ref{convasm},	by Proposition \ref{convscheme}, we have
	$$\sum_{i=0}^{n-1} \mathbb{E}\bigg[\int_{t_i}^{t_{i+1}} | Z^{\gamma_t,\epsilon}(s)-\hat{Z}^{\gamma_t,\pi,\epsilon}(t_i)|^2\,ds\bigg]\le C(1+\lVert{\gamma_t}\rVert_{\infty}^2)|\pi|\log\frac{1}{|\pi|}$$ for some constant $C>0.$
	Therefore,
	\begin{equation}
		\begin{aligned}
			&\quad\sum_{i=0}^{n-1} \mathbb{E}\bigg[\int_{t_i}^{t_{i+1}} | Z^{\gamma_t}(s)-\hat{Z}^{\gamma_t,\pi}(t_i)|^2\,ds\bigg]\\
			& \le C\bigg(\sum_{i=0}^{n-1} \mathbb{E}\bigg[\int_{t_i}^{t_{i+1}} | Z^{\gamma_t}(s)-Z^{\gamma_t,\epsilon}(s)|^2\,ds\bigg]+\sum_{i=0}^{n-1} \mathbb{E}\bigg[\int_{t_i}^{t_{i+1}} | Z^{\gamma_t,\epsilon}(s)-\hat{Z}^{\gamma_t,\pi,\epsilon}(t_i)|^2\,ds\bigg]\\
			&\quad+\sum_{i=0}^{n-1} \mathbb{E}\bigg[\int_{t_i}^{t_{i+1}} | \hat{Z}^{\gamma_t,\pi,\epsilon}(t_i)-\hat{Z}^{\gamma_t,\pi}(t_i)|^2\,ds\bigg]\bigg)\\
            &\le C(1+\lVert{\gamma_t}\rVert_{\infty}^2)|\pi|\log\frac{1}{|\pi|}.
		\end{aligned}
	\end{equation}
	This completes the proof.	
\end{proof}

\begin{remark}
In the Markovian setting, it was shown in \cite{forwardbsde} that the
convergence error is of order
\[
C_1\bigg(
|\pi|
+
\Big(\frac{1}{2}+C_2|\pi|\Big)^m
\bigg).
\]
In the path-dependent setting, however, an additional term of order $|\pi|\log\frac{1}{|\pi|}$ appears. Indeed, consider the case
\[
t=0,\qquad \gamma(0)=0,\qquad T=1,\qquad
b=0,\qquad \sigma=1,\qquad f=0,
\]
with
\[
g(\omega)
=
\sup_{0\le s\le T}
\left|\omega(s)-\omega(s-)\right|,
\]
and let $\pi$ be a uniform partition. These coefficients satisfy
Assumptions~\ref{asmsde} and~\ref{asmbsde}, and the corresponding FBSDE
admits the solution
\[
X^{\gamma_t}=W,
\qquad
Y^{\gamma_t}=0,
\qquad
Z^{\gamma_t}=0
\]
since $g(W)=0$ almost
surely. On the other hand, for the piecewise-constant Euler approximation
$X^\pi$,
\[
g(X^\pi)
=
\max_i
\left|W_{t_{i+1}}-W_{t_i}\right|.
\]
Since $f=0$, we have 
$Y^{\gamma_t,\pi,m}(0)
=
\mathbb{E}\left[g(X^\pi)\right]
$
for every $m\ge 1$. Therefore,
\begin{equation}
    \left|Y^{\gamma_t,\pi,m}(0)-Y^{\gamma_t}(0)
\right|^2=
\mathbb{E}\left[
g(X^\pi)-g(X)
\right]
^2=
\mathbb{E}\left[
\max_i
\left|W_{t_{i+1}}-W_{t_i}\right|
\right]
^2 \ge
c|\pi|\log\frac{1}{|\pi|}
\end{equation}
for some constant $c>0$ and all sufficiently small $|\pi|$ by \cite[section 2.1]{maxgaussian}. This shows
that the term $|\pi|\log\frac{1}{|\pi|}$
in the path-dependent convergence rate is, in general, necessary.
\end{remark}

\begin{remark}
	In the upper bound $$C_1\bigg(|\pi|\log\frac{1}{|\pi|}+\Big(\frac{1}{2}+C_2 |\pi|\Big)^m\bigg)$$ 
	appearing in \eqref{thm:main}, one may take
	$$C_2=4T(1+T)^2\ell^2K^4+K^3(1+T).$$
	This follows from the same argument as in Bender and Denk~\cite[Remark 6]{forwardbsde}.
\end{remark}

\section{Second-order parabolic path-dependent PDEs}
\label{sec:PPDE}
In this section, we establish a numerical scheme for PPDEs.
Consider a second-order parabolic PPDE
\begin{equation} \label{ppde}
	\begin{aligned}
		&\partial_{s}u+\frac{1}{2}\textnormal{Tr}(\sigma\sigma^{T}(s,\omega)\partial_{\omega\omega}^2 u) +b(s,\omega)\cdot\partial_{\omega} u +f(s,\omega,u,\partial_{\omega} u\sigma)=0 \,,\;0\le s<T\\
		&u(T,\omega)=g(\omega)
	\end{aligned}
\end{equation}
A solution to this PPDE is closely related to that of the FBSDE.
See Wang et al.~\cite{pathkac} and Bally et al.~\cite{functionalito} for proofs of the following proposition.
.

\begin{proposition} \label{uniclassical}
	Suppose Assumptions \ref{asmsde} and \ref{asmbsde} hold.   Denote by  $(Y^{\gamma_t},\,Z^{\gamma_t})$   the unique solution to the  FBSDE \eqref{eqn:pbsde}.
	
	\begin{enumerate}[(i)]
		\item If $u\in C_{loc}^{1,2}(\Lambda_T)$ is a solution to the PPDE \eqref{ppde} such that $u(t,\gamma)$ and $\nabla_{\omega}u(t,\gamma)$ have polynomial growth in $\gamma$ uniformly in $t$, then $(Y^{\gamma_t}(s),Z^{\gamma_t}(s))=(u(s,X_s^{\gamma_t}),\nabla_{\omega}u(s,X^{\gamma_t}_s)\sigma(s,X^{\gamma_t}_s)$ for all $s\ge t.$ 
		\item Suppose further that Assumption \ref{convasm} hold. Then, the PPDE has a unique classical solution $u$ \cite[Definition 2.5]{pathkac}. Moreover, $u\in C_{loc}^{1,2}(\Lambda_T)$ and  $u(t,\gamma)=Y^{\gamma_t}(t).$
	\end{enumerate} 
\end{proposition}

This proposition establishes that the unique solution $u$ of the PPDE \eqref{ppde} can be obtained via the FBSDE \eqref{eqn:pbsde}. In particular, the value $u(t,\gamma)$ can be computed numerically through our discretization scheme. A limitation of this approach is that it delivers the value of $u(t,\gamma)$ only for a given initial path $(t,\gamma)$; for distinct initial paths, the scheme must be repeated, thereby increasing the computational burden. To address this drawback, we construct a neural network within a supervised learning framework.

Let $L\ge2$ and $d_0,d_1,\ldots,d_L\in\mathbb N$, and let
$\varrho:\mathbb R\to(0,1)$ be the sigmoid activation function
\[
\varrho(x)=\frac{1}{1+e^{-x}}.
\]
For any $\ell=
1,\cdots,L,$ let $W_\ell: \mathbb{R}^{d_{\ell-1}}\to \mathbb{R}^{d_{\ell}}$
be an affine function. A function $\mathcal{N} : \mathbb{R}^{d_{0}} \to \mathbb{R}^{d_{L}}$
defined as
$$\mathcal{N}=W_L\circ F_{L-1}\circ \cdots \circ F_1 \textnormal{ with } F_\ell=\varrho\circ W_\ell \textnormal{ for } \ell=1,2,\cdots,L-1$$
is called a neural network. 
The activation function $\varrho$ is
applied componentwise. Here, 
$L$ denotes the number of layers, $d_1,\cdots, d_{L-1}$ denote the dimensions of the hidden layers and $d_0,d_L$ represent the dimensions of the input and output layers, respectively.
Each affine function   $W_\ell$
is of the form
$W_\ell(x) =
A^\ell x + b^\ell$
for some matrix $A^\ell\in \mathbb{R}^{d_\ell\times d_{\ell-1}}$ and vector $b
\in \mathbb{R}^{d_\ell} $. 
We let  $\theta=(A^1,\cdots,A^L,b^1,\cdots,b^L)$
denote the collection of parameters of the network. To emphasize the dependence on these parameters, we may write the neural network as   $\mathcal{N}_\theta$.
In the approximation result below, the hidden-layer widths may depend on the desired accuracy.

\begin{theorem} \label{thm:NN}
	Suppose  $u\in C_{loc}^{1,2}(\Lambda_T)$ and $K$ is a compact subset of $C([0,T],\mathbb{R}^d)$. For any $\epsilon>0$,  there are a  time partition
	$\{0=t_1<t_2<\ldots <t_n=T\}$ and a neutral network $\mathcal{N}_\theta: \mathbb{R}\times\mathbb{R}^{d\times n}\to\mathbb{R} $ such that $$|u(t,\gamma)-\mathcal{N}_\theta(t,\gamma(t_1),\cdots,\gamma(t_n))|\le \epsilon $$
	for all $(t,\gamma)\in [0,T] \times K.$
\end{theorem}

\begin{proof} We first construct the desired  partition of the interval $[0,T]$. 
	We recall that there is a constant $L>0$ such that 
	$|u(t,\gamma)-u(t',\gamma')|\leq L|\!|\gamma_t-\gamma_t'|\!|_\infty$ for all $(t,\gamma),(t',\gamma')\in \Lambda_T$.	
	Since $K$ is compact, it is bounded and equicontinuous. For given $\epsilon>0,$ there is a partition $\pi=	\{0=t_1<t_2<\ldots <t_n=T\}$ such that $|\!|\gamma-\gamma^\pi|\!|_{\infty}<\frac{\epsilon}{2L}$ for all $\gamma\in K$.
	where $\gamma^\pi$ is the piecewise linear path interpolating $(t_1,\gamma(t_1)),\cdots,(t_n,\gamma(t_n))$.
	
	We now construct the desired neural network. 	Let $R:=\sup_{f\in K}|\!|f|\!|_\infty.$ 
	For $x=(x_1,\cdots,x_n)\in \mathbb{R}^{d\times n}$, we denote as $\eta^x$ the piecewise linear path interpolating $(t_1,x_1),\cdots,(t_n,x_n)$. Consider the map
	$F:[0,T]\times \mathbb{R}^{d\times n}\to \mathbb{R}$
	given as $(t,x)\mapsto u(t,\eta^x).$
Since $F$ is continuous and $[0,T]\times[-R,R]^{d\times n}$ is compact,
the universal approximation theorem of Hornik et al.~\cite[Theorem~2.4]{universialapp}
yields a neural network $\mathcal N_\theta$
with the sigmoid activation function such that
\[
\sup_{(t,x)\in[0,T]\times[-R,R]^{d\times n}}
|F(t,x)-\mathcal N_\theta(t,x)|<\frac{\epsilon}{2}.
\]
For any $(t,\gamma) \in [0,T]\times K,$ let $x=(\gamma(t_1),\cdots,\gamma(t_n))$ then  $\gamma^\pi=\eta^x.$
	It follows that \begin{equation}
		\begin{aligned}
			&|u(t,\gamma)-\mathcal{N}_\theta(t,\gamma(t_1),\cdots,\gamma(t_n))|\\
			&\quad\le |u(t,\gamma)-u(t,\gamma^\pi)|
			+|u(t,\eta^x)-\mathcal{N}_\theta(t,x)|\\
			&\quad\le \frac{\epsilon}{2L}L+\frac{\epsilon}{2}=\epsilon\,.
		\end{aligned}
	\end{equation}	
	This completes the proof.
\end{proof}

One of the key ideas is to use our discretization scheme to generate training data for the supervised learning of a neural network.
By directly mapping inputs to known outputs, this approach provides neural network that are both interpretable and reliable, making it a preferred method in solving PPDEs.	
Let $K$ be a compact subset of  $C([0,T],\mathbb{R}^d)$ and set $R:=\sup_{f\in K}|\!|f|\!|_\infty$. Fix a time partition	$\{0=t_1<t_2<\ldots <t_n=T\}$. Select points
$t_{(1)},\cdots,t_{(I)}\in [0,T]$ 
and
$x_{(1)},\cdots, x_{(J)}\in \mathbb{R}^{d\times n}$. 
Using our scheme, we compute numerical solutions $u(t_{(i)},\eta^{x_{(j)}})$ for $i=1,\ldots,I$ and $j=1,\ldots,J$, which serve as training data. The network is then trained by minimizing
$$\sum_{i=1}^I\sum_{j=1}^J|u(t_{(i)},\eta^{x_{(j)}})-\mathcal{N}_\theta(t_{(i)},x_{(j)})|^2\,.$$
Once trained, the network can instantly provide an approximate solution $u(t,\gamma)$ for any initial path $(t,\gamma)\in[0,T]\times K$. In this way, the neural network approach overcomes the drawback of our direct method when applied to PPDEs.

\section{Error analysis for the martingale integrand}
\label{sec:error}

This section conducts an error analysis for the Monte Carlo estimator in \eqref{scheme}.
Theorem \ref{thm_con_rate}
demonstrates the
concentration inequalities that quantify
the statistical error of the estimator. 	The following lemma can be found in the proof of Theorem 6.2.1 in Gobet~\cite{montesimbook}.

\begin{lemma}\label{eqn:expo_inequ}
	Let $Y$ be an $\ell$-dimensional standard normal random variable, and $h: \mathbb{R}^{\ell} \mapsto \mathbb{R}$ be a Lipschitz continuous function with Lipschitz constant $|h|_{Lip}$. Then, 		$\mathbb{E}[ e^{\lambda (h(Y)-E[h(Y)])}] \le e^{\frac{1}{2} |h|_{Lip}^2\lambda^2}$ for all $\lambda \in \mathbb{R}$.
\end{lemma}

\begin{lemma} \label{meanerrorlemma}
	Let $Y=(Y_1,\cdots,\,Y_{\ell})$ be an $\ell$-dimensional standard normal random variable and $h: \mathbb{R}^{\ell} \mapsto \mathbb{R}$ be a Lipschitz continuous function. Then, there exist positive constants $c,\delta$ such that   $\mathbb{E}[ e^{\lambda (Y_1 h(Y)-E[Y_1 h(Y)])}] \le e^{c\lambda^2}$ for all $\lambda \in(-\delta,\delta)$.
\end{lemma}
\begin{proof}
	Denote by $|h|_{Lip}$ the Lipschitz constant of $h$ and assume $|h|_{Lip}\ne 0$. Define $H(\lambda):=\mathbb{E}[ e^{\lambda Y_1 h(Y)}]$. Then, it can be easily verified that $H(0)=1$, $H(\cdot)>0$,
	$H$ is smooth on  $|\lambda|< \frac{1}{2|h|_{Lip}}$, and 
	$H'(0)=\mathbb{E}[Y_1 h(Y)].$
	It is enough to show that there exist positive constants $c,\delta$ such that
	$$I(\lambda):=\dfrac{e^{c\lambda^2+\lambda \mathbb{E}[Y_1 h(Y)]}}{H(\lambda)}\ge 1$$ 
	for $|\lambda|<\delta$. 
	Because $I(0)=1,$ this can be achieved by showing that 
	$I$ has a local minimum at $\lambda=0.$ 
	We claim that 
	$I'(\lambda)>0$ for $0<\lambda<\delta$ and $I'(\lambda)<0$ for $-\delta<\lambda<0$ for some $\delta>0.$
	A direct calculation gives
	\begin{equation}
		\begin{aligned}
			I'(\lambda)& = \dfrac{(2c\lambda+\mathbb{E}[Y_1 h(Y)])e^{c\lambda^2+\lambda \mathbb{E}[Y_1 h(Y)]}H(\lambda)-e^{c\lambda^2+\lambda \mathbb{E}[Y_1 h(Y)]}H'(\lambda)}{H(\lambda)^2} \\
			&=\frac{e^{c\lambda^2+\lambda \mathbb{E}[Y_1 h(Y)]}}{H(\lambda)}\bigg(2c\lambda+\mathbb{E}[Y_1 h(Y)]-\dfrac{H'(\lambda)}{H(\lambda)}\bigg).
		\end{aligned}
	\end{equation}  
	Because, $\frac{e^{c\lambda^2+\lambda \mathbb{E}[Y_1 h(Y)]}}{H(\lambda)}$ is positive,
	the signs of $I'(\lambda)$ and $2c\lambda+\mathbb{E}[Y_1 h(Y)]-\frac{H'(\lambda)}{H(\lambda)}$ are the same. The Taylor expansion of $\frac{H'(\lambda)}{H(\lambda)}$ at $\lambda=0$ is
	$$\dfrac{H'(\lambda)}{H(\lambda)}= \mathbb{E}[Y_1 h(Y)]+\Big(\mathbb{E}[Y_1^2h(Y)^2]-\mathbb{E}[Y_1 h(Y)]^2\Big)\lambda+o(\lambda^2).$$
	By choosing $c>0$ satisfying $2c>\mathbb{E}[Y_1 ^2h(Y)^2]-\mathbb{E}[Y_1 h(Y)]^2$ and $\delta>0$ sufficiently small, we obtain the desired result. 
\end{proof}

For the remainder of this section, we assume that
$\pi=\{t_0,\ldots,t_n\}$ is a uniform partition and denote
its common step size by
\[
h:=|\pi|=t_{j+1}-t_j,\qquad j=0,\ldots,n-1.
\]

\begin{lemma} \label{liplem}
	Suppose Assumptions \ref{asmsde} and \ref{asmbsde} hold, and $\sigma$ is uniformly bounded.
	Given  $t,\gamma,\pi,m,i,$ and $\omega,$ we define a map $\Phi:D([0,T],\mathbb{R}^\ell)\to \mathbb{R}$ as 
	\begin{equation}\label{eqn:Phi}	\Phi(\eta):=\,g(X_{T}^{\gamma_{t},\pi}(\omega \mathop{\oplus}_{t_i}\eta)) +\sum_{j=i}^{n-1} f(t_j,X_{t_j}^{\gamma_{t},\pi}(\omega \mathop{\oplus}_{t_i} \eta),Y^{\gamma_{t},\pi,m}(t_j,\omega \mathop{\oplus}_{t_i} \eta),Z^{\gamma_{t},\pi,m}(t_j,\omega \mathop{\oplus}_{t_i} \eta))h.
	\end{equation}
	For
	$y\in \mathbb{R^\ell},$ $k=i+1,i+2,\cdots,n$, and
	a Brownian path $W$,  let $W^{\pi,y,k}=W^{y}$ be a process  on $[0,T]$ such that $W^{y}(s)=W(s)$ for $s\in [0,t_i]$ and 
	$W^{y}=W^{y}(s)$ is a piecewise constant path   
	for $s\in [ t_i,T]$
	satisfying $$W^{y}(t_{j+1})-W^{y}(t_j)=W(t_{j+1})-W(t_j)$$ for $j=i,\cdots,n-1$ except for $j=k-1$ and $$W^{y}(t_k)-W^{y}(t_{k-1})=y\,.$$
	Then a map $H$ defined as 
	$H(y):= \mathbb{E}\big[\Phi(W_T^{y})|W(t_i),\cdots,\, W(t_{k-1})\big]$    
	is Lipschitz continuous and the Lipschitz constant does not depend on $W(t_i),\cdots,\, W(t_{k-1})$.  \end{lemma}

\begin{proof}
	For simplicity, we assume $d=\ell=1$. Define
	\begingroup\small
	\begin{align*}
	\delta X^{\gamma_t,\pi}(s,y,y')
	&:=X^{\gamma_t,\pi}\bigl(s,\omega\mathop{\oplus}_{t_i}W_T^y\bigr)
	 -X^{\gamma_t,\pi}\bigl(s,\omega\mathop{\oplus}_{t_i}W_T^{y'}\bigr),\\
	\delta X_s^{\gamma_t,\pi}(y,y')
	&:=\delta X^{\gamma_t,\pi}(\cdot\wedge s,y,y'),\\
	\delta Y^{\gamma_t,\pi,m}(s,y,y')
	&:=Y^{\gamma_t,\pi,m}\bigl(s,\omega\mathop{\oplus}_{t_i}W_T^y\bigr)
	 -Y^{\gamma_t,\pi,m}\bigl(s,\omega\mathop{\oplus}_{t_i}W_T^{y'}\bigr),\\
	\delta Z^{\gamma_t,\pi,m}(s,y,y')
	&:=Z^{\gamma_t,\pi,m}\bigl(s,\omega\mathop{\oplus}_{t_i}W_T^y\bigr)
	 -Z^{\gamma_t,\pi,m}\bigl(s,\omega\mathop{\oplus}_{t_i}W_T^{y'}\bigr).
	\end{align*}
	\endgroup
	These definitions hold for $s\in[t,T]$ and $y,y'\in\mathbb{R}$. We also set
	\[
	\mathbb{E}_{t_i,\ldots,t_j}[\,\cdot\,]
	:=\mathbb{E}[\,\cdot\mid W(t_i),\ldots,W(t_j)]
	\]
	for $j=i,\ldots,n$. To prove that $H$ is Lipschitz continuous, observe that
	\begin{equation}
	\begin{aligned}
			&|H(y)-H(y')|  \le\mathbb{E}_{t_{i},\cdots,t_{k-1}} \Big[\big|g(X_{T}^{\gamma_{t},\pi}(\omega \mathop{\oplus}_{t_i} W_T^{y}))-g(X_{T}^{\gamma_{t},\pi}(\omega \mathop{\oplus}_{t_i} W_T^{y'}))\big|\Big] \\
			&+\mathbb{E}_{t_{i},\cdots,t_{k-1}} \Big[\sum_{j=i}^{n-1} \big| f(t_j,X_{t_j}^{\gamma_{t},\pi}(\omega \mathop{\oplus}_{t_i} W_T^{y}),Y^{\gamma_{t},\pi,m}(t_j,\omega \mathop{\oplus}_{t_i} W_T^{y}),Z^{\gamma_{t},\pi,m}(t_j,\omega \mathop{\oplus}_{t_i} W_T^{y}))\\
			&- f(t_j,X_{t_j}^{\gamma_{t},\pi}(\omega \mathop{\oplus}_{t_i} W_T^{y'}),  Y^{\gamma_{t},\pi,m}(t_j,\omega \mathop{\oplus}_{t_i} W_T^{y'}),Z^{\gamma_{t},\pi,m}(t_j,\omega \mathop{\oplus}_{t_i} W_T^{y'}))\big|h \Big] \\
			&\le \mathbb{E}_{t_{i},\cdots,t_{k-1}} \Big[K_1(1+T)\big|\!\big|\delta{X}_T^{\gamma_{t},\pi}(y,y')\big|\!\big|_{\infty} \Big] + \mathbb{E}_{t_{i},\cdots,t_{k-1}}\Big[K_1\sum_{j=i}^{n-1} \big|\delta Y^{\gamma_{t},\pi,m}(t_j,y,y') \big|h \Big]\\
			& + \mathbb{E}_{t_{i},\cdots,t_{k-1}}\Big[K_1\sum_{j=i}^{n-1} \big|\delta Z^{\gamma_{t},\pi,m}(t_j,y,y') \big|h\Big]
			=:I_1+I_2+I_3\,.
		\end{aligned}
	\end{equation}

	We first show that  
	$ I_1\leq K_1(1+T) \lVert \sigma \rVert_{\infty}  \big(1+K_1 h +K_1 \sqrt{h}\big)^{n-k} |y-y'| .$
	As $(\omega \mathop{\oplus}_{t_i} W_T^{y})(s)=(\omega \mathop{\oplus}_{t_i} W_T^{y'})(s)$ for $0\le s\le t_{k-1},$  we know that $\delta X(s,y,y')=0$ for $0\le s\le t_{k-1}$ by the Euler scheme in \eqref{eulerscheme}.
	Thus, the mapping \begin{equation}
	\begin{aligned}
			y \mapsto X^{\gamma_{t},\pi}(t_{k},\omega \mathop{\oplus}_{t_i} W_T^{y})&=X^{\gamma_{t},\pi}(t_{k-1},\omega \mathop{\oplus}_{t_i} W_T^{y})+ b(t_{k-1},X_{t_{k-1}}^{\gamma_{t},\pi}(\omega \mathop{\oplus}_{t_i} W_T^{y}))h \\ 
			&\quad+\sigma(t_{k-1},X_{t_{k-1}}^{\gamma_{t},\pi}(\omega \mathop{\oplus}_{t_i} W_T^{y}))y
		\end{aligned}
	\end{equation}
	is Lipschitz continuous with Lipschitz constant $\lVert \sigma \rVert_{\infty}$, as $X_{t_{k-1}}^{\gamma_{t},\pi}(\omega \mathop{\oplus}_{t_i} W_T^{y})$ does not depend on $y$. 
	Define $\Delta W_j:=W(t_{j+1})-W(t_j)$ for $j=0,\cdots,n-1.$   
	We show that 
	\begin{equation}
		\label{eqn:aa} \big|\delta{X}^{\gamma_{t},\pi}(t_{j+1},y,y')\big|\leq\lVert \sigma \rVert_{\infty}|y-y'| \prod_{\nu=k}^{j}\big(1+K_1 h +K_1 |\Delta W_{\nu}|\big)
	\end{equation} 
	for $j=k,\cdots,n-1$  by induction on $j.$
	For $j=k,$ 
	\begingroup\small
\begin{equation}
\fitdisplay{
\begin{aligned}
			|\delta{X}^{\gamma_{t},\pi}(t_{k+1},y,y')| &\le | \delta{X}^{\gamma_{t},\pi}(t_{k},y,y')| + | (b(t_{k},X_{t_{k}}^{\gamma_{t},\pi}(\omega \mathop{\oplus}_{t_i} W_T^{y}))-b(t_{k},X_{t_{k}}^{\gamma_{t},\pi}(\omega \mathop{\oplus}_{t_i} W_T^{y'})))|h \\
			& +| (\sigma(t_{k},X_{t_{k}}^{\gamma_{t},\pi}(\omega \mathop{\oplus}_{t_i} W_T^{y}))-\sigma(t_{k},X_{t_{k}}^{\gamma_{t},\pi}(\omega \mathop{\oplus}_{t_i} W_T^{y'})))\Delta W_{k}| \\
			&\le \; \lVert \sigma \rVert_{\infty} |y-y'|+ K_1\lVert \sigma \rVert_{\infty} h |y-y'|+K_1\lVert \sigma \rVert_{\infty}|y-y'||\Delta W_{k}|\,.
		\end{aligned}
}
\end{equation}
\endgroup	
	Similarly, 
	\begin{equation}
	\begin{aligned}
			|\delta{X}^{\gamma_{t},\pi}(t_{j+1},y,y')| & \le (1+K_1h+K_1|\Delta W_{j}|)| X^{\gamma_{t},\pi}(t_{j},\omega \mathop{\oplus}_{t_i} W_T^{y})-X^{\gamma_{t},\pi}(t_{j},\omega \mathop{\oplus}_{t_i} W_T^{y'})|\\
			&\le \; \lVert \sigma \rVert_{\infty}|y-y'| \prod_{\nu=k}^{j+1}\big(1+K_1 h +K_1 |\Delta W_{\nu}|\big)
		\end{aligned}
	\end{equation}
	by the induction hypothesis.  
	Then,
	we have 
	\begin{equation}
		\begin{aligned}
			I_1&=K_1(1+T)\mathbb{E}_{t_{i},\cdots,t_{k-1}}\Big[ \max_{j=0,\cdots, n}|  \delta{X}^{\gamma_{t},\pi}(t_{j},y,y')| \Big]\\
			&\le K_1(1+T)\lVert \sigma \rVert_{\infty}|y-y'| \prod_{\nu=k}^{n-1}\mathbb{E}\Big[1+K_1 h +K_1 |\Delta W_{\nu}|\Big]\\      &\leq K_1(1+T)\lVert \sigma \rVert_{\infty}  \big(1+K_1 h +K_1 \sqrt{h}\big)^{n-k} |y-y'|\,,
		\end{aligned}
	\end{equation}
	which gives the desired result.
	
	We focus on $I_3$ as a similar analysis applies to $I_2$.  Observe that
	$Y^{\gamma_{t},\pi,m}$ and $Z^{\gamma_{t},\pi,m}$ are non-anticipative functionals, and $\omega \mathop{\oplus}_{t_i} W_T^{y}=\omega \mathop{\oplus}_{t_i} W_T^{y'}$ on $[0,t_{k-1}]$ 
	from the construction of $W_T^y$ and $W_T^{y'}$. 
	It follows that
	$\delta Y^{\gamma_{t},\pi,m}(t_j,y,y')=0$ and $\delta Z^{\gamma_{t},\pi,m}(t_j,y,y')=0$ for $i\le j <k$. Therefore, the summation in $I_3$ need only be taken over indices $j$ with $k \le j \le n-1$. We have

	\begingroup\small
\begin{equation}
\fitdisplay{
\begin{aligned}
			I_3&=\mathbb{E}_{t_{i},\cdots,t_{k-1}}\bigg[ K_1\sum_{j=k}^{n-1}\big|\delta Z^{\gamma_{t},\pi,m}(t_j,y,y') \big|h \bigg]\\
			& \le K_1\sum_{j=k}^{n-1} \mathbb{E}_{t_{i},\cdots,t_{k-1}}  \bigg[ \bigg|\frac{\Delta W_{j}}{h}\bigg|\bigg( \big| g(X_{T}^{\gamma_{t},\pi}(\omega \mathop{\oplus}_{t_i} W_T^{y}))-g(X_{T}^{\gamma_{t},\pi}(\omega \mathop{\oplus}_{t_i} W_T^{y'})) \big|\\
			&\quad+ \sum_{r=j+1}^{n-1} \big|f(t_{r},X_{t_{r}}^{\gamma_{t},\pi}(\omega \mathop{\oplus}_{t_i} W_T^{y}),Y^{\gamma_{t},\pi,m-1}(t_{r},\omega \mathop{\oplus}_{t_i} W_T^{y}),Z^{\gamma_{t},\pi,m-1}(t_{r},\omega \mathop{\oplus}_{t_i} W_T^{y})\\ 
			&\quad- f(t_{r},X_{t_{r}}^{\gamma_{t},\pi}(\omega \mathop{\oplus}_{t_i} W_T^{y'}),Y^{\gamma_{t},\pi,m-1}(t_{r},\omega \mathop{\oplus}_{t_i} W_T^{y'}),Z^{\gamma_{t},\pi,m-1}(t_{r},\omega \mathop{\oplus}_{t_i} W_T^{y'})\big|h\bigg) h\bigg]\\
			&\le K_1 \sum_{j=k}^{n-1} \mathbb{E}_{t_{i},\cdots,t_{k-1}} \bigg[ K_1(1+T)|\Delta W_{j}|\lVert \delta{X}_T^{\gamma_{t},\pi}(y,y')\rVert_{\infty}\bigg] \\
			&\quad+K_1\sum_{j=k}^{n-1} \mathbb{E}_{t_{i},\cdots,t_{k-1}} \bigg[|\Delta W_{j}|\sum_{r=j+1}^{n-1} K_1\Big(\big|\delta Y^{\gamma_{t},\pi,m-1}(t_r,y,y') \big|+\big|\delta Z^{\gamma_{t},\pi,m-1}(t_r,y,y') \big|\Big)h \bigg]\,.
		\end{aligned}
}
\end{equation}
\endgroup 
	We repeat this procedure until $m=0$. Then, every term in the summation of $I_3$ takes the form 
	\begin{equation}
		\sum_{k \le r_1<\cdots<r_v\le n-1, v\le m+1} \mathbb{E}_{t_{i},\cdots,t_{k-1}}\big[ |\Delta W_{r_1}|\cdots |\Delta W_{r_v}| \lVert\delta{X}_T^{\gamma_{t},\pi}(y,y')\rVert_{\infty} \big]
	\end{equation}
	up to multiplicative constants.
	Finally, note that
	\begin{equation}
		\begin{aligned}
			&\mathbb{E}_{t_{i},\cdots,t_{k-1}}\big[ |\Delta W_{r_1}|\cdots |\Delta W_{r_v}| \lVert \delta{X}_T^{\gamma_{t},\pi}(y,y')\rVert_{\infty} \big] \\
			&\le \mathbb{E}\big[|\Delta W_{r_1}|^2\big]^{\frac{1}{2}}\cdots \mathbb{E}\big[|\Delta W_{r_v}|^2\big]^{\frac{1}{2}} \mathbb{E}_{t_{i},\cdots,t_{k-1}}\big[\lVert \delta{X}_T^{\gamma_{t},\pi}(y,y')\rVert_{\infty}^2\big]^{\frac{1}{2}}
			\le c|y-y'|
		\end{aligned}
	\end{equation}
	by H$\ddot{\text{o}}$lder's inequality, together with
	$$\mathbb{E} \big[\big|\!\big|\delta{X}_T^{\gamma_{t},\pi}(y,y')\big|\!\big|_{\infty}^2 \big] \le 3^{n-k}\lVert \sigma \rVert_{\infty}^2  \big(1+K_1^2 h^2 +K_1^2 h\big)^{n-k} |y-y'|^2 $$
	which follows directly from \eqref{eqn:aa}. Hence we conclude that $ I_3\le c|y-y'|$ for some constant $c>0.$
\end{proof}

Now, we establish the concentration inequality for the proposed numerical scheme. The following theorem is one of main results of this paper.

\begin{theorem}\label{thm_con_rate}
	Suppose Assumptions \ref{asmsde} and \ref{asmbsde} hold and that $\sigma$ is uniformly bounded.
	Let $W_T$ be a Brownian motion and 
	$(W_T^{(l)})_{1\le l \le L}$ be independent $L$ copies of $W_T.$ Given $t,\gamma,\pi,m,i, \omega,$ we have the following. 
	\begin{enumerate}[(i)]
		\item  Recall the  map $\Phi:D([0,T],\mathbb{R}^\ell)\to \mathbb{R}$ from  \eqref{eqn:Phi}. There exists a constant  $c>0$ such that
		\begin{equation}
			\mathbb{P}\Big(\Big| \frac{1}{L} \sum_{l=1}^L \Phi(W_T^{(l)})-\mathbb{E}[\Phi(W_T)] \Big|>\epsilon \Big) \le 2e^{-\frac{L\epsilon^2}{4c}}\text{ for all } \epsilon>0\,.
		\end{equation}  
		\item  Define a map $\Psi^{k}:D([0,T],\mathbb{R}^\ell)\to \mathbb{R}$ as 
		\begin{equation}
			\begin{aligned}
				\Psi^{k}(\eta):=& \,\frac{\eta_{k}(t_{i+1})-\eta_{k}(t_i)}{h}\Big( g(X_{T}^{\gamma_{t},\pi}(\omega \mathop{\oplus}_{t_i} \eta)) \\ &+\sum_{j=i+1}^{n-1} f(t_j,X_{t_j}^{\gamma_{t},\pi}(\omega \mathop{\oplus}_{t_i} \eta),Y^{\gamma_{t},\pi,m}(t_j,\omega \mathop{\oplus}_{t_i} \eta),Z^{\gamma_{t},\pi,m}(t_j,\omega \mathop{\oplus}_{t_i} \eta))h \Big)
			\end{aligned}
		\end{equation} where $\eta_{k}$ is $k$-th coordinate of $\eta$ for $1\le k \le \ell$. Then, there exist positive constants $c$ and $\delta$ such that 
		\begin{equation}
			\mathbb{P}\Big(\Big| \dfrac{1}{L} \sum_{l=1}^L \Psi^{k}(W_T^{(l)})-\mathbb{E}[\Psi^{k}(W_T)] \Big|>\epsilon \Big) \le 2e^{-\frac{L\epsilon^2}{4 c}}
		\end{equation} 
		for all $1\le k\le \ell$ and  $0<\epsilon<2c\delta.$	
	\end{enumerate} 
\end{theorem}

\begin{proof}
	Since the proof of (i) is analogous to that of (ii), we only provide the proof of (ii). 
	For fixed $k$,  denote $\Psi=\Psi^{k}$ for notational convenience. We show that there exist a constant $c>0$ such that
	\begin{equation}\label{exponentialineq}
		\mathbb{E}[ e^{\lambda (\Psi(W_T)-E[\Psi(W_T)])}] \le e^{c\lambda^2}  \textnormal{ for } |\lambda|< \delta
	\end{equation}
	for some $\delta>0$. 	Once this is proven, by the independence of $(W_T^{(l)})_{1\le l \le L}$, we obtain
	\begin{equation}\label{expopsi}
		\mathbb{E}[ e^{\lambda (\frac{1}{L} \sum_{l=1}^L\Psi(W_T^{(l)})-E[\Psi(W_T)])}] \le e^{\frac{c\lambda^2}{L}}
	\end{equation}
	for $|\lambda|<\delta {L}.$
	By the Chebyshev exponential inequality, for any $\epsilon>0$ and $|\lambda|<\delta L$,
	\begin{equation}
		\mathbb{P}\Big(\Big| \dfrac{1}{L} \sum_{l=1}^L \Psi(W_T^{(l)})-\mathbb{E}[\Psi(W_T)] \Big|>\epsilon \Big) 	\le 2e^{\frac{c\lambda^2}{L}-\epsilon \lambda} 
	\end{equation}
	If $0<\epsilon<2c\delta$, then 
	$\lambda=\frac{L\epsilon}{2c}$ gives the desired result. 
	
	Now, we prove the inequality in \eqref{exponentialineq}.
	Define $\Psi_j:=\mathbb{E}_{t_{i},\cdots,t_{j}}[\Psi(W_T)]$.  It is evident that
	\begin{equation}
		\Psi(W_T)-\mathbb{E}[\Psi(W_T)]=\sum_{j=i}^{n-1} (\Psi_{j+1}-\Psi_j)=\sum_{j=i}^{n-1} (\Psi_{j+1}-\mathbb{E}_{t_{i},\cdots,t_{j}}[\Psi_{j+1}])\,.
	\end{equation}
	From the definition of $\Psi$, we have that for $j=i+1,\cdots,n-1$,
	$$\lambda (\Psi_{j+1}-\mathbb{E}_{t_{i},\cdots,t_{j}}[\Psi_{j+1}])=\lambda \frac{\Delta W_i}{h} (\Phi_{j+1}-\mathbb{E}_{t_{i},\cdots,t_{j}}[\Phi_{j+1}])\,.$$
	Then,
	\begin{equation}
\fitdisplay{
\begin{aligned}
			&\;\quad\mathbb{E}[ e^{\lambda (\Psi(W_T)-E[\Psi(W_T)])}]=\mathbb{E}[ e^{\lambda \sum_{j=i}^{n-2} (\Psi_{j+1}-\mathbb{E}_{t_{i},\cdots,t_{j}}[\Psi_{j+1}]) }\mathbb{E}_{t_{i},\cdots,t_{n-1}}[ e^{\lambda (\Psi_{n}-\mathbb{E}_{t_{i},\cdots,t_{n-1}}[\Psi_{n}]) }]] \\
			&=\mathbb{E}\Big[ e^{\lambda \sum_{j=i}^{n-2} (\Psi_{j+1}-\mathbb{E}_{t_{i},\cdots,t_{j}}[\Psi_{j+1}]) }\mathbb{E}_{t_{i},\cdots,t_{n-1}}\Big[ e^{\lambda \frac{\Delta W_i}{h} (\Phi_{n}-\mathbb{E}_{t_{i},\cdots,t_{n-1}}[\Phi_{n}]) }\Big]\Big] \\
			&\le \mathbb{E}\Big[ e^{\lambda \sum_{j=i}^{n-2} (\Psi_{j+1}-\mathbb{E}_{t_{i},\cdots,t_{j}}[\Psi_{j+1}]) }e^{c_1\lambda^2\big(\frac{\Delta W_i}{h}\big)^2}\Big]
		\end{aligned}
}
\end{equation}
	for some positive constant $c_1$.
	The last inequality can be derived from Lemmas \ref{eqn:expo_inequ} and \ref{liplem} by considering the Lipschitz map
	\begin{equation}\label{eqn:h}\begin{aligned}
			y\mapsto 		h(y)=\lambda \frac{\Delta W_i}{h}\mathbb{E}\big[\Phi(W_T^{y})|W(t_i),\cdots,\, W(t_{n-1})\big]
		\end{aligned}
	\end{equation} with Lipschitz constant 		
	$c\big|\frac{\Delta W_i}{h}\big|$ where $W^{y}$ is the process defined in Lemma \ref{liplem} with $W^y(t_n)-W^y(t_{n-1})=y$.
	Repeating this procedure until
	$j=i+1$, we obtain
	\begin{equation}
		\begin{aligned}
			\mathbb{E}\Big[ e^{\lambda (\Psi(W_T)-E[\Psi(W_T)])}\Big]& \le \mathbb{E}\Big[ e^{\lambda  \big(\Psi_{i+1}-\mathbb{E}[\Psi_{i+1}|W(t_i)]\big) }e^{\frac{c_1(n-i-1)\lambda^2}{h^2}(\Delta W_i)^2}\Big] \\
			& \le \mathbb{E}\Big[ e^{2\lambda  \big(\Psi_{i+1}-\mathbb{E}[\Psi_{i+1}|W(t_i)]\big) } \Big]^{\frac{1}{2}} \mathbb{E} \Big[e^{\frac{2c_1(n-i-1)\lambda^2}{h^2}(\Delta W_i)^2}\Big]^{\frac{1}{2}}=:I_1  I_2\,.
		\end{aligned}
	\end{equation}
	By Lemma \ref{meanerrorlemma}, there exist positive constants $c_2$ and $\tilde{\delta}$ such that for   $|\lambda|<\tilde{\delta}$,
	\begin{equation}
		I_1= \mathbb{E}\Big[ e^{2\lambda  \big(\Psi_{i+1}-\mathbb{E}[\Psi_{i+1}]\big) } \Big]^{\frac{1}{2}}= \mathbb{E}\Big[ e^{2\lambda  \big(\frac{\Delta W_i}{h}\Phi_{i+1}-\mathbb{E} [\frac{\Delta W_i}{h}\Phi_{i+1}] \big) } \Big]^{\frac{1}{2}}\le e^{c_2\lambda^2}\,.
	\end{equation} Moreover, using the fact that $\frac{\Delta W_i}{\sqrt{h}}$ has a standard normal distribution, we can directly derive that
	\begin{equation}
		\begin{aligned}
			I_2^2=\dfrac{1}{\sqrt{1-\frac{4c_1(n-i-1)}{h}\lambda^2}}\le e^{\frac{4c_1(n-i-1)}{h}\lambda^2}
		\end{aligned}
	\end{equation}
	for $\lambda\in\mathbb{R}$ satisfying 
	$\frac{4c_1(n-i-1)}{h}\lambda^2<\frac{1}{2}$.
	Combining these inequalities for $I_1$ and $I_2$, we obtain
	$$\mathbb{E}[ e^{\lambda (\Psi(W_T)-E[\Psi(W_T)])}] \le e^{c\lambda^2}  \textnormal{ for } |\lambda|< \delta$$
	where $\delta=\min(\tilde{\delta}, {\sqrt{\frac{h}{8c}}})$ and $c=\frac{2c_1}{h}(n-i-1)+c_2$. 
\end{proof}

\section{Estimating the conditional expectations}
\label{sec:condi}

A direct implementation of this Picard iteration scheme is inefficient, as it involves nested expectations over concatenated Brownian motion paths. To evaluate the approximation $Y^{\gamma_t,\pi,m+1}(t_i,\omega)$, one must store the values of $Y^{\gamma_t,\pi,m}(t_j,\omega \oplus_{t_i} W_T^{(l)})$ and $Z^{\gamma_t,\pi,m}(t_j,\omega \oplus_{t_i} W_T^{(l)})$ for all $i \le j \le n-1$ and $1 \le l \le L$, as required by the scheme:
\begin{equation}
\fitdisplay{
\begin{aligned}
		&\quad Y^{\gamma_{t},\pi,m+1}(t_i,\omega)\\
		&= \mathbb{E}\bigg[ g(X_{T}^{\gamma_{t},\pi}(\omega \mathop{\oplus}_{t_i} W_T))\\
		&
		\quad+\sum_{j=i}^{n-1} f(t_j,X_{t_j}^{\gamma_{t},\pi}(\omega \mathop{\oplus}_{t_i} W_T),Y^{\gamma_{t},\pi,m}(t_j,\omega \mathop{\oplus}_{t_i} W_T),Z^{\gamma_{t},\pi,m}(t_j,\omega \mathop{\oplus}_{t_i} W_T))(t_{j+1}-t_j)\bigg]\\
		&\approx \frac{1}{L}\sum_{l=1}^L\bigg[ g(X_{T}^{\gamma_{t},\pi}(\omega \mathop{\oplus}_{t_i} W_T^{(l)}))\\
		&
		\quad+\sum_{j=i}^{n-1} f(t_j,X_{t_j}^{\gamma_{t},\pi}(\omega \mathop{\oplus}_{t_i} W_T^{(l)}),Y^{\gamma_{t},\pi,m}(t_j,\omega \mathop{\oplus}_{t_i} W_T^{(l)}),Z^{\gamma_{t},\pi,m}(t_j,\omega \mathop{\oplus}_{t_i} W_T^{(l)}))(t_{j+1}-t_j)\bigg]\,.
	\end{aligned}
}
\end{equation}
However, in order to store the values $Y^{\gamma_t,\pi,m}(t_j,\omega \oplus_{t_i} W_T^{(l)})$ and $Z^{\gamma_t,\pi,m}(t_j,\omega \oplus_{t_i} W_T^{(l)})$, one must first compute $Y^{\gamma_t,\pi,m-1}(t_k,\omega \oplus_{t_i} W_T^{(l)} \oplus_{t_j} \tilde{W}_T^{(\tilde{l})})$ and $Z^{\gamma_t,\pi,m-1}(t_k,\omega \oplus_{t_i} W_T^{(l)} \oplus_{t_j} \tilde{W}_T^{(\tilde{l})})$. 
Iterating this argument implies that
the evaluation at Picard step $m$ requires values along $m$ concatenated Brownian motion paths, which makes the direct implementation computationally expensive.

To overcome this difficulty, at each Picard iteration step we employ the neural network architecture introduced in Saporito and Zhang~\cite{deepgalerkin}, which combines feed-forward networks with a long short term memory (LSTM) structure, to approximate $Y^{\gamma_t,\pi,m+1}(t_i,\omega)$ and $Z^{\gamma_t,\pi,m+1}(t_i,\omega)$ for $t_i \in \pi$ and $\omega \in D([0,T],\mathbb{R}^d)$.  For simplicity we restrict the presentation to the one-dimensional case $d=\ell=1$, although the same approach applies in the multidimensional setting.  Observe that both $Y^{\gamma_t,\pi,m+1}(t_i,\omega)$ and $Z^{\gamma_t,\pi,m+1}(t_i,\omega)$ depend only on the path history $\{\omega(t_j)\}_{j=0}^i$, and these quantities are completely determined by $X_{t_i}^{\gamma_t,\pi}(\omega)$. Therefore, instead of feeding $(t_i,\omega)$ directly into the network, we use the pair $(t_i,X_{t_i}^{\gamma_t,\pi}(\omega))$ as the input representation. The neural network is then trained using $(t_i,X_{t_i}^{\gamma_t,\pi}(\omega))$ as input. Accordingly, at the $(m+1)$-th Picard iteration, we approximate $Y^{\gamma_t,\pi,m+1}(t_i,\omega)$ and $Z^{\gamma_t,\pi,m+1}(t_i,\omega)$ by the neural network outputs $u^{m+1}(t_i,\omega;\theta^{m+1})$ and $v^{m+1}(t_i,\omega;\theta^{m+1})$, respectively, given by
\begin{equation}
\fitdisplay{
\begin{aligned}
		u^{m+1}(t_i,\omega;\theta^{m+1})=&\tilde{u}^{m+1}(t_i,X_{t_i}^{\gamma_t,\pi}(\omega);\theta^{m+1})=\varphi^{m+1}(t_i, X^{\gamma_t,\pi}(t_i,\omega), a_{t_{i-1}}^{m+1,u};\theta^{m+1,f})\,, \\ v^{m+1}(t_i,\omega;\theta^{m+1})=&\tilde{v}^{m+1}(t_i,X_{t_i}^{\gamma_t,\pi}(\omega);\theta^{m+1})=\psi^{m+1}(t_i, X^{\gamma_t,\pi}(t_i,\omega), a_{t_{i-1}}^{m+1,v};\theta^{m+1,f})\,. 
	\end{aligned}
}
\end{equation}
Here, $\varphi^{m+1}$ and $\psi^{m+1}$ denote feed-forward neural networks in $\mathbb{NN}_{k+2,1}^l$ with $k$ hidden units, while $a^{m+1,u}$ and $a^{m+1,v}$ are the corresponding output vectors of the LSTM, for each $m=0,1,\cdots$. The precise definition can be found in Saporito and Zhang~\cite{deepgalerkin}. Here, the fixed initial path $\gamma_t$ is not treated as an input path variable in the neural network.

We now define the loss functions for the neural networks based on the scheme \eqref{scheme}. The approximated labels $\mathcal{U}^{m+1}(t_i,\omega)$ and $\mathcal{V}^{m+1}(t_i,\omega)$ are defined as
\begin{equation}\label{labelofscheme}
\begin{aligned}
		\mathcal{U}^{m+1}(t_i,\omega)&:=\frac{1}{L}\sum_{l=1}^L\bigg[ g(X_{T}^{\gamma_{t},\pi}(\omega \mathop{\oplus}_{t_i} W_T^{(l)}))\\
		&
		\quad+\sum_{j=i}^{n-1} f(t_j,X_{t_j}^{\gamma_{t},\pi}(\omega \mathop{\oplus}_{t_i} W_T^{(l)}),\\
		&\qquad u^{m}(t_j,\omega \mathop{\oplus}_{t_i} W_T^{(l)}),v^{m}(t_j,\omega \mathop{\oplus}_{t_i} W_T^{(l)}))(t_{j+1}-t_j)\bigg],\,\\
		\mathcal{V}^{m+1}(t_i,\omega)&:=\frac{1}{L}\sum_{l=1}^L \frac{W^{(l)}(t_{i+1})-W^{(l)}(t_i)}{t_{i+1}-t_i}\bigg[ g(X_{T}^{\gamma_{t},\pi}(\omega \mathop{\oplus}_{t_i} W_T^{(l)}))\\
		&
		\quad+\sum_{j=i+1}^{n-1} f(t_j,X_{t_j}^{\gamma_{t},\pi}(\omega \mathop{\oplus}_{t_i} W_T^{(l)}),\\
		&\qquad u^{m}(t_j,\omega \mathop{\oplus}_{t_i} W_T^{(l)}),v^{m}(t_j,\omega \mathop{\oplus}_{t_i} W_T^{(l)}))(t_{j+1}-t_j)\bigg]\,.
	\end{aligned}
\end{equation}
We then define the loss function $J_{n,L}^{m+1}$ as follows:
\begin{equation}\label{lossfunc}
	\begin{aligned}
	J_{n,L}^{m+1}(\theta^{m+1})
	:=\frac{1}{nL}\sum_{i=0}^{n-1}\sum_{l=1}^L\Big(&
	|\mathcal{U}^{m+1}(t_i,\omega^{(l)})
	-u^{m+1}(t_i,\omega^{(l)};\theta^{m+1})|^2\\
	&+|\mathcal{V}^{m+1}(t_i,\omega^{(l)})
	-v^{m+1}(t_i,\omega^{(l)};\theta^{m+1})|^2\Big).
	\end{aligned}
\end{equation}
Before presenting the algorithm, we emphasize that minimizing the loss function in \eqref{lossfunc} requires evaluating the target values in \eqref{labelofscheme} for every $(t_i,\omega^{(l)})$. Recomputing these target values at every epoch in order to update the network parameters $\theta$ would lead to an impractically long training time. Instead, we compute these target values in advance and construct a fixed training dataset prior to the optimization loop. Consequently, at each epoch, we can efficiently sample random mini-batches from this precomputed dataset. The complete procedure is described in Algorithm \ref{algorithm}.

\begin{algorithm}[H] 
	\caption{Picard iteration scheme for path-dependent FBSDEs using neural networks}
	\label{algorithm}
	\begin{algorithmic}[1]
		\State \textbf{Input:} initial path $\gamma$, initial time $t$, time grid $\{t_i\}_{i=0}^n$, Picard steps $M$, training dataset size $K$, the number of training paths $L$, the number of Monte Carlo simulation paths in \eqref{labelofscheme} $\tilde{L}$
		\State Initialize a neural network and set the initial iterate $(u^0,v^0) \equiv 0$
		
		\For{$m=1,\dots,M$}
		\State Initialize the network parameters with $\theta^{m-1}$
		\State Construct a training dataset $\mathcal{D}^{(m)}=\{\mathcal{D}_k^{(m)}\}_{k=1}^K$ as follows:
		\For{$k=1,\dots,K$}
		\State Generate $L$ training paths and $\tilde{L}$ Brownian motion paths
		\State Compute \eqref{labelofscheme} with $\tilde{L}$ Monte Carlo samples 
		\State Store the computed $(\mathcal{U}^{m},\mathcal{V}^{m})$ together with the time and training paths in $\mathcal{D}_k^{(m)}$
		\EndFor
		
		\For{epochs}
		\State Sample a mini-batch from $\mathcal{D}^{(m)}$
		\State Compute $(u^m(\cdot;\theta^m),v^m(\cdot;\theta^m))$ at each time $t_i$ for each training path in the mini-batch
		\State Update $\theta^m$ by minimizing $J_{n,L}^{m}(\theta^{m})$ defined in \eqref{lossfunc}
		\EndFor
		\EndFor
		
		\State \Return the trained network $u^M(\cdot ;\theta^M)$ and $v^M(\cdot ;\theta^M)$
	\end{algorithmic}
\end{algorithm}

\section{Numerical simulation} \label{sec:NS}
This section presents numerical experiments illustrating the performance of the proposed Picard iteration scheme  based on Algorithm \ref{algorithm}. To assess its accuracy, we consider two examples with closed-form solutions and compare the numerical approximations generated by the scheme with the corresponding exact solutions.
  Throughout the experiments, we use the parameters
\begin{equation}\label{parameter}
	\begin{aligned}
	&n=70,\qquad
	\pi=\Big\{t+\frac{(T-t)i}{n}\,\Big|\,i=0,\ldots,n\Big\},
	\qquad M=10,\quad K=70,\\
	&L=100,\qquad \tilde{L}=15000,\qquad
	t=0,\qquad \gamma(0)=1,\qquad T=1.
	\end{aligned}
\end{equation}
and simulate the training paths as Brownian motion paths. At each Picard iteration step $m$, we evaluate $u^m(t,\cdot;\theta^m)$ and $v^m(t,\cdot;\theta^m)$ on $100{,}000$ test Brownian motion paths and take their empirical averages, denoted by $\tilde{Y}^m(t)$ and $\tilde{Z}^m(t)$, which serve as approximations to $Y^{\gamma_t,\pi,m}(t)$ and $Z^{\gamma_t,\pi,m}(t)$, respectively. In the sequel, we define
\begin{equation}
	\begin{aligned}
		& \text{Loss}_Y^m:= \frac{1}{nL}\sum_{i=0}^{n-1} \sum_{l=1}^L  |\mathcal{U}^{m+1}(t_i,\omega^{(l)})-u^{m+1}(t_i,\omega^{(l)};\theta^{m+1})|^2\,, \\
		&\text{Loss}_Z^m:= \frac{1}{nL}\sum_{i=0}^{n-1} \sum_{l=1}^L 
		|\mathcal{V}^{m+1}(t_i,\omega^{(l)})-v^{m+1}(t_i,\omega^{(l)};\theta^{m+1})|^2
	\end{aligned}
\end{equation}
to separately report the approximation errors of $Y$ and $Z$.

The overall running time remains moderate for reasonably small values of $M$. As stated in Theorem \ref{thm:main}, the convergence rate is exponential until the error reaches the discretization level $C|\pi|\log\frac{1}{|\pi|}$. Thus, the approximation error is already small for moderate values of $M$. Since only a few neural networks need to be trained, the overall computational cost remains manageable. Consequently, the proposed scheme is computationally efficient and has a moderate running time. In the examples, the algorithm already achieves small approximation errors for $M=5$ or $6$.

\subsection{Example I}
\label{sec:ex1}
In this section, we compute a numerical solution to the second-order parabolic PPDE
\begin{equation} \label{eqn:appli_PDE}
	\begin{aligned}
		&\partial_{s}u+\frac{1}{2}\sigma^2 \partial_{\omega\omega}^2 u-\frac{1}{2} \omega(s) \partial_{\omega} u + \frac{1}{2} \omega(s) + \int_0^s \omega(u)\,du -u=0 \,,\;0\le s<T\,,\\
		&u(T,\omega)=\omega(T)+\int_0^T \omega(u)\,du\,.
	\end{aligned}
\end{equation}
For initial time $t$ and initial path $\gamma$, we consider the   SDE
\begin{equation} 
	\begin{aligned}
		X^{\gamma_{t}}(s)&=\gamma_{t}(t)-\int_{t}^{s} \dfrac{1}{2} X^{\gamma_{t}}(u)\,du+ \int_{t}^{s} \sigma \,dW(u) \,,&& t\le s \le T\,, \\
		X^{\gamma_{t}}(s)&=\gamma_{t}(s) \,, && 0 \le s \le t
	\end{aligned}
\end{equation}
and the BSDE
\begin{equation}
\fitdisplay{
\begin{aligned} 
		Y^{\gamma_{t}}(s)&=\bigg( X^{\gamma_{t}}(T)+\int_{0}^{T} X^{\gamma_{t}}(u)\,du \bigg) \\
		&+\int_{s}^{T}\bigg( \frac{1}{2} X^{\gamma_{t}}(u)+\int_{0}^{u} X^{\gamma_{t}}(r)\,dr -Y^{\gamma_{t}}(u) \bigg) \,du-\int_{s}^{T} Z^{\gamma_{t}}(u)\,dW(u) \,,\; t\le s\le T.
	\end{aligned}
}
\end{equation}
By Proposition \ref{uniclassical}, the mapping $u(t,\gamma):=Y^{\gamma_t}(t)$ is the unique classical solution to \eqref{eqn:appli_PDE}. We compare our numerical solution with the exact solution. This PPDE admits the closed-form solution
\begin{equation}
	u(t,\gamma)=\gamma(t)+\int_0^t \gamma(u)\,du.
\end{equation}
Moreover, since $\partial_{\omega}u(t,\gamma)=1$, it follows that
\begin{equation}
	Y^{\gamma_t}(t)=\gamma(t) +\int_0^t \gamma(u)\,du   \quad\text{and}\quad    Z^{\gamma_t}(t)=\sigma.
\end{equation}
In this experiment, we set $\sigma=0.3$ and use the parameters in \eqref{parameter}. We train the model for $5000$ epochs, with the learning rate initialized at $0.01$ and reduced by a factor of $0.98$ every $50$ epochs during training.

Table \ref{ppdevalue_transposed} reports the approximated values $\tilde{Y}^m(t)$ and $\tilde{Z}^m(t)$ for $1 \le m \le 10$.
For $m \geq 5$, the relative errors with respect to the exact values $Y^{\gamma_t}(0)=1$ and $Z^{\gamma_t}(0)=0.3$ are below $0.5\%$ for $\tilde{Y}^m(t)$ and below $3.3\%$ for $\tilde{Z}^m(t)$. Moreover, for each Picard iteration, the loss values at the final epoch are almost between $10^{-5}$ and $2\times 10^{-5}$ for $\text{Loss}_Y$, and between $7\times 10^{-3}$ and $8\times 10^{-3}$ for $\text{Loss}_Z$.
As can be seen from the scheme \eqref{scheme}, the $Z$-term involves an additional Brownian increment, which increases the variance of the corresponding estimator. As a result, $\text{Loss}_Z$ is generally larger than $\text{Loss}_Y$. This implies that learning $Z$ is more challenging and may produce larger approximation errors than learning $Y$.

\begin{table}[htbp]
\caption{Approximated values $\tilde{Y}^m(t),\, \tilde{Z}^m(t)$ for each Picard iteration for \eqref{eqn:appli_PDE}}
\centering
\begin{tabular}{|c|c|c|}
\hline
$m$ & $\tilde{Y}^m(t)$ & $\tilde{Z}^m(t)$ \\ \hline
1 & 1.3991 & 0.3789 \\ \hline
2 & 0.8248 & 0.2412 \\ \hline
3 & 1.0602 & 0.3197 \\ \hline
4 & 0.9904 & 0.2986 \\ \hline
5 & 1.0038 & 0.2964 \\ \hline
6 & 0.9944 & 0.2971 \\ \hline
7 & 0.9995 & 0.2993 \\ \hline
8 & 1.0050 & 0.2902 \\ \hline
9 & 0.9961 & 0.2923 \\ \hline
10 & 0.9974 & 0.3001 \\ \hline
\end{tabular}

\label{ppdevalue_transposed}
\end{table}

\subsection{Example II}
This section considers a nonlinear PPDE inspired by a two person zero sum game.
Motivated by Feng et al.~\cite{deepsignature}, we consider a second-order parabolic PPDE with an additional $\sigma$-dependent term.
\begingroup\small
\begin{equation}\label{eqn:ex2}
\fitdisplay{
\begin{aligned}
		&\partial_{s}u+\frac{1}{2} \sigma^2 \partial_{\omega\omega}^2u + \min_{\mu \in [ \underline{\mu},\bar{\mu}]} \mu \partial_{\omega}u+\max(\text{sin}(\omega(s)+\int_0^s \omega(u)\,du),0)(\bar{\mu}+\omega(s))\\ &\quad\quad+\min(\text{sin}(\omega(s)+\int_0^s \omega(u)\,du),0)(\underline{\mu}+\omega(s))+\frac{1}{2} \sigma^2 \text{cos}\big(\omega(s)+\int_0^s \omega(u)\,du \big)=0 \,,\;0\le s<T\,,\\
		&u(T,\omega)=\text{cos}\big(\omega(T)+\int_0^T \omega(u)\,du\big)
	\end{aligned}
}
\end{equation}
\endgroup
for $\underline{\mu}< \bar{\mu}.$
We estimate a solution $u$ to this PPDE in the following way. 
For initial time $t$ and initial path $\gamma$, consider the SDE 
\begin{equation} 
	\begin{aligned}
		X^{\gamma_{t}}(s)&=\gamma_{t}(t)+ \int_{t}^{s} \sigma dW(u)=\gamma_{t}(t)+\sigma(W(s)-W(t)) \,,&& t\le s \le T\,, \\
		X^{\gamma_{t}}(s)&=\gamma_{t}(s) \,, && 0 \le s \le t
	\end{aligned}
\end{equation}
and the BSDE
\begingroup\small
\begin{equation}
\begin{aligned} 
		Y^{\gamma_{t}}(s)&=\text{cos}\bigg( X^{\gamma_{t}}(T)+\int_{0}^{T} X^{\gamma_{t}}(u)\,du \bigg) \\
		&\;+\int_{s}^{T}\bigg( \min_{\mu \in [ \underline{\mu},\bar{\mu}]} \frac{\mu}{\sigma} Z^{\gamma_t}(u)\\
		&\qquad+\max(\text{sin}(X^{\gamma_t}(u)+\int_0^u X^{\gamma_t}(r)\,dr),0)(\bar{\mu}+X^{\gamma_t}(u))\\
		&\qquad+\min(\text{sin}(X^{\gamma_t}(u)+\int_0^u X^{\gamma_t}(r)\,dr),0)(\underline{\mu}+X^{\gamma_t}(u))\\
		&\qquad+\frac{1}{2}\sigma^2 \text{cos}\big(X^{\gamma_t}(u)+\int_0^u X^{\gamma_t}(r)\,dr \big) \bigg) \,du\\
		&\;-\int_{s}^{T} Z^{\gamma_{t}}(u)\,dW(u) \,,\; t\le s\le T.
	\end{aligned}
\end{equation}
\endgroup
Although, this path-dependent FBSDE does not satisfy the theoretical assumptions established in this paper, nevertheless its numerical implementation still yields satisfactory results.

We compare the numerical solution obtained by our scheme with the exact solution.   Simulations are carried out using the parameters $\underline{\mu}=0.2$, $\bar{\mu}=0.3$, and $\sigma=0.3$, together with \eqref{parameter}.    
We train the model for  $10{,}000$ epochs, with the learning rate initialized at $0.01$ and reduced by a factor of $0.98$ every $50$ epochs during training. This PPDE admits the closed-form solution $u(t,\gamma)=\cos(\gamma(t)+\int_0^t \gamma(s)\,ds)$.
From this closed-form solution, we obtain $Y^{\gamma_t}(0)=\cos(1)\approx 0.5403$ and $Z^{\gamma_t}(0)=-\sigma\sin(1)\approx -0.2524$. Table \ref{table:nonlinearppde_transposed} shows $\tilde{Y}^m(t)$ and $\tilde{Z}^m(t)$ for $1 \le m \le 10$.
For $m\ge 5$, the relative errors are below $1.5\%$ for $\tilde{Y}^m(t)$ and below $3.4\%$ for $\tilde{Z}^m(t)$, except for $m=8$. In this case, the loss values at the final epoch are almost between $5\times 10^{-6}$ and $10^{-5}$ for $\text{Loss}_Y$, and between $5\times 10^{-3}$ and $10^{-2}$ for $\text{Loss}_Z$.

\begin{table}[htbp]
\caption{Approximated values $\tilde{Y}^m(t),\, \tilde{Z}^m(t)$ for each Picard iteration for \eqref{eqn:ex2}}
\centering
\begin{tabular}{|c|c|c|}
\hline
$m$ & $\tilde{Y}^m(t)$ & $\tilde{Z}^m(t)$ \\ \hline
1 & -0.3750 & -0.4884 \\ \hline
2 & 0.4490 & -0.1431 \\ \hline
3 & 0.5743 & -0.2297 \\ \hline
4 & 0.5478 & -0.2462 \\ \hline
5 & 0.5394 & -0.2439 \\ \hline
6 & 0.5481 & -0.2461 \\ \hline
7 & 0.5447 & -0.2575 \\ \hline
8 & 0.5418 & -0.2365 \\ \hline
9 & 0.5467 & -0.2478 \\ \hline
10 & 0.5448 & -0.2542 \\ \hline
\end{tabular}

\label{table:nonlinearppde_transposed}
\end{table}

	\subsection{Lookback options}
	Option pricing  is one of the important topics in mathematical finance.
	A floating-strike lookback call option  is  a financial derivative that allows the holder to buy the underlying asset at its lowest price during a specified time period.
	In this section, we study a numerical method for estimating the
	floating-strike lookback call option prices
	under the Black--Scholes model, 
	following the approach of Saporito and Zhang~\cite{deepgalerkin}.
	The option price is given as a solution to the system of FBSDE
	\begin{equation}
\fitdisplay{
\begin{aligned}
			X^{\gamma_{t}}(s)&=\gamma_{t}(t)+\int_{t}^{s} r X^{\gamma_{t}}(u)\,du+ \int_{t}^{s} \sigma X^{\gamma_{t}}(u) \,dW(u) \,, &&t\le s \le T \,,\\
			X^{\gamma_{t}}(s)&=\gamma_{t}(s) \,, &&0 \le s \le t \,,\\    
			Y^{\gamma_{t}}(s)&=\big[ X^{\gamma_{t}}(T)-\inf_{0\le u \le T} X^{\gamma_{t}}(u)  \big] 
			-\int_{s}^{T} rY^{\gamma_{t}}(u) \,du-\int_{s}^{T} Z^{\gamma_{t}}(u)\,dW(u) \,, &&t\le s\le T\,.
		\end{aligned}
}
\end{equation}
	This system of FBSDE has the closed-form solution 
	$$
\fitdisplay{
Y^{\gamma_{t}}(t)=\gamma(t) \Phi(a_1) - m(t) e^{-r(T-t)}\Phi(a_2)-\gamma(t) \frac{\sigma^2}{2r}\bigg(\Phi(-a_1)-e^{-r(T-t)} \bigg(\frac{m(t)}{\gamma(t)} \bigg)^{\frac{2r}{\sigma^2}}\Phi(-a_3)\bigg)
}
$$
	where $m(t)=\inf_{0\le u \le t} \gamma(u)$ and $$a_1=\dfrac{\text{log}\big(\frac{\gamma(t)}{m_{t}} \big)+(r+\frac{\sigma^2}{2})(T-t)}{\sigma \sqrt{T-t}}\,,\; a_2=a_1-\sigma\sqrt{T-t}\,,\; a_3=a_1-\frac{2r}{\sigma}\sqrt{T-t}.$$

	We compare the obtained numerical solution 
	with the exact solution.
	Simulations are conducted for the aforementioned path-dependent FBSDE using 
	the parameters $\sigma=0.2$ and \eqref{parameter}. $Y^{\gamma_t}(0)$ using the closed-form solution is about $0.162986$.

    \begin{table}[htbp]
\caption{Approximated values $\tilde{Y}^m(t),\, \tilde{Z}^m(t)$ for each Picard iteration}
\centering
\begin{tabular}{|c|c|c|}
    \hline
     $m$& $\tilde{Y}^m(t)$ & $\tilde{Z}^m(t)$ \\ \hline
     1&0.157250  &0.049531 \\ \hline
     2& 0.152850 &0.040240\\ \hline
     3& 0.148156  &0.047449  \\ \hline
     4&0.149145  &0.037716 \\ \hline
     5&  0.149297 &0.047657  \\ \hline
     6& 0.154162  &0.046807  \\ \hline
     7& 0.146335  &  0.041746 \\ \hline
     8& 0.153920 & 0.042105 \\ \hline
     9&0.150660 &  0.042581\\ \hline
     10&0.151872 &0.048110  \\ \hline
    \end{tabular}
\label{lookbackvalue}
\end{table}
    To examine the discretization error, we perform a Monte Carlo simulation for equation $\frac{1}{L}\sum_{l=1}^L e^{-r(T-t)}\big[ X^{\gamma_{t},\pi}(T)-\inf_{0\le u \le n} X^{\gamma_{t},\pi}(t_u)  \big]  $ with $n=70$ and $L=200{,}000$. The resulting estimate is $0.151638$. In this case, for each Picard iteration, the loss values at the final epoch are almost between $10^{-5}$ and $5\times 10^{-5}$ for $\text{Loss}_Y$.

\section{Black--Scholes model with path-dependent dividends}
\label{sec:deri}
 
		In this section, a numerical method for pricing derivatives in models involving dividend-paying stocks is discussed.
		This problem can be formulated as a FBSDE with a path-dependent generator, which is
		an extended version of the model introduced in Bender and Denk~\cite{forwardbsde}.
		For simplicity, we let $d=\ell=1$ in this section.   
		Consider  a stock price process following the usual Black--Scholes model dynamics:
		\begin{equation} 
			\begin{aligned}
				X^{\gamma_{t}}(s)&=\gamma_{t}(t)+\int_{t}^{s} \mu X^{\gamma_{t}}(u)\,du+ \int_{t}^{s} \sigma X^{\gamma_{t}}(u) \,dW(u) \,,  && t\le s \le T\,, \\
				X^{\gamma_{t}}(s)&=\gamma_{t}(s) \,, && 0 \le s \le t 
			\end{aligned}
		\end{equation}
		for positive $\gamma\in D([0,T],\mathbb{R})$, 
		with the dividend rate given by
		$\delta(u,X_u^{\gamma_{t}})=(\frac{D}{c} \int_{u-c}^{u} X^{\gamma_{t}}(s)\,ds)\wedge q $ 
		for a non-anticipative functional $\delta:\Lambda_T\to \mathbb{R}$ and three positive constants $c$, $D$, and $q.$
		We set $X^{\gamma_{t}}(s)=\gamma(0)>0$ for $-c\le s\le 0.$ 
		This specification implies that the dividend rate is determined by the average of the stock prices over the past $c$ time periods, multiplied by the constant $D,$ and capped by the upper bound $q.$
		The short rate is assumed to be a constant $r\ge0.$
		In this stock model, we price an option with a straddle payoff for strike price $K$ and maturity $T$.
		The pricing BSDE is given by 
		\begin{equation}\label{dividenequ}
\fitdisplay{
\begin{aligned}
				Y^{\gamma_{t}}(s)&= | X^{\gamma_{t}}(T)-K  | \\
				&-\int_{s}^{T}\big[ rY^{\gamma_{t}}(u)+\frac{\mu+\delta(u,X_{u}^{\gamma_{t}})-r}{\sigma}Z^{\gamma_{t}}(u) \big] \,du-\int_{s}^{T} Z^{\gamma_{t}}(u)\,dW(u)\,, \quad  t\le s\le T\,.
			\end{aligned}
}
\end{equation}
		and the desired option price is given by the solution $ Y^{\gamma_{t}}(t)$.

		We work with a truncated version of this FBSDE.
		As the driver $f:=-ry-\frac{\mu+\delta(s,\omega_s)-r}{\sigma}z$ of this BSDE is not Lipschitz continuous, the results of this paper are not applicable.
		To address this issue, we truncate the driver as follows.
		Define a bound function $h^{\eta_1,\eta_2}:\mathbb{R}\to\mathbb{R}$ as  $h^{\eta_1,\eta_2}(x)=(x\vee \eta_1)\wedge \eta_2$ for given constants $\eta_1<\eta_2$, and consider the truncated BSDE  
		\begin{equation} \label{trun_dividenequ}
			\begin{aligned}
				Y^{\gamma_t,\eta_1,\eta_2}(s)&= | X^{\gamma_{t}}(T)-K  |\\
				&-\int_{s}^{T}\big[ rY^{\gamma_t,\eta_1,\eta_2}(u)+\frac{\mu+\delta(u,X_{u}^{\gamma_{t}})-r}{\sigma}h^{\eta_1,\eta_2}(Z^{\gamma_t,\eta_1,\eta_2}(u)) \big] \,du \\
				&-\int_{s}^{T} Z^{\gamma_t,\eta_1,\eta_2}(u)\,dW(u)\,, \quad  t\le s\le T\,.
			\end{aligned}
		\end{equation}
		This truncated BSDE satisfies   the Lipschitz condition stated in Assumption \ref{asmbsde}. Therefore, a unique solution $(Y^{\gamma_t,\eta_1,\eta_2},Z^{\gamma_t,\eta_1,\eta_2})$ exists in $\mathbb{S}^2(t,T;\mathbb{R})\times \mathbb{H}^2(t,T;\mathbb{R})$, and the numerical scheme developed in this paper can be applied to solve it.
		The following proposition means that a solution  $(Y^{\gamma_t,\eta_1,\eta_2},Z^{\gamma_t,\eta_1,\eta_2})$ to this truncated BSDE approaches to the solution $(Y^{\gamma_t},Z^{\gamma_t})$ of the original BSDE as
		$\eta_1\to-\infty$ and $\eta_2\to\infty.$
		Hence, we regard 
		$(Y^{\gamma_t,\eta_1,\eta_2},Z^{\gamma_t,\eta_1,\eta_2})$ as a numerical approximation of the true solution $(Y^{\gamma_t},Z^{\gamma_t})$.
		The following proposition is a direct consequence of Theorems 4.2.1, 4.3.1, and 4.4.3 in Zhang~\cite{BSDE}.
		\begin{prop} The following statements hold.
			\begin{enumerate}[(i)]
				\item There exists a unique solution $(Y^{\gamma_t},Z^{\gamma_t})$ to \eqref{dividenequ} in $\mathbb{S}^2(t,T;\mathbb{R})\times \mathbb{H}^2(t,T;\mathbb{R})$.
				\item Let $(Y^{\gamma_t,\eta_1,\eta_2},Z^{\gamma_t,\eta_1,\eta_2})$ be the unique solution to \eqref{trun_dividenequ}   in   $\mathbb{S}^2(t,T;\mathbb{R})\times \mathbb{H}^2(t,T;\mathbb{R})$. Then,   
				\begin{equation}
\fitdisplay{
\lim_{\eta_1\to -\infty,\,\eta_2\to\infty}\mathbb{E}\bigg[ \sup_{t\le s \le T} |Y^{\gamma_t}(s)-Y^{\gamma_t,\eta_1,\eta_2}(s)|^2+\int_t^T |Z^{\gamma_t}(s)-Z^{\gamma_t,\eta_1,\eta_2}(s)|^2\,ds \bigg]=0\,.
}
\end{equation}  
			\end{enumerate}
		\end{prop}

		We compute the values $Y^{\gamma_{t},\eta_1,\eta_2,\pi,m}(t)$ and $Z^{\gamma_{t},\eta_1,\eta_2,\pi,m}(t)$ using the following parameters:
		\begin{equation}
			\begin{aligned}
				&\mu=0.1,\,\sigma=0.4,\, r=0.03,\,t=0.1,\,T=1,\,n_0=100,\, n=40,\, c=0.01,\,q=0.015\, \\ &\eta_1=-1000,\,\eta_2=1000,\,\gamma(s)=400(1-s)^2,\, K=324,\,D=\frac{q}{2\gamma(t)},\,m=10.
			\end{aligned}
		\end{equation}
		We take $n_{\delta}=5$   partitions of the interval $[u-c,u]$ for  estimating the dividends $\delta(u,X_u^{\gamma_{t}})\approx \frac{D}{c} \sum_{i=1}^{n_{\delta}} X^{\gamma_{t}}(u-c+(i-1)\frac{c}{n_{\delta}})\frac{c}{n_{\delta}}$. Our discretization scheme yields the estimates
$Y^{\gamma_t,\eta_1,\eta_2,\pi,m}(t)=94.3490$ and
$Z^{\gamma_t,\eta_1,\eta_2,\pi,m}(t)=44.7177$.


\section{Duration-dependent CEV model}
\label{sec:cev_application}

This section applies the Picard iteration framework to derivative pricing under a duration-dependent CEV model, focusing on the effects of the duration parameter on option values and the martingale integrand. 
In this model, the CEV volatility coefficient depends on a weighted average of past stock prices over a prescribed duration, rather than solely on the current stock price. The resulting pricing problem incorporates path dependence in both the stock price dynamics and the derivative payoff. Throughout this section, we restrict attention to the one-dimensional setting $d=\ell=1$ and interpret $\mathbb{P}$ as the risk-neutral pricing measure.

\subsection{Duration-dependent CEV model and pricing FBSDE}

Let $X^{\gamma_t}=(X^{\gamma_t}(s))_{s\geq 0}$ denote the stock price process, and let $r\geq 0$ be a constant short rate. 
The classical CEV model is
\begin{equation}
	\label{eq:cev_classical_application}
	dX^{\gamma_t}(s)=rX^{\gamma_t}(s)\,ds+\nu\left(\frac{X^{\gamma_t}(s)}{X^{\gamma_t}(0)}\right)^{\beta}X^{\gamma_t}(s)\,dW(s),
	\qquad -1<\beta<0,
\end{equation}
where $X^{\gamma_t}(0)>0$ and $\nu>0$. This model was introduced by Cox~\cite{coxnote} and developed further by Cox and Ross~\cite{coxvaluation}. We use the usual absorbing convention at zero; see Delbaen and Shirakawa~\cite{attainableref}. Setting $\sigma=\nu X^{\gamma_t}(0)^{-\beta}$, equation~\eqref{eq:cev_classical_application} becomes
\begin{equation}
	\label{eq:cev_classical_application_sigma}
	dX^{\gamma_t}(s)=rX^{\gamma_t}(s)\,ds+\sigma X^{\gamma_t}(s)^{\beta}X^{\gamma_t}(s)\,dW(s).
\end{equation}

To introduce duration dependence into the CEV specification, consider the
non-anticipative weighted-average functional
\begin{equation}
	\label{eq:cev_history_factor}
	\eta^{\delta,k}(s,\omega)
	:=
	\frac{\displaystyle\int_{(s-\delta)\vee0}^{s}
	e^{-k(s-u)}\omega(u)\,du}
	{\displaystyle\int_{(s-\delta)\vee0}^{s}
	e^{-k(s-u)}\,du},
	\qquad s>0,\quad \delta>0,
\end{equation}
with the convention $\eta^{\delta,k}(s,\omega)=\omega(s)$ when $s=0$
or $\delta=0$. The natural duration-dependent extension of the classical
CEV model is
\begin{equation}
	\label{eq:cev_path_application}
	dX^{\gamma_t}(s)
	=
	rX^{\gamma_t}(s)\,ds
	+
	\sigma
	\bigl(\eta^{\delta,k}(s,X^{\gamma_t}_s)\bigr)^{\beta}
	X^{\gamma_t}(s)\,dW(s),
	\qquad -1<\beta<0.
\end{equation}
The parameter $\delta$ specifies the length of the trailing interval used
to construct the weighted average of past stock-price levels, whereas
$k$ determines the rate at which earlier observations within that
interval are discounted. When $\delta=0$,
equation~\eqref{eq:cev_path_application} reduces to the classical CEV
model. The instantaneous CEV feedback based solely on the current stock
price is therefore replaced by duration-dependent feedback through a
weighted average of past prices. This specification can distinguish
temporary price movements from price regimes that persist over a longer
duration.

Although equation~\eqref{eq:cev_path_application} is the model that
motivates our analysis, its volatility coefficient does not satisfy the
global Lipschitz condition in Assumption~\ref{asmsde}. In particular,
because $\beta<0$, the negative-power term becomes singular when the
historical average approaches zero. We therefore consider a truncated
and regularized version of this specification. For $R>0$, define
\begin{equation}
	h_R(x):=(x^{+})\wedge R.
\end{equation}
For $\epsilon_1>0$, $\epsilon_2>0,$ and $-1<\beta<0,$ we let
\begin{equation}
	\label{eq:cev_regularized_history_factor}
	\begin{aligned}
	&\eta_{R,\epsilon_1}^{\delta,k}(s,\omega)
	:=
	\frac{
		\displaystyle
		\int_{(s-\delta)\vee0}^{s}
		e^{-k(s-u)}
		h_R\bigl(\omega(u)\bigr)\,du
		+
		\epsilon_1 h_R\bigl(\omega(s)\bigr)
	}{
		\displaystyle
		\int_{(s-\delta)\vee0}^{s}
		e^{-k(s-u)}\,du
		+
		\epsilon_1
	},
	\\
	&\qquad(s,\omega)\in[0,T]\times D([0,T];\mathbb{R}).
	\end{aligned}
\end{equation}
We then define the truncated duration-dependent CEV volatility by
\begin{equation}
	\label{eq:cev_truncated_volatility}
	\sigma_{\rm CEV}^{R,\epsilon_1,\epsilon_2}(s,\omega)
	:=
	\sigma h_R\bigl(\omega(s)\bigr)
	\left(
		\eta_{R,\epsilon_1}^{\delta,k}(s,\omega)
		\vee\epsilon_2
	\right)^{\beta}.
\end{equation}
Accordingly, the model studied in this section is
\begin{equation}
	\label{eq:cev_truncated_path_application}
	dX^{\gamma_t}(s)
	=
	rX^{\gamma_t}(s)\,ds
	+
	\sigma_{\rm CEV}^{R,\epsilon_1,\epsilon_2}
	\bigl(s,X^{\gamma_t}_s\bigr)\,dW(s).
\end{equation}
We impose zero  as an absorbing boundary.
Here, $h_R$ bounds the current stock-price level and its historical
average, $\epsilon_1$ keeps the denominator in
\eqref{eq:cev_regularized_history_factor} uniformly away from zero, and
$\epsilon_2$ removes the singularity caused by the negative exponent.
Since $h_R$ is bounded and globally Lipschitz and the map
$x\mapsto(x\vee\epsilon_2)^\beta$ is globally Lipschitz on
$[0,\infty)$, the coefficient
$\sigma_{\rm CEV}^{R,\epsilon_1,\epsilon_2}$ is a bounded, non-anticipative,
and globally Lipschitz path functional. Hence, the modified model
\eqref{eq:cev_truncated_path_application} satisfies
Assumption~\ref{asmsde} and falls within the theoretical framework of
this paper.
For the numerical experiments, we choose $R$ sufficiently large and
$\epsilon_1$ and $\epsilon_2$ sufficiently small so that the effects of
truncation and regularization are minimized while the modified volatility
coefficient continues to satisfy Assumption~\ref{asmsde}.

For a non-anticipative payoff functional $g:D([0,T],\mathbb{R})\to\mathbb{R}$, the pricing FBSDE is
\begin{equation}
	\label{eq:cev_pricing_bsde_application}
	Y^{\gamma_t}(s)
	=g(X^{\gamma_t}_T)
	-\int_s^T rY^{\gamma_t}(u)\,du
	-\int_s^T Z^{\gamma_t}(u)\,dW(u),
	\qquad 0\leq s\leq T.
\end{equation}
This is the BSDE~\eqref{eqn:pbsde} with the linear non-anticipative driver
$f(s,\omega,y,z)=-ry$. The process $Y^{\gamma_t}$ is the option value process, while $Z^{\gamma_t}$ is the martingale integrand. On $\{X^{\gamma_t}(s)>0\}$, the number of shares in the replicating stock position is
\begin{equation}
	\label{eq:cev_stock_position_application}
	\phi(s)
	=\frac{Z^{\gamma_t}(s)}{\sigma_{\rm CEV}^{R,\epsilon_1,\epsilon_2}
	\bigl(s,X^{\gamma_t}_s\bigr)},
\end{equation}
and we set $\phi(s)=0$ after absorption at zero.

We consider the European call payoff
\[
g_{\rm call}(\omega)=(\omega(T)-K)^+
\]
and the double-barrier digital payoff
\begin{equation}
	\label{eq:cev_barrier_payoff_application}
	g_{\rm bar}(\omega)
	=\mathbf{1}_{\{\underline{\omega}_T>X(0)-L\}}
	\mathbf{1}_{\{\overline{\omega}_T<X(0)+U\}},
	\qquad
	\underline{\omega}_T:=\min_{0\leq s\leq T}\omega(s),\quad
	\overline{\omega}_T:=\max_{0\leq s\leq T}\omega(s).
\end{equation}
The European call payoff depends only on the terminal value, whereas the double-barrier digital payoff depends on the entire path through its running minimum and maximum. Since the latter payoff is discontinuous, it does not satisfy Assumption~\ref{asmbsde}. Accordingly, the numerical results for the double-barrier digital option should be viewed as illustrating the practical applicability of the Picard construction beyond the theoretical setting of Theorem~\ref{thm:main}, rather than as a direct consequence of that theorem.

\subsection{Numerical setup and European call options}

We implement the Picard iteration in Section~\ref{sec:Picard} with the parameter sets
 $$ T=1,\, t=0,\, \gamma(0)=1,\, k=0,\,n=50,\, \tilde{n}=5,\, m=8,\,R=1000,\,\epsilon_1=\epsilon_2=10^{-8}\,.$$ 

For the Euler approximation of the forward process, each of the $n$ time intervals is further divided into $\tilde n$ equal subintervals, resulting in a total of $n\tilde n$ time steps. This refinement is introduced to more accurately capture the path-dependent feature of the forward SDE. We emphasize that this additional refinement is used only for the simulation of the forward process, while the BSDE is still discretized over the original $n$ time intervals. Moreover, if the Euler approximation of $X^{\gamma_t}$ becomes negative at any time step, it is set equal to zero and remains zero thereafter. Because the driver is linear in $Y^{\gamma_t}$ and independent of $Z^{\gamma_t}$, the conditional expectations in~\eqref{scheme} are evaluated directly by Monte Carlo sampling over concatenated Brownian paths.

The diffusion parameters are estimated from Apple and Amazon stock prices over January 1, 2020--December 31, 2025, following the procedure in Anel and Fan~\cite[Section~4.1]{newtypecev}. Since the model is specified under the risk-neutral measure, these historical estimates are used only as illustrative volatility parameters. We normalize $X^{\gamma_t}(0)=1$, so that $\sigma=\nu$, and obtain the values in Table~\ref{tab:cev_parameters_application}.

\begin{table}[htbp]
	\caption{Illustrative CEV parameter values for Apple and Amazon}
\centering
		\begin{tabular}{|c|c|c|}
				\hline
				& $\beta$ & $\nu$ \\ \hline
				Apple & $-0.7188$ & $0.5194$ \\ \hline
				Amazon & $-0.6685$ & $0.4784$ \\ \hline
	\end{tabular}
	\label{tab:cev_parameters_application}
\end{table}

To isolate the effect of duration dependence, we set $k=0$ and vary the duration parameter
$\delta\in\{0,0.05,0.1,0.2,\ldots,0.9\}$. The functional~\eqref{eq:cev_history_factor} then becomes a simple average over the trailing interval of duration $\delta$. We take $T=1$, consider $r\in\{0.02,0.04\}$, and use the at-the-money strike $K=X^{\gamma_t}(0)=1$. The corresponding BSDE is
\begin{equation}
	\label{eq:cev_call_bsde_application}
	Y^{\gamma_t}(s)
	=(X^{\gamma_t}(T)-K)^+
	-\int_s^T rY^{\gamma_t}(u)\,du
	-\int_s^T Z^{\gamma_t}(u)\,dW(u).
\end{equation}

Tables~\ref{tab:cev_call_apple_application}
and~\ref{tab:cev_call_amazon_application} report the Picard estimates of
the option value and the initial martingale integrand.
For both parameter sets, the option value increases as the interest rate
rises from $0.02$ to $0.04$, whereas its dependence on the duration
parameter $\delta$ is relatively mild. This is consistent with the fact
that the European call payoff depends only on $X^{\gamma_t}(T)$:
duration dependence in the CEV coefficient affects the option value
only indirectly through the terminal distribution. In contrast,
$\widehat Z_0^{\gamma_t,m}$ exhibits more pronounced variation with
$\delta$. This indicates that the martingale integrand, and hence the
stock position in~\eqref{eq:cev_stock_position_application}, is more
sensitive to the duration parameter than the call value itself.

\begin{table}[htbp]
	\caption{European call estimates under the Apple parameter set}
	\centering
	\begin{tabular}{|c|cc|cc|}
		\hline
		& \multicolumn{2}{c|}{$r=0.02$}
		& \multicolumn{2}{c|}{$r=0.04$}\\ \cmidrule{2-5}
		$\delta$
		& $\widehat Y_0^{\gamma_t,m}$ & $\widehat Z_0^{\gamma_t,m}$
		& $\widehat Y_0^{\gamma_t,m}$ & $\widehat Z_0^{\gamma_t,m}$\\ \hline
		$0$    & 0.2141 & 0.2788 & 0.2226 & 0.2927\\ \hline
		$0.05$ & 0.2145 & 0.2782 & 0.2226 & 0.2876\\ \hline
		$0.1$  & 0.2150 & 0.2812 & 0.2231 & 0.2906\\ \hline
		$0.2$  & 0.2143 & 0.2706 & 0.2224 & 0.2796\\ \hline
		$0.3$  & 0.2148 & 0.2650 & 0.2229 & 0.2741\\ \hline
		$0.4$  & 0.2146 & 0.2699 & 0.2227 & 0.2795\\ \hline
		$0.5$  & 0.2145 & 0.2718 & 0.2226 & 0.2815\\ \hline
		$0.6$  & 0.2134 & 0.2629 & 0.2215 & 0.2722\\ \hline
		$0.7$  & 0.2138 & 0.2657 & 0.2219 & 0.2753\\ \hline
		$0.8$  & 0.2138 & 0.2687 & 0.2219 & 0.2784\\ \hline
		$0.9$  & 0.2135 & 0.2646 & 0.2216 & 0.2742\\ \hline
	\end{tabular}
	\label{tab:cev_call_apple_application}
\end{table}

\begin{table}[htbp]
	\caption{European call estimates under the Amazon parameter set}
	\centering
	\begin{tabular}{|c|cc|cc|}
		\hline
		& \multicolumn{2}{c|}{$r=0.02$}
		& \multicolumn{2}{c|}{$r=0.04$}\\ \cmidrule{2-5}
		$\delta$
		& $\widehat Y_0^{\gamma_t,m}$ & $\widehat Z_0^{\gamma_t,m}$
		& $\widehat Y_0^{\gamma_t,m}$ & $\widehat Z_0^{\gamma_t,m}$\\ \hline
		$0$    & 0.1981 & 0.2584 & 0.2067 & 0.2719\\ \hline
		$0.05$ & 0.1983 & 0.2580 & 0.2066 & 0.2673\\ \hline
		$0.1$  & 0.1987 & 0.2611 & 0.2070 & 0.2704\\ \hline
		$0.2$  & 0.1980 & 0.2515 & 0.2063 & 0.2604\\ \hline
		$0.3$  & 0.1985 & 0.2469 & 0.2068 & 0.2626\\ \hline
		$0.4$  & 0.1983 & 0.2515 & 0.2066 & 0.2609\\ \hline
		$0.5$  & 0.1983 & 0.2536 & 0.2066 & 0.2632\\ \hline
		$0.6$  & 0.1979 & 0.2490 & 0.2062 & 0.2584\\ \hline
		$0.7$  & 0.1978 & 0.2481 & 0.2061 & 0.2575\\ \hline
		$0.8$  & 0.1978 & 0.2511 & 0.2062 & 0.2606\\ \hline
		$0.9$  & 0.1976 & 0.2471 & 0.2059 & 0.2565\\ \hline
	\end{tabular}
	\label{tab:cev_call_amazon_application}
\end{table}

\subsection{Double-barrier digital options}

We next consider the payoff in~\eqref{eq:cev_barrier_payoff_application}. The associated pricing BSDE is
\begin{equation}
	\label{eq:cev_barrier_bsde_application}
	Y^{\gamma_t}(s)
	=\mathbf{1}_{\{\underline X^{\gamma_t}_T>X^{\gamma_t}(0)-L\}}
	\mathbf{1}_{\{\overline X^{\gamma_t}_T<X^{\gamma_t}(0)+U\}}
	-\int_s^T rY^{\gamma_t}(u)\,du
	-\int_s^T Z^{\gamma_t}(u)\,dW(u),
\end{equation}
where
$\underline X^{\gamma_t}_T:=\min_{0\leq u\leq T}X(u)$ and
$\overline X^{\gamma_t}_T:=\max_{0\leq u\leq T}X(u)$. We set $L=U=0.4$ and retain all other parameter and simulation choices from the preceding subsection.

Tables~\ref{tab:cev_barrier_apple_application}
and~\ref{tab:cev_barrier_amazon_application} report the Picard estimates
of the double-barrier digital option value and the initial martingale
integrand. The barrier-option values are substantially more sensitive
to the duration parameter $\delta$ than the European call values.
Comparing $\delta=0$ with $\delta=0.9$ gives a relative value difference
of roughly $5$--$6\%$. This stronger response is natural because
changing the duration used in the CEV coefficient affects not only the
terminal distribution of $X^{\gamma_t}(T)$ but also the probability
that the entire trajectory remains between the two barriers. The
initial martingale integrand also varies substantially with $\delta$.

\begin{table}[htbp]
	\caption{Double-barrier digital estimates under the Apple parameter set}
	\centering
	\begin{tabular}{|c|cc|cc|}
		\hline
		& \multicolumn{2}{c|}{$r=0.02$}
		& \multicolumn{2}{c|}{$r=0.04$}\\ \cmidrule{2-5}
		$\delta$
		& $\widehat Y_0^{\gamma_t,m}$ & $\widehat Z_0^{\gamma_t,m}$
		& $\widehat Y_0^{\gamma_t,m}$ & $\widehat Z_0^{\gamma_t,m}$\\ \hline
		$0$    & 0.1911 & -0.0189 & 0.1867 & -0.0247\\ \hline
		$0.05$ & 0.1930 & -0.0210 & 0.1875 & -0.0206\\ \hline
		$0.1$  & 0.1942 & -0.0150 & 0.1888 & -0.0167\\ \hline
		$0.2$  & 0.1944 & -0.0039 & 0.1892 & -0.0174\\ \hline
		$0.3$  & 0.1971 &  0.0039 & 0.1920 & -0.0032\\ \hline
		$0.4$  & 0.1975 &  0.0118 & 0.1937 &  0.0065\\ \hline
		$0.5$  & 0.1997 &  0.0181 & 0.1957 &  0.0078\\ \hline
		$0.6$  & 0.1999 &  0.0107 & 0.1959 &  0.0056\\ \hline
		$0.7$  & 0.2019 &  0.0278 & 0.1979 &  0.0205\\ \hline
		$0.8$  & 0.2028 &  0.0305 & 0.1985 &  0.0187\\ \hline
		$0.9$  & 0.2024 &  0.0202 & 0.1981 &  0.0139\\ \hline
	\end{tabular}
	\label{tab:cev_barrier_apple_application}
\end{table}

\begin{table}[htbp]
	\caption{Double-barrier digital estimates under the Amazon parameter set}
	\centering
	\begin{tabular}{|c|cc|cc|}
		\hline
		& \multicolumn{2}{c|}{$r=0.02$}
		& \multicolumn{2}{c|}{$r=0.04$}\\ \cmidrule{2-5}
		$\delta$
		& $\widehat Y_0^{\gamma_t,m}$ & $\widehat Z_0^{\gamma_t,m}$
		& $\widehat Y_0^{\gamma_t,m}$ & $\widehat Z_0^{\gamma_t,m}$\\ \hline
		$0$    & 0.2522 & -0.0318 & 0.2453 & -0.0389\\ \hline
		$0.05$ & 0.2539 & -0.0223 & 0.2468 & -0.0312\\ \hline
		$0.1$  & 0.2548 & -0.0247 & 0.2485 & -0.0317\\ \hline
		$0.2$  & 0.2557 & -0.0276 & 0.2501 & -0.0339\\ \hline
		$0.3$  & 0.2582 & -0.0112 & 0.2526 & -0.0189\\ \hline
		$0.4$  & 0.2602 & -0.0041 & 0.2540 & -0.0100\\ \hline
		$0.5$  & 0.2620 & -0.0010 & 0.2555 & -0.0075\\ \hline
		$0.6$  & 0.2625 & -0.0070 & 0.2565 & -0.0121\\ \hline
		$0.7$  & 0.2644 &  0.0107 & 0.2583 &  0.0058\\ \hline
		$0.8$  & 0.2652 &  0.0114 & 0.2590 &  0.0083\\ \hline
		$0.9$  & 0.2647 &  0.0032 & 0.2588 & -0.0019\\ \hline
	\end{tabular}
	\label{tab:cev_barrier_amazon_application}
\end{table}

A notable feature is that $\widehat Z_0^{\gamma_t,m}$ changes sign as the duration parameter varies. This is consistent with the double-barrier structure. A positive Brownian shock can move the stock away from the lower barrier, increasing the survival probability, but it can simultaneously move the stock closer to the upper barrier, decreasing that probability. The dominant effect depends on the future path distribution and on how the duration-dependent CEV coefficient responds to the stock-price average over the selected trailing interval. The sign changes therefore show that the direction of the local hedging response is sensitive to the duration specification.

\section{Conclusion} \label{sec:con}	 

We construct a Picard iteration scheme, prove its convergence to the exact solution, and derive an explicit convergence rate.  
We provide a novel estimator for the martingale integrand in the FBSDE, specifically designed to handle path-dependence more reliably than existing methods.
We derive a concentration inequality that quantifies the statistical error of this estimator in a Monte Carlo framework.
Based on these results, we investigate a supervised learning method with neural networks for solving path-dependent PDEs.

\section*{Funding}
Hyungbin Park was supported by National Research Foundation of Korea grants funded by the Ministry of Science and ICT (2022R1A5A6000840 and RS-2026-25488333).

\section*{Declaration of AI assistance}
The authors used ChatGPT (OpenAI) to improve the language and readability of the article and to assist in correcting errors.
The ideas and results presented in this article are the authors' original intellectual contributions.
The authors take full responsibility for the content of the article, including its mathematical statements, proofs, and references.

\section*{Disclosure statement}
No potential conflict of interest was reported by the authors.

\section*{Data availability} 
 
Data will be made available upon request.

\bibliographystyle{plainnat}
\bibliography{references}

\end{document}